\documentclass{amsart}
\usepackage{hyperref}
\usepackage{xcolor}
\usepackage{amsfonts}
\usepackage[T1]{fontenc}
\usepackage{enumitem}
\usepackage{amsmath,amssymb,amsfonts,amsthm}
\usepackage{bm}
\usepackage{latexsym}
\usepackage{amscd}
\usepackage[all,cmtip]{xy}
\usepackage{xfrac}
\usepackage{mathrsfs}
\usepackage{upgreek}
\usepackage{graphicx}
\usepackage{enumitem} 
\usepackage{braket}
\usepackage{verbatim}
\usepackage{stmaryrd}
\usepackage{commath}
\usepackage{cancel}
\usepackage{esint}
\usepackage[mathscr]{eucal}
\usepackage[utf8]{inputenc}

\hypersetup{
    colorlinks,
    linktoc=all,
    linkcolor={blue},
    citecolor={blue},
    urlcolor={blue}
}

\calclayout

\allowdisplaybreaks

\numberwithin{equation}{section}

\theoremstyle{definition}

\newtheorem{example}{Example}[section]

\theoremstyle{plain}
\newtheorem{proposition}{Proposition}[section]
\newtheorem{theorem}[proposition]{Theorem}
\newtheorem{lemma}[proposition]{Lemma}
\newtheorem{corollary}[proposition]{Corollary}

\makeatletter

\def\chaptermark#1{}

\def\chapter{%
  \if@openright\cleardoublepage\else\clearpage\fi
  \thispagestyle{plain}\global\@topnum\z@
  \@afterindenttrue \secdef\@chapter\@schapter}

\def\@chapter[#1]#2{\refstepcounter{chapter}%
  \ifnum\c@secnumdepth<\z@ \let\@secnumber\@empty
  \else \let\@secnumber\thechapter \fi
  \typeout{\chaptername\space\@secnumber}%
  \def\@toclevel{0}%
  \ifx\chaptername\appendixname \@tocwriteb\tocappendix{chapter}{#2}%
  \else \@tocwriteb\tocchapter{chapter}{#2}\fi
  \chaptermark{#1}%
  \addtocontents{lof}{\protect\addvspace{10\p@}}%
  \addtocontents{lot}{\protect\addvspace{10\p@}}%
  \@makechapterhead{#2}\@afterheading}
\def\@schapter#1{\typeout{#1}%
  \let\@secnumber\@empty
  \def\@toclevel{0}%
  \ifx\chaptername\appendixname \@tocwriteb\tocappendix{chapter}{#1}%
  \else \@tocwriteb\tocchapter{chapter}{#1}\fi
  \chaptermark{#1}%
  \addtocontents{lof}{\protect\addvspace{10\p@}}%
  \addtocontents{lot}{\protect\addvspace{10\p@}}%
  \@makeschapterhead{#1}\@afterheading}
\newcommand\chaptername{Chapter}

\def\@makechapterhead#1{\global\topskip 7.5pc\relax
  \begingroup
  \fontsize{\@xivpt}{18}\bfseries\centering
    \ifnum\c@secnumdepth>\m@ne
      \leavevmode \hskip-\leftskip
      \rlap{\vbox to\z@{\vss
          \centerline{\normalsize\mdseries
              \uppercase\@xp{\chaptername}\enspace\thechapter}
          \vskip 3pc}}\hskip\leftskip\fi
     #1\par \endgroup
  \skip@34\p@ \advance\skip@-\normalbaselineskip
  \vskip\skip@ }
\def\@makeschapterhead#1{\global\topskip 7.5pc\relax
  \begingroup
  \fontsize{\@xivpt}{18}\bfseries\centering
  #1\par \endgroup
  \skip@34\p@ \advance\skip@-\normalbaselineskip
  \vskip\skip@ }
\def\appendix{\par
  \c@chapter\z@ \c@section\z@
  \let\chaptername\appendixname
  \def\thechapter{\@Alph\c@chapter}}

\newcounter{chapter}

\newif\if@openright

\def\subsubsection{\@startsection{subsubsection}{3}%
  \z@{.5\linespacing\@plus.7\linespacing}{-.5em}%
  {\normalfont\small\bfseries}}

\makeatletter 
\def\l@subsection{\@tocline{2}{0pt}{1pc}{5pc}{}} \def\l@subsection{\@tocline{2}{0pt}{2pc}{6pc}{}} 
\makeatother

\newcommand{\R}{\mathbb{R}}
\renewcommand{\SS}{\mathbb{S}}

\newcommand{\cL}{\mathcal L}

\newcommand{\dist}{\mathrm{dist}}
\newcommand{\beq}{\begin{equation}}
\newcommand{\eeq}{\end{equation}}
\newcommand{\beqa}{\begin{eqnarray}}
\newcommand{\eeqa}{\end{eqnarray}}

\def\multiset#1#2{\ensuremath{\left(\kern-.3em\left(\genfrac{}{}{0pt}{}{#1}{#2}\right)\kern-.3em\right)}}   

\DeclareMathOperator{\spt}{spt}

\DeclareMathOperator{\row}{row}
\DeclareMathOperator{\diag}{diag}

\renewcommand{\div}{\operatorname{div}}
\renewcommand{\emptyset}{\varnothing}
\renewcommand{\slash}{\setminus}

\newcommand{\wkly}{\overset{\ast}{\rightharpoonup}}

\newcommand{\ep}{\varepsilon}

\newcommand{\Ra}{\Rightarrow}
\newcommand{\La}{\Leftarrow}

\renewcommand{\bar}{\overline}
\newcommand{\mres}{\mathbin{\vrule height 1.6ex depth 0pt width
0.16ex\vrule height 0.13ex depth 0pt width .9ex}}		

\newcommand{\rb}{\partial^*\!}
\newcommand{\eb}{\partial^M\!}

\author[Yeh]{Kuan-Ting Yeh}
\title{The Anisotropic Gaussian Isoperimetric Inequality and Ehrhard Symmetrization}
\address{Department of Mathematics, University of Washington, Seattle, WA}
\email{kty1116@uw.edu}

\begin{document}
\maketitle

\vspace{-.9cm}
\begin{abstract}
In this paper, we prove the isoperimetric inequality for the anisotropic Gaussian measure and characterize the cases of equality. We also find an example that shows Ehrhard symmetrization fails to decrease for the anisotropic Gaussian perimeter and gives a new inequality that includes an error term. This new inequality, in particular, gives us a hint to prove a uniqueness result for the anisotropic Ehrhard symmetrization.
\end{abstract}

\tableofcontents

\section{Introduction}

	The Euclidean isoperimetric problem says that the minimizers of the perimeter among sets with fixed volume are Euclidean balls. More precisely: for any (Lebesgue) measurable set $E\subset \mathbb{R}^{n}$ with $|E|<\infty$,
$$
P(E) \geq n \omega_{n}^{1 / n}|E|^{(n-1) / n} ,
$$
where $\omega_n$ is the volume of the unit ball, $P(E)$ the perimeter of $E$ and $|E|$ its $n$-dimension Lebesgue measure. Moreover, equality holds if and only if $E$ is equivalent to a ball, i.e., $|E\Delta B(x,r)|=0$ for some $x\in \R^n$, $r>0$. One is also interested in a similar problem where volume and perimeter are replaced by Gaussian measure and Gaussian perimeter, which are defined as
$$\gamma(E)=\frac{1}{(2\pi)^{n/2}}\int_Ee^{-|x|^2/2}\ dx,$$
and
$$
P_\gamma(E)=\frac{1}{(2\pi)^{\frac{n-1}{2}}}\int_{\partial^*E}\ e^{-|x|^2/2}\ d \mathcal{H}^{n-1}(x) ,
$$
respectively. The resulting inequality is called the Gaussian isoperimetric inequality. It states that for any measurable set $E\subset \R^n$,
$$P_\gamma(E)\geq e^{-[\phi^{-1}(\gamma(E))]^2/2},$$
where
$$\phi(x)=\frac{1}{\sqrt{2\pi}}\int_{-\infty}^xe^{-t^2/2}\ dt.$$
Moreover, equality holds if and only if $E$ is equivalent to a half-space. The Gaussian isoperimetric problem was first studied by Sudakov, Tsirelson and Borell via Paul Levy's spherical isoperimetric inequality (see, for example, \cite{Borell} \cite{Sudakov-Tsirelson} \cite{LedouxGaussianIso}). In addition, using the Ornstein-Uhlenbeck semigroup techniques, Carlen-Kerce \cite{Carlen-Kerce} characterized half-spaces as the unique minimizers in the Gaussian isoperimetric problem. A geometric approach using Ehrhard symmetrization has been provided by Cianchi-Fusco-Maggi-Pratelli \cite{Cianchi-Fusco-Maggi-Pratelli}. A natural generalization of the Gaussian isoperimetric problem is to study the following question:
$$\inf\left\{P_{\gamma_A}(E):\gamma_A(E)=r\right\},$$
 where $\gamma_A$ is called the {\bf $A$-anisotropic Gaussian measure (mass)} 
$$\gamma_A(E)=\fint_E e^{-\langle Ax,x\rangle/2}\ dx=\frac{\sqrt{\det A}}{(2\pi)^{\frac{n}{2}}}\int_{E} e^{-\langle Ax,x\rangle /2}\ dx,$$
and the perimeter with respect to $\gamma_A$ is called the {\bf $A$-anisotropic Gaussian perimeter}
$$P_{\gamma_A}(E)=\frac{\sqrt{\det A}}{(2\pi)^{\frac{n-1}{2}}}\int_{\eb E} e^{-\langle Ax,x\rangle /2}\ d\mathcal{H}^{n-1}(x).$$
Here $A$ is a positive definite matrix, and $\eb E$ is the $(n-1)$-dimensional essential boundary of $E$ (we will define this in Section \ref{background}). We may assume without loss of generality that our positive definite matrix $A$ is symmetric as
$$\langle Ax,x\rangle =\left\langle \frac12\left(A+A^{\mathsf{T}}\right)x,x\right\rangle$$
and $\frac12\left(A+A^{\mathsf{T}}\right)$ is a symmetric positive definite matrix. We will assume that $A$ is symmetric throughout the rest of this paper. Our first result is the following:

\begin{theorem}\label{AnisotropicGaussainIso}
Let $A$ be a symmetric positive definite matrix and $E$ be a measurable set in $\mathbb{R}^{n}$. Then
\begin{align}\label{theorem1_eq}
P_{\gamma_{A}}\left(E\right) \geq e^{-[\phi^{-1}(\gamma_A(E))]^2/2}\frac{1}{\|(\sqrt{A})^{-1}\|}.
\end{align}
Here $\|\cdot\|$ is the matrix norm induced by the Euclidean norm. Moreover, 
\begin{enumerate}
\item if $n=1$, equality holds if and only if either $\gamma_{A}(E)=0$ or $\gamma_{A}(E)=1$, or $E$ is equivalent to a half-line of the form
$$\left(-\infty,\frac{\phi^{-1}(\gamma_A(E))}{\sqrt{A}}\right)\mbox{\quad or \quad} \left(-\frac{\phi^{-1}(\gamma_A(E))}{\sqrt{A}},+\infty\right).$$
\item if $n\geq 2$, equality holds if and only if either $\gamma_{A}(E)=0$ or $\gamma_{A}(E)=1$, or $E$ is equivalent to a half-space of the form
$$H\left(\omega,\frac{\phi^{-1}(\gamma_A(E))}{d_{\min}}\right)\mbox{\qquad for some unit vector $\omega\in V_{d_{\min}}(\sqrt{A})$},$$
where $d_{\min}$ is the smallest eigenvalue of $\sqrt{A}$ and $V_{d_{\min}}(\sqrt{A})$ is the eigenspace of $\sqrt{A}$ associated with $d_{\min}$.
\end{enumerate}
\end{theorem}

The anisotropic Gaussian isoperimetric inequality (\ref{theorem1_eq}) is a special case of the Bakry-Ledoux isoperimetric inequality for log-concave measures if we consider the log-concave measure  $e^{-\langle Ax,x\rangle/2}dx$ and use the $\ep$-enlargement definition for the perimeter (see \cite{Ledoux_LogSobolev}, Theorem 1.1 and \cite{Ledoux-Bakry}). The main contribution here is to characterize the cases of equality in (\ref{theorem1_eq}) for the anisotropic Gaussian measure.\\

One of the most important properties in the Lebesgue measure is that the (Euclidean) perimeter decreases under Steiner symmetrization (see, for example, \cite{maggi}, Theorem 14.4). A similar result in the Gaussian measure was first mentioned by Ehrhard \cite{Ehrhard1983}, where he introduced another way to symmetrize sets, now called the Ehrhard symmetrization. One of the key properties in his setting is that the Gaussian measure has a product structure, so that the problem can be reduced to the one-dimensional case. Based on this, we also want to generalize this result to the anisotropic Gaussian measure. The main difficulty in our setting is that the measure $\gamma_A$ doesn't have the product structure, i.e., it has cross terms, which means new ideas will need to be developed to address this issue. In fact, the anisotropic Gaussian perimeter does not behave monotonously under Ehrhard symmetrization (see Example \ref{counterexample_1}). Our second result shows that we are still able to find an upper bound for the perimeter of Ehrhard symmetrization set in terms of the original perimeter plus a term involving the deviation of $A$ from the identity in the direction of the symmetrization times a term involving the differences of the anisotropic Gaussian barycenters. To be more precise, we have the following:

\begin{theorem}\label{Ehrhard_Sym_Ineq_II}
Let $n \geq 2$, $A$ be a symmetric positive definite matrix, and let $E$ be a set of finite $A$-anisotropic Gaussian perimeter in $\mathbb{R}^{n}$ and $u\in \SS^{n-1}$. Then, $E_{A,u}^s$ is a set of locally finite perimeter in $\R^n$. Moreover, for every Borel set $B \subseteq \langle u \rangle^{\perp}$ with $|u|=1$ we have
\begin{align*}
P_{\gamma_A}\left(E_{A,u}^{s} ; B \oplus \langle u\rangle\right) &\leq P_{\gamma_A}\left(E ; B \oplus \langle u\rangle\right)\\
&\quad+ \sqrt{2\pi}\|Au-\langle Au,u\rangle u\|\langle b_{\gamma_A}(E^s_{A,u}\cap (B \oplus \langle u\rangle))-b_{\gamma_A}(E\cap (B \oplus \langle u\rangle)),u\rangle.
\end{align*}
Here $B \oplus \langle u\rangle:=\{z+tu:z\in B,t\in \R\}$ and $E_{A,u}^{s}$ is the {\bf Ehrhard symmetrization of $E$} with respect to the $u$-direction and matrix $A$ (see definition (\ref{new_def})) and
$$b_{\gamma_A}(E)=\int_E x\ d\gamma_A(x)$$
is called the {\bf $A$-anisotropic Gaussian barycenter} of $E$.
\end{theorem}

Theorem \ref{Ehrhard_Sym_Ineq_II} ensures that the anisotropic Gaussian perimeter decreases if we do the Ehrhard symmetrization with respect to any eigenvector direction of $A$. Our final result says that the converse of it is also true, i.e., the only situation in which the anisotropic Gaussian perimeter behaves monotonously is when we Ehrhard symmetrize the set with respect to the eigenvector direction of $A$.

\begin{theorem}\label{Uniqueness_Ehrhard_Sym}
Let $n \geq 2$, $A$ be a symmetric positive definite matrix, and let $u\in \SS^{n-1}$. Then,
\begin{align*}
&P_{\gamma_A}(E_{A,u}^{s} ) \leq P_{\gamma_A}\left(E \right) \mbox{ for all finite $A$-anisotropic Gaussian perimeter set $E$ in $\mathbb{R}^{n}$}\\
&\iff u\in V_\lambda(A)\cap \SS^{n-1}\mbox{ for some $\lambda >0$}
\end{align*} 
where $V_\lambda(A)$ is the eigenspace of $A$ associated with eigenvalue $\lambda$. Moreover,
$$\mbox{$\gamma_A$ is Ehrhard symmetrizable}\iff A=aI_n\mbox{ for some constant $a>0$.}$$
Here $\gamma_A$ is called {\bf Ehrhard symmetrizable} if
$$P_{\gamma_A}(E_{A,u}^s)\leq P_{\gamma_A}(E)$$
for all $u\in \SS^{n-1}$, and for all measurable set $E\subset \R^n$.
\end{theorem}

Now we describe the structure of this paper. We first collect some important definitions and theorems about sets of locally finite perimeter and finite anisotropic Gaussian perimeter in Section \ref{background}. In Section \ref{anisotropic Gaussian Isoperimetric Inequality}, we provide a proof for the anisotropic Gaussian isoperimetric inequality using an approximation argument and characterize the cases of equality (Theorem \ref{AnisotropicGaussainIso}). We introduce the Ehrhard symmetrization in Section \ref{Anisotropic Gaussian Inequality under Ehrhard Symmetrization} with other essential tools and apply those tools to prove Theorem \ref{Ehrhard_Sym_Ineq_II}. Finally, in Section \ref{Characterization of Ehrhard symmetrizable measures}, we discuss some regularity results for Ehrhard symmetrization sets and prove Theorem \ref{Uniqueness_Ehrhard_Sym}.

\addtocontents{toc}{\protect\setcounter{tocdepth}{0}}
\section*{Acknowledgement}
This paper would not be possible without the support of my advisor Prof. Tatiana Toro. I would like to thank her for introducing me to the field of Geometric Measure Theory and this problem about the anisotropic Gaussian measure. Her support and encouragement were vital in carrying me throughout the stressful pandemic. I also want to sincerely thank Dr. Sean McCurdy for the great discussions we had about the Ehrhard symmetrization and minimization problems. I would like to thank Prof. Francesco Maggi for bringing this problem to our attention.
\addtocontents{toc}{\protect\setcounter{tocdepth}{2}}

\section{Background and Notation}\label{background}
	
\subsection{Sets of locally finite perimeter}\mbox{}\vspace{-.22cm}\\

In this section, we will recall some useful definitions and theorems from Maggi's book \cite{maggi} and Evans-Gariepy's book \cite{evansgariepy}.\\

Let $U$ be an open subset in $\R^n$. A function $f\in L^1(U)$ has {\bf bounded variation} in $U$ if
$$\sup\Big\{\int_Uf\div\varphi\ dx: \varphi\in C^1_c(U;\R^n),\ |\varphi|\leq 1\Big\}<\infty.$$
We write $BV(U)$ to denote the space of such functions. A function $f \in L_{\text {loc}}^{1}(U)$ has {\bf locally bounded variation} in $U$ if for every open set $V \subset \subset U$,
$$
\sup \left\{\int_{V} f \operatorname{div} \varphi \ d x: \varphi \in C_{c}^{1}\left(V ; \mathbb{R}^{n}\right),\ |\varphi| \leq 1\right\}<\infty.
$$
We write $B V_{\mathrm{loc}}(U)$ to denote the space of such functions. A $\mathcal{L}^{n}$-measurable subset $E \subset \mathbb{R}^{n}$ has {\bf finite perimeter} in $U$ if $\chi_{E} \in B V(U)$. A $\mathcal{L}^{n}$-measurable subset $E \subset \mathbb{R}^{n}$ has {\bf locally finite perimeter} in $U$ if $\chi_{E} \in B V_{\mathrm{loc}}(U)$.\\

We recall the following theorem from \cite{evansgariepy}, Chapter 5, Theorem 1.

\begin{theorem}\label{structure_theorem}
Let $f\in BV_{loc}(U)$. Then there exists a Radon measure $\mu$ on $U$ and a $\mu$-measurable function $\sigma:U\to \R^n$ such that
\begin{enumerate}
\item $|\sigma(x)|=1$ $\mu$-a.e.;
\item For any $\varphi\in C^1_c(U;\R^n)$,
$$\int_Uf\div\varphi\ dx=-\int_U\varphi\cdot \sigma\:d\mu.$$
\end{enumerate}
We write $|Df|$ for $\mu$, $Df:= \sigma|Df|$, and $D_if:= \sigma_i|Df|$. Moreover,
\begin{align*}
|Df|(V)&=\sup \left\{\int_{V} f \operatorname{div} \varphi \ d x: \varphi \in C_{c}^{1}\left(V ; \mathbb{R}^{n}\right),\ |\varphi| \leq 1\right\}\\
&=\sup \left\{\int_{V}  \varphi \cdot d Df: \varphi \in C_{c}^{1}\left(V ; \mathbb{R}^{n}\right),\ |\varphi| \leq 1\right\}
\end{align*}
for any $V\subset \subset U$, i.e., the total variation of $Df$ is $|Df|$. Also,
$$\mbox{$E$ is a set of locally finite perimeter}\iff |D\chi_E|(K)<\infty \mbox{\quad for every compact set $K\subset \R^n$.}$$
\end{theorem}
\noindent{\bf Remark}: If $f=\chi_{E}$, and $E$ is a set of locally finite perimeter in $U$, we will write 
$$\nu^E:= \sigma,\quad \nu_E:=-\sigma,\quad  \mu_E:= \nu_E |D\chi_E|,$$
where $\nu^E(x)$ ($\nu_E(x)$) is called the {\bf generalized inner (outer) unit normal} of $E$ at $x$ and the $\mathbb{R}^{n}$-valued Radon measure $\mu_{E}$ on $\mathbb{R}^{n}$ is called the \textbf{Gauss-Green measure} of $E$. Let $E$ be a set of locally finite perimeter. The \textbf{reduced boundary} $\partial^* E$ of $E$ is the set of those $x\in \spt \mu_E$ such that 
\beq\label{outer unit vector of reduced boundary}
	\nu_E(x)=\frac{d\mu_E}{d|\mu_E|}(x):=\lim_{r\to 0^+}\frac{\mu_E(B(x,r))}{|\mu_E|(B(x,r))}\text{ exists and is in }\SS^{n-1},
\eeq
where $\spt\mu_E:=\{x:\mu_E(B(x,r))>0\mbox{ for all $r>0$}\}.$
In fact, we have
$$\rb E\subset \spt \mu_E\subset \partial E$$
and $\spt\mu_E=\{x:0<|E\cap B(x,r)|<|B(x,r)|\mbox{ for all $r>0$}\}$.
Moreover, the De Giorgi structure theorem states that $\rb E$ is $(n-1)$-rectifiable and that 
\beq\label{DeGiorgi}
\mu_E=\nu_E \mathcal{H}^{n-1}\mres\: \rb E,\qquad |\mu_E|=\mathcal{H}^{n-1}\mres\: \rb E,
\eeq
or equivalently,
\beq\label{DeGiorgi_inner_normal}
D\chi_E=\nu^E \mathcal{H}^{n-1}\mres\: \rb E,\qquad |D\chi_E|=\mathcal{H}^{n-1}\mres\: \rb E,
\eeq
where $\mathcal{H}^{n-1}$ denotes the $(n-1)$-dimensional Hausdorff measure. Hence, we have the divergence theorem in the following form:
\beq \label{DivergenceThm}
\int_{\rb E}F\cdot \nu_E \ d\mathcal{H}^{n-1}(x)=\int_{\R^n}F\cdot d\mu_E=\int_E\div F, \mbox{\qquad for any $F\in C^1_c(\R^n;\R^n)$},
\eeq
where $\cdot$ is the Euclidean dot product (see \cite{maggi}, Proposition 12.19, Theorem 15.9).\\

Let $E$ be any measurable subset in $\mathbb{R}^{n}$ and $0 \leq d \leq 1$. The set of points of {\bf density $d$ of $E$} is defined as
$$
E^{(d)}=\left\{x \in \mathbb{R}^{n}: \theta_n(E)(x):= \lim _{\rho \rightarrow 0} \frac{\mathcal{L}^{n}\left(E \cap Q_{\rho}(x)\right)}{\mathcal{L}^{n}\left(Q_{\rho}(x)\right)}=d\right\}
$$
where $Q_{\rho}(x)$ is the cube centered at $x$, whose sides are parallel to the coordinate axes with length $2 \rho$. We will use $|\cdot|$ or $\mathcal{L}^n$ for Lebesgue measure on $\R^n$. By Lebesgue points theorem,
$$\theta_n(E)=1\mbox{\quad a.e. on $E$},\qquad \theta_n(E)=0\mbox{\quad a.e. on $\R^n\slash E$}.$$
Therefore, $|E\Delta E^{(1)}|=0$, i.e., every Lebesgue measurable set $E$ is equivalent to $E^{(1)}$. Now we introduce the {\bf essential boundary} $\partial^M E$ of a measurable set $E$ defined as
$$\partial^M E:= \R^n\slash \left(E^{(0)}\cup E^{(1)}\right).$$
Then Federer's theorem tells us that for any set of locally finite perimeter $E$,
$$\partial^* E\subset E^{(1/2)}\subset \partial^M E,\qquad \mathcal{H}^{n-1}(\partial^M E\slash \rb E)=0.$$ 
Moreover,
\begin{align}\label{Federer's theorem}
\mbox{$E$ is a set of locally finite perimeter}\iff\mathcal{H}^{n-1}(\partial^M E\cap K)<\infty,\mbox{\ $\forall K$ compact in $\R^n$}.
\end{align}
We also define the {\bf (relative) perimeter of $E$ in $F$} as 
$$
P(E ; F)=\mathcal{H}^{n-1}(\partial^M E\cap F),
$$
for any Borel set $F\subset \R^n$ (see \cite{maggi}, Corollary 15.8, Theorem 16.2 and \cite{federergmt}, Theorem 4.5.11).

\subsection{Important background results}\mbox{}\vspace{-.22cm}\\

In this subsection, we collect some significant results that will be used in the later sections.

\begin{proposition}[\cite{maggi}, Proposition 4.29]\mbox{}\\
If $\mu_k$ and $\mu$ are vector-valued Radon measures with $\mu_k \wkly \mu$, then for every open set $A \subset \mathbb{R}^n$ we have
$$
|\mu|(A) \leq \liminf _{k \rightarrow \infty}\left|\mu_k\right|(A).
$$
\end{proposition}

\begin{proposition}[Diffeomorphic images of sets of finite perimeter (\cite{maggi}, Proposition 17.1)]\label{Diffeomorphic images}\mbox{}\\ 
If $E$ is a set of locally finite perimeter in $\mathbb{R}^{n}$ and $f$ is a diffeomorphism of $\mathbb{R}^{n}$ with $g=f^{-1}$, then $f(E)$ is a set of locally finite perimeter in $\mathbb{R}^{n}$ with $\mathcal{H}^{n-1}\left(f\left(\partial^{*} E\right) \Delta \partial^{*} f(E)\right)=0 $, and
$$
\int_{\partial^{*} f(E)} \varphi \nu_{f(E)} \ d\mathcal{H}^{n-1}=\int_{\partial^{*} E}(\varphi \circ f) J f(D g \circ f)^{*} \nu_{E} \ d\mathcal{H}^{n-1}
$$
for every $\varphi \in C_{\mathrm{c}}\left(\mathbb{R}^{n}\right)$, where $Jf=|\det(Df)|$ is the Jacobian of $f$ on $\R^n$.
\end{proposition}

\begin{theorem}[Ehrhard's Inequality (\cite{Borell_EhrhardIneq}, Theorem 1.1)]\label{Ehrhard's Inequality}\mbox{}\\ 
If $A,B$ are Borel sets in $\R^n$, then
$$
\phi^{-1}(\gamma(\lambda A+(1-\lambda) B)) \geq \lambda \phi^{-1}(\gamma(A))+(1-\lambda) \phi^{-1}(\gamma(B)), \text {\qquad for } \lambda \in(0,1),
$$
where
$$\gamma(E)=\frac{1}{(2\pi)^{n/2}}\int_Ee^{-|x|^2/2}\ dx,\qquad\phi(x)=\frac{1}{\sqrt{2\pi}}\int_{-\infty}^xe^{-t^2/2}\ dt.$$
\end{theorem}

Recall that we define $\R^m$-valued Radon measure $\mu$ on $\R^n$ as the bounded linear functional on $C_c(\R^n,\R^m)$ and set
$$\langle \mu,\varphi\rangle:=\int_{\R^n}\varphi\cdot d\mu,\qquad \varphi\in C_c(\R^n;\R^m).$$
The following three propositions are useful when we calculate the total variation. The first one can be found in \cite{maggi}, Remark 4.8, and the rest are straightforward applications of the results in \cite{maggi}, Chapter 4.

\begin{proposition}[\cite{maggi}, Remark 4.8]\label{total_variation_1}\mbox{}\\
Let $\mu$ be a Radon measure on $\R^n$ and let $f:\R^n\to \R^m$ be a $\R^m$-valued function with $f\in L^1_{loc}(\R^n,\mu;\R^m)$. Then we may define a bounded linear functional $f\mu: C_c(\R^n;\R^m)\to \R$ as
$$\langle f\mu ,\varphi\rangle:=\int_{\R^n} f\cdot \varphi\ d\mu$$
for any $\varphi \in C_c(\R^n;\R^m)$, i.e., $f\mu$ is a $\R^m$-valued Radon measure on $\R^n$. Moreover, the total variation of $f\mu$ is $|f\mu|=|f|\mu$, where
$$|f|\mu(E):=\int_E|f|\ d\mu,\qquad E\in \mathcal{B}(\R^n). $$
\end{proposition}

\begin{proposition}\label{total_variation_2}
Let $\mu$ be a $\R^m$-valued Radon measure on $\R^n$ and let $h:\R^n\to \R$ be a real-valued bounded Borel function. Then we may
define a bounded linear functional $h\mu: C_c(\R^n;\R^m)\to \R$ as
$$\langle h\mu ,\varphi\rangle:=\int_{\R^n} h\varphi\cdot d\mu$$
for any $\varphi \in C_c(\R^n;\R^m)$, i.e., $h\mu$ is a $\R^m$-valued Radon measure on $\R^n$. Moreover, the total variation of $h\mu$ is $|h\mu|=|h||\mu|$, where
$$|h||\mu|(E):=\int_E|h|\ d|\mu|,\qquad E\in \mathcal{B}(\R^n). $$
\end{proposition}

\begin{proposition}\label{total_variation_3}
Let $\mu$ be a $\R^m$-valued Radon measure on $\R^n$ and let $f:\R^n\to \R^n$ be a homeomorphism. Then we may
define a bounded linear functional $f_{\#} \mu: C_c(\R^n;\R^m)\to \R$ as
$$\langle f_{\#} \mu ,\varphi\rangle:=\int_{\R^n} (\varphi\circ f)\cdot d\mu$$
for any $\varphi \in C_c(\R^n;\R^m)$, i.e., $f_{\#} \mu$ is a $\R^m$-valued Radon measure on $\R^n$. Moreover, the total variation of $f_{\#} \mu$ is $|f_{\#} \mu|=f_{\#} |\mu|$, where
$$f_{\#} |\mu|(E):=|\mu|(f^{-1}(E)),\qquad E\in \mathcal{B}(\R^n) $$
is called the {\bf push-forward of $|\mu|$ through $f$}.
\end{proposition}

\subsection{Anisotropic Gaussian Hausdorff measure and anisotropic Gaussian perimeter}\mbox{}\vspace{-.22cm}\\

Let $A\in M_n(\R)$ be a symmetric positive definite matrix. There exists a unique symmetric positive definite matrix $\sqrt{A}$ such that 
$$A=(\sqrt{A})^2$$
(see \cite{Horn-Johnson}, Theorem 7.2.6). We will use the notation $A\succ 0$ ($A \succeq 0$) to mean the matrix $A$ is symmetric positive definite (symmetric positive semi-definite). Notice that we have the following equalities:
$$e^{-\langle Ax,x\rangle/2}=e^{-\langle \sqrt{A}x,\sqrt{A}x\rangle/2}=e^{-|\sqrt{A}x|^2/2},\qquad \sqrt{\det A}=\det\sqrt{A}.$$
The {\bf matrix norm induced by the Euclidean norm} is defined as
$$\|A\|:=\max\limits_{\|x\|_2=1}\|Ax\|_2=\max\limits_{\|x\|_2=\|y\|_2=1}\langle Ax,y\rangle.$$
Notice that
$$\|A\|=\|\sqrt{A}\sqrt{A}\|=\Big\|\sqrt{A}^\mathsf{T}\sqrt{A}\Big\|=\|\sqrt{A}\|^2\implies \|A\|^{\frac12}=\|\sqrt{A}\|.$$
In addition, $\sqrt{A^{-1}}=(\sqrt{A})^{-1}$ and hence
$$\|A^{-1}\|^{\frac12}=\|(\sqrt{A})^{-1}\|.$$
Now we define the {\bf $A$-anisotropic Gaussian measure (mass)} as
$$\gamma_A(E)=\fint_E e^{-\langle Ax,x\rangle/2}\ dx=\frac{\sqrt{\det A}}{(2\pi)^{\frac{n}{2}}}\int_{E} e^{-\langle Ax,x\rangle /2}\ dx,$$
for any (Lebesgue) measurable set $E\subset \R^n$. The connection between $\gamma_A$ and $\gamma$ is
$$\gamma_A(E)=\gamma(\sqrt{A}(E)),$$
where $\gamma:= \gamma_{I_n}$ is the (standard) Gaussian measure on $\R^n$, i.e.,
$$\gamma(E)=\frac{1}{(2\pi)^{n/2}}\int_{E}e^{-|x|^2/2}\ dx.$$

Given any $k \in \mathbb{N}$ with $0 \leq k \leq n$, we define the {\bf $k$-dimensional $A$-anisotropic Gaussian Hausdorff measure} $\mathcal{H}_{\gamma_A}^{k}$ by
\begin{align}\label{Hausdorff_measure}
\mathcal{H}_{\gamma_A}^{k}(B)=\frac{\sqrt{\det A}}{(2\pi)^{\frac{k}{2}}}\int_{B}e^{-\langle Ax,x\rangle /2}\ d \mathcal{H}^{k}(x),\mbox{\quad for any Borel set $B$}.
\end{align}

Let $E$ be a measurable set in $\R^n$ and $F$ be a Borel set in $\R^n$. The  {\bf (relative) $A$-anisotropic Gaussian perimeter of $E$ in $F$} is defined by
$$
P_{\gamma_A}(E ; F)=\mathcal{H}_{\gamma_A}^{n-1}\left(\eb E \cap F\right) ,
$$
and we say $E$ is a set of {\bf locally finite $A$-anisotropic Gaussian perimeter} if $P_{\gamma_A}(E;K)=\mathcal{H}_{\gamma_A}^{n-1}(\eb E\cap K)<\infty$ for every compact set $K\subset \R^n$. In particular, $E$ is a set of {\bf finite $A$-anisotropic Gaussian perimeter} if $P_{\gamma_A}(E)<\infty$. We will omit the notation $A$ and simply say $E$ is a set of finite anisotropic Gaussian perimeter if there is no confusion.

\begin{proposition}\label{locally finite anisotropic Gaussian perimeter}
$E$ is a set of locally finite perimeter if and only if $E$ is a set of locally finite anisotropic Gaussian perimeter, that is, $P_{\gamma_A}(E;K)=\mathcal{H}_{\gamma_A}^{n-1}(\eb E\cap K)<\infty$ for every compact set $K\subset \R^n$.
\end{proposition}
\begin{proof} $(\Ra)$ For any compact set $K$,
\begin{align*}
P_{\gamma_A}(E;K)&=\frac{\sqrt{\det A}}{(2\pi)^{\frac{n-1}{2}}}\int_{\eb E\cap K} e^{-|\sqrt{A}x|^2/2}d \mathcal{H}^{n-1}(x)\leq\frac{\sqrt{\det A}}{(2\pi)^{\frac{n-1}{2}}}\int_{\eb E\cap K}1\ d \mathcal{H}^{n-1}(x)=\frac{\sqrt{\det A}}{(2\pi)^{\frac{n-1}{2}}} P(E;K)<\infty.
\end{align*}
$(\La)$ For any compact set $K$, let
$$m_K:=\min_{x\in K}e^{-|\sqrt{A}x|^2/2}\geq \min_{x\in K}e^{-d_{\max}^2|x|^2/2}>0,$$
where $d_{\max}$ is the largest eigenvalue of $\sqrt{A}$. Then
\begin{align*}
\infty>P_{\gamma_A}(E;K)&=\frac{\sqrt{\det A}}{(2\pi)^{\frac{n-1}{2}}}\int_{\eb E\cap K} e^{-|\sqrt{A}x|^2/2}d \mathcal{H}^{n-1}(x)\geq \frac{\sqrt{\det A}}{(2\pi)^{\frac{n-1}{2}}}m_KP(E;K).
\end{align*}
\end{proof}
\noindent{\bf Remark}: It is clear that
$$\mbox{$E$ is a set of finite perimeter $\implies$ $E$ is a set of finite anisotropic Gaussian perimeter.}$$
However, the converse is not true. For example, let $n=2$, $E_\alpha=[-\alpha,\alpha]\times(0,\infty)$, and
$$A=2\left[ \begin{array}{rr} a & b \\ b & c \end{array} \right]\succ 0$$
with $a,c>0$ and $b>0$. Then
$$e^{-ax^2-2bxy-cy^2}=e^{-\langle A(x,y),(x,y)\rangle/2}=e^{-|\sqrt{A}(x,y)|^2/2}\leq e^{-\|(\sqrt{A})^{-1}\|^{-2}(x^2+y^2)/2}\leq e^{-\|(\sqrt{A})^{-1}\|^{-2}y^2/2}$$
and
\begin{align*}
P_{\gamma_A}(E_\alpha)&=\frac{\sqrt{\det A}}{\sqrt{2\pi}}\left(\int_{0}^{\infty}e^{-a\alpha^2-2b\alpha y-cy^2}\ dy+\int_{0}^{\infty}e^{-a\alpha^2+2b\alpha y-cy^2}\ dy+\int_{-\alpha}^{\alpha}e^{-ax^2}\ dx\right)\\
&\leq \frac{\sqrt{\det A}}{\sqrt{2\pi}}\left(2\int_0^{\infty}e^{-\|(\sqrt{A})^{-1}\|^{-2}y^2/2}dy+2\int_{0}^{\alpha}e^{-ax^2}\ dx\right)<\infty.
\end{align*}
That is, $E_\alpha$ is a set of finite anisotropic Gaussian perimeter and clearly $P(E_\alpha)=\infty$, i.e., $E_\alpha$ is not a set of finite perimeter.\\

Now we establish the lower semicontinuity, locality, complementation, and subadditivity for the anisotropic Gaussian perimeter. Proposition \ref{lower_semicontinuity} is a straightforward consequence of Proposition \ref{locally finite anisotropic Gaussian perimeter}. The locality, complementation, and subadditivity can be deduced by Proposition \ref{lower_semicontinuity} with results in \cite{maggi}, Chapter 16.

\begin{proposition}[Lower semicontinuity]\label{lower_semicontinuity}\mbox{}\\
If $E$ is a set of locally finite anisotropic Gaussian perimeter and $U \subseteq \mathbb{R}^{n}$ is an open set, then
$$
P_{\gamma_A}(E ; U)=\sqrt{2\pi}\sup \left\{\int_{E} \operatorname{div} \varphi(x)-\langle \varphi(x), Ax\rangle\ d \gamma_{A}(x): \varphi \in C_{c}^{1}\left(U ; \mathbb{R}^{n}\right), \sup _{U}|\varphi| \leq 1\right\} .
$$
Moreover, for any sequence of sets of locally finite perimeter $E_{k}$ with $\chi_{E_{k}} \rightarrow \chi_{E}$ in $L_{\text {loc}}^{1}\left(\mathbb{R}^{n},\gamma_A\right)$, 
$$
P_{\gamma_A}(E ; U) \leq \liminf _{k \rightarrow \infty} P_{\gamma_A}\left(E_{k} ; U\right) .
$$
Conversely, if $E$ is a measurable set, $U$ is an open set, and for any open set $V\subset \subset U$,
$$
\sup \left\{\int_{E} \operatorname{div} \varphi(x)-\langle \varphi(x), Ax\rangle\ d \gamma_{A}(x): \varphi \in C_{c}^{1}\left(V ; \mathbb{R}^{n}\right), \sup _{V}|\varphi| \leq 1\right\} <\infty,
$$
then $E$ is a set of locally finite anisotropic Gaussian perimeter in $U$.
\end{proposition}

\begin{proposition}[Properties of perimeter]\label{Properties of perimeter}\mbox{}
\begin{enumerate}
\item {\rm (Locality)} Let $E$ be a set of locally finite perimeter. If $F$ is equivalent to $E$ in some open set $U\subset \R^n$, i.e., $|(E\Delta F)\cap U|=0$, then
$$P_{\gamma_A}(E;U)=P_{\gamma_A}(F;U).$$
\item {\rm (Complementation)} Let $E$ be a set of locally finite perimeter and $U$ be an open set in $\R^n$. Then $E^c$ is also a set of locally finite perimeter and
$$P_{\gamma_A}(E;U)=P_{\gamma_A}(E^c;U).$$
\item {\rm (Subadditivity)} If $E,F$ are sets of locally finite perimeter and $U$ is an open set in $\R^n$,
$$P_{\gamma_A}(E\cup F;U)+P_{\gamma_A}(E\cap F;U)\leq P_{\gamma_A}(E;U)+P_{\gamma_A}(F;U).$$
\end{enumerate}
\end{proposition}

Although the anisotropic Gaussian measure satisfies $\gamma_A(E)=\gamma_{I_n}(\sqrt{A}E)$, this kind of relation doesn't hold in the anisotropic Gaussian perimeter, i.e.,
$$P_{\gamma_A}(E)\not=P_{\gamma_{I_n}}(\sqrt{A}E).$$
\noindent In fact, we have the following formula:

\begin{proposition}\label{Gaussian_perimeter}
If $E$ is a set of locally finite perimeter, then
\begin{align*}
\int_{F\cap \rb (\sqrt{A}E)}\nu_{\sqrt{A}E}(y)\ d\mathcal{H}_{\gamma}^{n-1}(y)
&=\int_{\big((\sqrt{A})^{-1}F\big)\cap\rb E}\Big[(\sqrt{A})^{-1}\nu_E(x)\Big]\ d\mathcal{H}_{\gamma_A}^{n-1}(x),
\end{align*}
equivalently,
\begin{align*}
\int_{F\cap \rb (\sqrt{A}E)}\left[\sqrt{A}\nu_{\sqrt{A}E}(y)\right] d\mathcal{H}_{\gamma}^{n-1}(y)
&=\int_{\big((\sqrt{A})^{-1}F\big)\cap\rb E}\nu_E(x)\ d\mathcal{H}_{\gamma_A}^{n-1}(x),
\end{align*}
and hence
$$P_{\gamma_{I_n}}(\sqrt{A}E;F)=\int_{\big((\sqrt{A})^{-1}F\big)\cap \rb E}\Big|(\sqrt{A})^{-1}\nu_E(x)\Big|\ d\mathcal{H}_{\gamma_A}^{n-1}(x),$$
$$P_{\gamma_{A}}(E;(\sqrt{A})^{-1}F)=\int_{F\cap \rb (\sqrt{A}E)}\Big|\sqrt{A}\nu_{\sqrt{A}E}(x)\Big|\ d\mathcal{H}_{\gamma}^{n-1}(x),$$
for any Borel set $F\subset \R^n$. Moreover,
\begin{enumerate}
\item if $F=\R^n$, then
$$\|\sqrt{A}\|^{-1}P_{\gamma_A}(E)\leq P_{\gamma_{I_n}}(\sqrt{A}E)\leq \|(\sqrt{A})^{-1}\|P_{\gamma_A}(E).$$
\item if $O$ is an orthogonal matrix and $A=I_n$, then
$$P_{\gamma_{I_n}}(OE)=P_{\gamma_{I_n}}(E).$$
\item if $O$ is an orthogonal matrix, then
$$P_{\gamma_{A}}(E;OF)=P_{\gamma_{O^{\mathsf{T}}AO}}(O^{-1}E;F).$$
In particular, $E$ is a set of finite $A$-anisotropic Gaussian perimeter if and only if $O^{-1}E$ is a set of finite $O^{\mathsf{T}}AO$-anisotropic Gaussian perimeter.
\end{enumerate}
\end{proposition}
\begin{proof}
Since $E$ is a set of locally finite perimeter and $x\mapsto \sqrt{A}x$ is a diffeomorphism, by Proposition \ref{Diffeomorphic images}, for any $\varphi\in C_c(\R^n)$,
\begin{align}\label{Gaussian_perimeter_eq0}
\int_{\rb (\sqrt{A}E)}\varphi(y)\nu_{\sqrt{A}E}(y)\ d\mathcal{H}^{n-1}(y)&=|\det(\sqrt{A})|\int_{\rb E}\varphi(\sqrt{A}x)\Big[[(\sqrt{A})^{-1}]^{\mathsf{T}}\nu_E(x)\Big]\ d\mathcal{H}^{n-1}(x)\notag\\
&=|\det(\sqrt{A})|\int_{\rb E}\varphi(\sqrt{A}x)\Big[(\sqrt{A})^{-1}\nu_E(x)\Big]\ d\mathcal{H}^{n-1}(x),
\end{align}
where $\sqrt{A}$ is symmetric. Let $F$ be a Borel set. We can set
$$\varphi(y)=\frac{1}{(2\pi)^{(n-1)/2}}e^{-|y|^2/2}\chi_F(y)$$
in (\ref{Gaussian_perimeter_eq0}) since we can first approximate open sets then Borel sets. Hence,
\begin{align}\label{Gaussian_perimeter_eq1}
\int_{F\cap \rb (\sqrt{A}E)}\nu_{\sqrt{A}E}(y)\ d\mathcal{H}_{\gamma}^{n-1}(y)
&=\int_{\big((\sqrt{A})^{-1}F\big)\cap\rb E}\Big[(\sqrt{A})^{-1}\nu_E(x)\Big]\ d\mathcal{H}_{\gamma_A}^{n-1}(x).
\end{align}
Since $F$ is an arbitrary Borel set, the following two measures are the same:
$$\nu_{\sqrt{A}E}\mathcal{H}^{n-1}_\gamma\mres  \rb (\sqrt{A}E)=(\sqrt{A})_{\#}\left(\Big[(\sqrt{A})^{-1}\nu_E\Big]\mathcal{H}_{\gamma_A}^{n-1}\mres \rb E\right).$$
Taking total variation on both sides, by Proposition \ref{total_variation_1} and \ref{total_variation_3}, we have
$$\mathcal{H}^{n-1}_\gamma\mres \rb (\sqrt{A}E)=(\sqrt{A})_{\#}\left(\Big|(\sqrt{A})^{-1}\nu_E\Big|\mathcal{H}_{\gamma_A}^{n-1}\mres \rb E\right).$$
That is,
\begin{align}\label{Gaussian_perimeter_eq2}
P_{\gamma}(\sqrt{A}E;F)=\int_{\big((\sqrt{A})^{-1}F\big)\cap \rb E}\Big|(\sqrt{A})^{-1}\nu_E(x)\Big|\ d\mathcal{H}_{\gamma_A}^{n-1}(x).
\end{align}
Applying $\sqrt{A}$ on both sides of equation (\ref{Gaussian_perimeter_eq1}),
\begin{align*}
\int_{F\cap \rb (\sqrt{A}E)}\left[\sqrt{A}\nu_{\sqrt{A}E}(y)\right] d\mathcal{H}_{\gamma}^{n-1}(y)
&=\int_{\big((\sqrt{A})^{-1}F\big)\cap\rb E}\nu_E(x)\ d\mathcal{H}_{\gamma_A}^{n-1}(x).
\end{align*}
Taking total variation on both sides, by Proposition \ref{total_variation_1} and \ref{total_variation_3}, we have
\begin{align}\label{Gaussian_perimeter_eq2-2}
P_{\gamma_{A}}(E;(\sqrt{A})^{-1}F)=\int_{F\cap \rb (\sqrt{A}E)}\Big|\sqrt{A}\nu_{\sqrt{A}E}(x)\Big|\ d\mathcal{H}_{\gamma}^{n-1}(x).
\end{align}
In fact, by using the same argument in (\ref{Gaussian_perimeter_eq1}) with any diffeomorphism $f$ on $\R^n$, we have
\begin{align}\label{Gaussian_perimeter_eq3}
&\int_{F\cap \rb f(E)}\nu_{f(E)}(y)\ d\mathcal{H}_{\gamma}^{n-1}(y)\\
&=\frac{1}{(2\pi)^{(n-1)/2}}\int_{f^{-1}(F)\cap\rb E}e^{-|f(x)|^2/2}\Big[Jf(\nabla g\circ f)^*\nu_E(x)\Big]\ d\mathcal{H}^{n-1}(x),\notag
\end{align}
where $g=f^{-1}$.\\
\noindent(1) Let $F=\R^n$ in (\ref{Gaussian_perimeter_eq2}). Then
$$P_{\gamma}(\sqrt{A}E)=\int_{\rb E}\Big|(\sqrt{A})^{-1}\nu_E(x)\Big|\ d\mathcal{H}_{\gamma_A}^{n-1}(x)\leq \|(\sqrt{A})^{-1}\|P_{\gamma_A}(E).$$
On the other hand,
$$1=\Big|\nu_E(x)\Big|=\Big|\sqrt{A}(\sqrt{A})^{-1}\nu_E(x)\Big|\leq \|\sqrt{A}\|\Big|(\sqrt{A})^{-1}\nu_E(x)\Big|\implies \Big|(\sqrt{A})^{-1}\nu_E(x)\Big|\geq \|\sqrt{A}\|^{-1}.$$
Thus,
$$P_{\gamma}(\sqrt{A}E)=\int_{\rb E}\Big|(\sqrt{A})^{-1}\nu_E(x)\Big|\ d\mathcal{H}_{\gamma_A}^{n-1}(x)\geq \|\sqrt{A}\|^{-1}P_{\gamma_A}(E).$$
(2) Let $O$ be a positive definite orthogonal matrix. Applying (\ref{Gaussian_perimeter_eq3}) to the map $f:x\mapsto Ox$, we have
$$P_{\gamma}(OE)=\int_{\rb E}\Big|O\nu_E(x)\Big|\ d\mathcal{H}_{\gamma}^{n-1}(x)=P_{\gamma}(E).$$
(3) Replacing $F$ as $\sqrt{A}OF$ in (\ref{Gaussian_perimeter_eq2-2}), we have
\begin{align}\label{Gaussian_perimeter_eq5}
P_{\gamma_A}(E;OF)=\int_{(\sqrt{A}OF)\cap \rb (\sqrt{A}E)}\left|\sqrt{A}\nu_{\sqrt{A}E}\right|d\mathcal{H}_{\gamma}.
\end{align}
On the other hand, by replacing $F$ as $\sqrt{O^{\mathsf{T}}AO}F$ in (\ref{Gaussian_perimeter_eq2-2}) again, we have
\begin{align}\label{Gaussian_perimeter_eq6}
P_{\gamma_{O^{\mathsf{T}}AO}}(O^{-1}E;F)&=\int_{(O^{\mathsf{T}}\sqrt{A}OF)\cap \rb (O^{\mathsf{T}}\sqrt{A}O(O^{-1}E))}\left|O^{\mathsf{T}}\sqrt{A}O\nu_{O^{\mathsf{T}}\sqrt{A}O(O^{-1}E)}\right|d\mathcal{H}_{\gamma}\\
&=\int_{(O^{\mathsf{T}}\sqrt{A}OF)\cap \rb (O^{\mathsf{T}}\sqrt{A}E)}\left|\sqrt{A}O\nu_{O^{\mathsf{T}}\sqrt{A}E}\right|d\mathcal{H}_{\gamma},\notag
\end{align}
where $\sqrt{O^{\mathsf{T}}AO}=O^{\mathsf{T}}\sqrt{A}O$. Applying $\sqrt{A}O$ on both sides of equation (\ref{Gaussian_perimeter_eq3}) with $f(x)=O^{\mathsf{T}}x$, $E$ as $\sqrt{A}E$, and $F$ as $O^{\mathsf{T}}\sqrt{A}OF$, we have
\begin{align*}
\int_{(O^{\mathsf{T}}\sqrt{A}OF)\cap \rb (O^{\mathsf{T}}\sqrt{A}E)}\sqrt{A}O\nu_{O^{\mathsf{T}}\sqrt{A}E}\ d\mathcal{H}_{\gamma}
&=\int_{(\sqrt{A}OF)\cap \rb (\sqrt{A}E)}\sqrt{A}\nu_{\sqrt{A}E}\ d\mathcal{H}_{\gamma}.
\end{align*}
Taking total variation on both sides, we have
\begin{align}\label{Gaussian_perimeter_eq7}
\int_{(O^{\mathsf{T}}\sqrt{A}OF)\cap \rb (O^{\mathsf{T}}\sqrt{A}E)}\left|\sqrt{A}O\nu_{O^{\mathsf{T}}\sqrt{A}E}\right|d\mathcal{H}_{\gamma}
&=\int_{(\sqrt{A}OF)\cap \rb (\sqrt{A}E)}\left|\sqrt{A}\nu_{\sqrt{A}E}\right|d\mathcal{H}_{\gamma}.
\end{align}
Therefore, by (\ref{Gaussian_perimeter_eq5}), (\ref{Gaussian_perimeter_eq6}), and (\ref{Gaussian_perimeter_eq7}),
\begin{align*}
P_{\gamma_A}(E;OF)=P_{\gamma_{O^{\mathsf{T}}AO}}(O^{-1}E;F).
\end{align*}
\end{proof}

\subsection{Approximation for the finite anisotropic Gaussian perimeter}

\begin{proposition}\label{approximation_gaussian_perimeter}
For any measurable set $E$ with $P_{\gamma_A}(E)<\infty$, there exists a sequence $\{E_k\}$ of bounded open sets with smooth boundary such that
$$\chi_{E_k}\to \chi_{E}\mbox{ in $L^1(\R^n,\gamma_A)$}\quad \mbox{and}\quad P_{\gamma_A}(E_k)\to P_{\gamma_A}(E).$$
\end{proposition}
\begin{proof}
For any measurable set $E$ with $P_{\gamma_A}(E)<\infty$, by Proposition \ref{locally finite anisotropic Gaussian perimeter}, $E$ is a set of locally finite perimeter. Hence, $E\cap B_R$ is also a set of finite perimeter, where $B_R:= B(0,R)$ is an open ball (see \cite{maggi}, Lemma 15.12). \\

\noindent{\bf Step 1}: We first claim that, as $R\to \infty$,
$$\chi_{E\cap B_{R}}\to \chi_{E}\mbox{ in $L^1(\R^n,\gamma_A)$}\quad\mbox{and}\quad P_{\gamma_A}(E\cap B_{R})\to P_{\gamma_A}(E).$$
To begin, by the dominated convergence theorem and $\gamma_A(\R^n)=1<\infty$,
$$\int_{\R^n}\left|\chi_E-\chi_{E\cap B_R}\right| d\gamma_A \to 0\mbox{\quad as $R\to \infty.$}$$
For the second part, since $P_{\gamma_A}(E)<\infty$,
$$\lim_{R\to \infty}P_{\gamma_A}(E;\R^n\slash B_{R})=0.$$
Moreover, by \cite{maggi}, Lemma 15.12, for a.e. $R>0$,
$$|\mu_{E \cap B_R}|=\mathcal{H}^{n-1}\mres(E \cap \partial B_R)+|\mu_{E}|\mres B_R.$$
That is,
\begin{align}\label{approximation_gaussian_perimeter_eq1}
P_{\gamma_A}(E\cap B_R)&=\mathcal{H}^{n-1}_{\gamma_A}(\rb (E\cap B_R)) =\frac{\sqrt{\det A}}{(2\pi)^{(n-1)/2}}\int_{\rb (E\cap B_R)}e^{-|\sqrt{A}x|^2/2}\ d\mathcal{H}^{n-1}\\
&=\frac{\sqrt{\det A}}{(2\pi)^{(n-1)/2}}\int_{\R^n}e^{-|\sqrt{A}x|^2/2}\ d|\mu_{E\cap B_R}|\nonumber\\
&=\frac{\sqrt{\det A}}{(2\pi)^{(n-1)/2}}\left(\int_{B_R}e^{-|\sqrt{A}x|^2/2}\ d|\mu_{E}|+\int_{E \cap \partial B_R}e^{-|\sqrt{A}x|^2/2}d\mathcal{H}^{n-1}(x)\right)\nonumber\\
&=\frac{\sqrt{\det A}}{(2\pi)^{(n-1)/2}}\int_{(\rb E)\cap B_R}e^{-|\sqrt{A}x|^2/2}\ d\mathcal{H}^{n-1}+\mathcal{H}_{\gamma_A}^{n-1}(E\cap \partial B_R)\nonumber\\
&=\mathcal{H}_{\gamma_A}^{n-1}((\rb E)\cap B_R)+\mathcal{H}_{\gamma_A}^{n-1}(E\cap \partial B_R)=P_{\gamma_A}(E; B_R)+\mathcal{H}_{\gamma_A}^{n-1}(E\cap \partial B_R).\nonumber
\end{align}
Notice that since
$$|x|=|(\sqrt{A})^{-1}\sqrt{A}x|\leq \|(\sqrt{A})^{-1}\||\sqrt{A}x|\implies |\sqrt{A}x|\geq \frac{1}{\|(\sqrt{A})^{-1}\|}|x|,$$
we have
\begin{align}\label{approximation_gaussian_perimeter_eq2}
\mathcal{H}^{n-1}_{\gamma_A}(E\cap \partial B_{R})\leq \mathcal{H}^{n-1}_{\gamma_A}( \partial B_{R})&=\frac{\sqrt{\det A}}{(2\pi)^{(n-1)/2}}\int_{\partial B_R}e^{-|\sqrt{A}x|^2/2}d\mathcal{H}^{n-1}(x)\\
&\leq \frac{\sqrt{\det A}}{(2\pi)^{(n-1)/2}}\int_{\partial B_R}e^{-\frac{1}{\|(\sqrt{A})^{-1}\|^2}|x|^2/2}d\mathcal{H}^{n-1}(x)\nonumber\\
&= \frac{\sqrt{\det A}}{(2\pi)^{(n-1)/2}}e^{-\frac{1}{\|(\sqrt{A})^{-1}\|^2}R^2/2}\mathcal{H}^{n-1}(\partial B_R)\nonumber\\
&= \frac{\sqrt{\det A}}{(2\pi)^{(n-1)/2}}e^{-\frac{1}{\|(\sqrt{A})^{-1}\|^2}R^2/2}\alpha_n R^{n-1}\to 0\mbox{\quad as $R\to \infty$},\nonumber
\end{align}
where $\alpha _n$ is the surface area of the unit ball in $\R^n$. Combining (\ref{approximation_gaussian_perimeter_eq1}) and (\ref{approximation_gaussian_perimeter_eq2}) together, we have
\begin{align*}
P_{\gamma_A}(E)&=P_{\gamma_A}(E;\R^n\slash B_R)+P_{\gamma_A}(E;B_R)\\
&=P_{\gamma_A}(E;\R^n\slash B_R)+(P_{\gamma_A}(E\cap B_R)-\mathcal{H}^{n-1}_{\gamma_A}(E\cap \partial B_R)),
\end{align*}
and hence, as $R\to \infty$,
$$|P_{\gamma_A}(E)-P_{\gamma_A}(E\cap B_R)|\leq P_{\gamma_A}(E;\R^n\slash B_R)+\mathcal{H}^{n-1}_{\gamma_A}(E\cap \partial B_R)\to 0.$$
{\bf Step 2}: Consider
$$E^R:= E\cap B_R.$$ 
Fix $R>0$. Applying \cite{maggi}, Theorem 13.8, on $E^R$, there exists a sequence $\{E^R_k\}_{k=1}^{\infty}$ of bounded open sets  with smooth boundary such that $E_k^R\subset B_{R+1}$ for all $k$,
\begin{align}\label{approximation_gaussian_perimeter_eq3}
\chi_{E_k^R}\to \chi_{E^R}\mbox{\quad as $k\to \infty$ in $L^1(\R^n)$ (and hence in $L^1(\R^n,\gamma_A)$)},
\end{align}
and
$$|\mu_{E^R_k}|\wkly |\mu_{E^R}|.$$
Let $\eta_\ep\in C^{\infty}_c(\R^n)$ be a smooth cutoff function with $\eta_\ep= 1$ on $B_{R+1}$ and $\eta_\ep\to \chi_{B_L}$ with $L>R+1$. Applying the weak-star convergence, as $k\to \infty$,
\begin{align*}
P_{\gamma_A}(E_k^R)&=\frac{\sqrt{\det A}}{(2\pi)^{(n-1)/2}}\int_{\partial E_k^R}e^{-|\sqrt{A}x|^2/2}\ d\mathcal{H}^{n-1}=\frac{\sqrt{\det A}}{(2\pi)^{(n-1)/2}}\int_{\partial E_k^R}e^{-|\sqrt{A}x|^2/2}\eta_{\ep}(x)\ d\mathcal{H}^{n-1}\\
&=\frac{\sqrt{\det A}}{(2\pi)^{(n-1)/2}}\int_{\R^n}e^{-|\sqrt{A}x|^2/2}\eta_{\ep}(x)d|\mu_{E_k^R}|\underset{k\to \infty}{\to}  \frac{\sqrt{\det A}}{(2\pi)^{(n-1)/2}}\int_{\rb E^R}e^{-|\sqrt{A}x|^2/2}\eta_{\ep}(x)d\mathcal{H}^{n-1}.
\end{align*}
Taking  $\ep\to 0^+$ and $L\to \infty$ on both sides, we have
\begin{align}\label{approximation_gaussian_perimeter_eq4}
\lim_{k\to \infty}P_{\gamma_A}(E_k^R)&=\lim_{L\to \infty}\frac{\sqrt{\det A}}{(2\pi)^{(n-1)/2}}\int_{\rb E^R}e^{-|\sqrt{A}x|^2/2}\chi_{B_L}(x)d\mathcal{H}^{n-1}=P_{\gamma_A}(E^R).
\end{align}
By Step 1, we can let $\{R_k\}$ be a sequence with $R_k\nearrow\infty$ such that
$$\mbox{$\chi_{E^{R_k}}\to \chi_E$ in $L^1(\R^n,\gamma_A)$}\quad \mbox{and} \quad \lim_{k\to \infty}P_{\gamma_A}(E^{R_k})=P_{\gamma_A}(E).$$
Now we use a diagonal argument with (\ref{approximation_gaussian_perimeter_eq3}), (\ref{approximation_gaussian_perimeter_eq4}) to finish the proof. By a diagonal argument, there exists a sequence $\{E_{N_k}^{R_k}\}_{k=1}^{\infty}$ of bounded open sets with smooth boundary such that
$$\chi_{E^{R_k}_{N_k}}\to \chi_{E}\mbox{ in $L^1(\R^n,\gamma_A)$}\quad \mbox{and}\quad P_{\gamma_A}(E_{N_k}^{R_k})\to P_{\gamma_A}(E).$$
\end{proof}

\section{Anisotropic Gaussian Isoperimetric Inequality}\label{anisotropic Gaussian Isoperimetric Inequality}
 
Define the function $\phi : \R \to (0, 1)$ as
$$\phi(x)=\frac{1}{\sqrt{2\pi}}\int_{-\infty}^xe^{-t^2/2}\ dt,$$
and notice that $\phi$ is strictly increasing from $0$ to $1$. The inverse function $\phi^{-1}:[0,1]\to [-\infty,+\infty]$ is also strictly increasing with $\phi^{-1}(0)=-\infty$ and $\phi^{-1}(1)=+\infty$. Moreover, we have
\begin{align}\label{limit_of_GaussianBall}
\lim_{t \rightarrow \infty} \frac{\phi^{-1}(\gamma(t B(0,1)))}{t}&=1 
\end{align}
(See \cite{Livshyts}, Section 3.3 and \cite{Nayar}, Lemma 9). We define the (nonrenormalized) {\bf $A$-anisotropic Gaussian barycenter} of the set $E$ as
$$b_{\gamma_A}(E):= \int_E x\ d \gamma_A(x).$$
Let $H(\omega,t)$ be the half-space of the form $H(\omega,t)=\left\{x:\langle x,\omega\rangle <t\right\}$, where $\omega\in \SS^{n-1}$, $t\in \R$, and $\langle \cdot,\cdot\rangle $ is the Euclidean inner product. Since $\gamma=\gamma_{I_n}$ is rotation invariant, we can compute the following quantities directly:
\beq\label{quantities for gamma}
\gamma(H(\omega,t))=\phi(t),\qquad P_\gamma(H(\omega,t))=e^{-t^2/2},\qquad b_\gamma(H(\omega,t))=\frac{-1}{\sqrt{2\pi}}e^{-t^2/2}\omega.
\eeq
Moreover, we have the following for half-spaces under $\gamma_A$:

\begin{proposition}\label{mass_perimeter_barycenter}
Let $H(\omega,t)$ be the half-space with $\omega\in \SS^{n-1}$ and $t\in \R$.
\begin{enumerate}
\item If $M$ is an invertible $n\times n$ matrix, then
$$M\left(H(\omega,t)\right)=H\left(\frac{(M^{\mathsf{T}})^{-1}\omega}{|(M^{\mathsf{T}})^{-1}\omega|}, \frac{t}{|(M^{\mathsf{T}})^{-1}\omega|}\right).$$
\item The anisotropic Gaussian mass of the half-space is
$$\gamma_A(H(\omega,t))=\phi\left(\frac{t}{|(\sqrt{A})^{-1}\omega|}\right).$$
\item The anisotropic Gaussian perimeter of the half-space is
$$P_{\gamma_A}(H(\omega,t))=e^{-\frac12\frac{t^2}{|(\sqrt{A})^{-1}\omega|^2}}\frac{1}{|(\sqrt{A})^{-1}\omega|}.$$
Moreover,
\begin{align*}
&P_{\gamma_A}(H_1)=P_{\gamma_A}(H_2)\mbox{\ for any half-spaces $H_1$, $H_2$ with $\gamma_A(H_1)=\gamma_A(H_2)$}\\
&\iff A=aI_n\mbox{ for some constant $a>0$}.
\end{align*}
\item The anisotropic Gaussian barycenter of the half-space is
$$b_{\gamma_A}(H(\omega,t))=\frac{-1}{\sqrt{2\pi}}e^{-\frac12\frac{t^2}{|(\sqrt{A})^{-1}\omega|^2}}\left(\frac{A^{-1}\omega}{|(\sqrt{A})^{-1}\omega|}\right).$$
Moreover,
\begin{align*}
&\left|b_{\gamma_A}(H_1)\right|=\left|b_{\gamma_A}(H_2)\right|\mbox{\ for any half-spaces $H_1$, $H_2$ with $\gamma_A(H_1)=\gamma_A(H_2)$}\\
&\iff A=aI_n\mbox{ for some constant $a>0$}.
\end{align*}
\end{enumerate}
\end{proposition}
\begin{proof} (1) Given any point $y=Mx\in M(H(\omega,t))$, we have $\langle x,\omega\rangle<t.$ Then
$$\left\langle y,\frac{(M^{\mathsf{T}})^{-1}\omega}{|(M^{\mathsf{T}})^{-1}\omega|}\right\rangle =\left\langle Mx,\frac{(M^{\mathsf{T}})^{-1}\omega}{|(M^{\mathsf{T}})^{-1}\omega|}\right\rangle=\left\langle x,\frac{\omega}{|(M^{\mathsf{T}})^{-1}\omega|}\right\rangle<\frac{t}{|(M^{\mathsf{T}})^{-1}\omega|}.$$
Conversely, for any $y\in H\left(\frac{(M^{\mathsf{T}})^{-1}\omega}{|(M^{\mathsf{T}})^{-1}\omega|}, \frac{t}{|(M^{T})^{-1}\omega|}\right)$, let $x=M^{-1}y$. Notice that
$$\left\langle x,\frac{\omega}{|(M^{\mathsf{T}})^{-1}\omega|}\right\rangle=\left\langle M^{-1}y,\frac{\omega}{|(M^{\mathsf{T}})^{-1}\omega|}\right\rangle=\left\langle y,\frac{(M^{\mathsf{T}})^{-1}\omega}{|(M^{\mathsf{T}})^{-1}\omega|}\right\rangle<\frac{t}{|(M^{\mathsf{T}})^{-1}\omega|}.$$
Thus, $\left\langle x,\omega\right\rangle<t$, i.e., $x\in H(\omega,t).$\\

\noindent (2) Observe that $\gamma_A(E)=\gamma(\sqrt{A}(E))$ for any Borel set $E$, and by (1),
$$\sqrt{A}(H(\omega,t))=H\left(\frac{((\sqrt{A})^{\mathsf{T}})^{-1}\omega}{|((\sqrt{A})^{\mathsf{T}})^{-1}\omega|}, \frac{t}{|((\sqrt{A})^{\mathsf{T}})^{-1}\omega|}\right)=H\left(\frac{(\sqrt{A})^{-1}\omega}{|(\sqrt{A})^{-1}\omega|}, \frac{t}{|(\sqrt{A})^{-1}\omega|}\right),$$
since $\sqrt{A}$ is symmetric. Applying equation (\ref{quantities for gamma}), we have
$$\gamma_A(H(\omega,t))=\phi\left(\frac{t}{|(\sqrt{A})^{-1}\omega|}\right).$$
(3) Notice that
\begin{align*}
&P_{\gamma_A}(H(\omega,t))=\frac{\sqrt{\det A}}{(2\pi)^{(n-1)/2}}\int_{\rb H(\omega,t)}e^{-|\sqrt{A}x|^2/2}\ d\mathcal{H}^{n-1}(x)=\frac{\sqrt{\det A}}{(2\pi)^{(n-1)/2}}\int_{\rb H(\omega,t)}\left(e^{-|\sqrt{A}x|^2/2}\omega\right)\cdot \omega\ d\mathcal{H}^{n-1}(x)\\
&=\frac{\sqrt{\det A}}{(2\pi)^{(n-1)/2}}\int_{H(\omega,t)}\div\left(e^{-|\sqrt{A}x|^2/2}\omega\right)\ dx=\frac{\sqrt{\det A}}{(2\pi)^{(n-1)/2}}\int_{H(\omega,t)}-\left\langle Ax,\omega\right\rangle e^{-|\sqrt{A}x|^2/2}\ dx\\
&=\frac{\sqrt{\det A}}{(2\pi)^{(n-1)/2}}\int_{H\left(\frac{(\sqrt{A})^{-1}\omega}{|(\sqrt{A})^{-1}\omega|},\frac{t}{|(\sqrt{A})^{-1}\omega|}\right)}-\left\langle y,\sqrt{A}\omega\right\rangle e^{-|y|^2/2}\ \frac{1}{|\det \sqrt{A}|}dy\\
&=-\sqrt{2\pi}\left\langle\int_{H\left(\frac{(\sqrt{A})^{-1}\omega}{|(\sqrt{A})^{-1}\omega|},\frac{t}{|(\sqrt{A})^{-1}\omega|}\right)} y\ d\gamma(y),\sqrt{A}\omega\right\rangle =-\sqrt{2\pi}\left\langle b_\gamma \left(H\left(\frac{(\sqrt{A})^{-1}\omega}{|(\sqrt{A})^{-1}\omega|},\frac{t}{|(\sqrt{A})^{-1}\omega|}\right)\right),\sqrt{A}\omega\right\rangle \\
&=-\sqrt{2\pi}\left\langle \frac{-1}{\sqrt{2\pi}}e^{-\frac12\frac{t^2}{|(\sqrt{A})^{-1}\omega|^2}}\left(\frac{(\sqrt{A})^{-1}\omega}{|(\sqrt{A})^{-1}\omega|}\right),\sqrt{A}\omega\right\rangle=e^{-\frac12\frac{t^2}{|(\sqrt{A})^{-1}\omega|^2}}\frac{1}{|(\sqrt{A})^{-1}\omega|},
\end{align*}
where we used the change of variables $y=\sqrt{A}x$ and the fact that the outer unit normal of $H(\omega,t)$ is $\nu_{H(\omega,t)}=\omega$. Next, we prove the second part. Let $H_1=H(\omega_1,t_1)$ and $H_2=H(\omega_2,t_2)$. Suppose that $A=aI_n$. Then
$$\gamma_A(H_1)=\gamma_A(H_2)\implies t_1=t_2.$$
Therefore,
$$P_{\gamma_A}(H(\omega_1,t_1))=e^{-at_1^2}\sqrt{a}=e^{-at_2^2}\sqrt{a}=P_{\gamma_A}(H(\omega_2,t_2)).$$
Conversely, for any $H_1=H(\omega_1,t_1)$ and $H_2=H(\omega_2,t_2)$ with $\gamma_A(H_1)=\gamma_A(H_2)$, i.e., $\frac{t_1}{|(\sqrt{A})^{-1}\omega_1|}=\frac{t_2}{|(\sqrt{A})^{-1}\omega_2|}$, we have
$$P_{\gamma_A}(H(\omega_1,t_1))=P_{\gamma_A}(H(\omega_2,t_2))\implies \frac{1}{|(\sqrt{A})^{-1}\omega_1|}=\frac{1}{|(\sqrt{A})^{-1}\omega_2|}.$$
That is, $|(\sqrt{A})^{-1}\omega|$ is a constant for all $\omega\in \SS^{n-1}$. Since $\sqrt{A}$ is orthogonally diagonalizable, say $\sqrt{A}=ODO^{-1}$, where the orthogonal matrix $O=(v_1\ v_2\cdots v_n)$ and $D=\diag(d_1,d_2,\cdots d_n)$, i.e., for $j=1,\cdots, n$, we have $|v_j|=1$ and
$$\sqrt{A}v_j=d_jv_j\quad (\mbox{i.e. $(\sqrt{A})^{-1}v_j=d_j^{-1}v_j$}).$$
Therefore, $d_1=d_2\cdots=d_n>0$ and hence $A=OD^2O^{-1}=O(d_1^2I_n)O^{-1}=d_1^2I_n$.\\

\noindent (4) Notice that
\begin{align*}
b_{\gamma_A}(E)&=\frac{\sqrt{\det A}}{(2\pi)^{n/2}}\int_{E}xe^{-|\sqrt{A}x|^2/2}\ dx=\frac{1}{(2\pi)^{n/2}}\int_{\sqrt{A}E}(\sqrt{A})^{-1}ye^{-|y|^2/2}\ dy=(\sqrt{A})^{-1}b_\gamma(\sqrt{A}E),
\end{align*}
where we used the change of variables $y=\sqrt{A}x$. In particular, using the calculation above and (\ref{quantities for gamma}), we have
\begin{align*}
b_{\gamma_A}(H(\omega,t))=(\sqrt{A})^{-1}b_\gamma(\sqrt{A}H(\omega,t))&=(\sqrt{A})^{-1}\frac{-1}{\sqrt{2\pi}}e^{-\frac12\frac{t^2}{|(\sqrt{A})^{-1}\omega|^2}}\left(\frac{(\sqrt{A})^{-1}\omega}{|(\sqrt{A})^{-1}\omega|}\right)\\
&=\frac{-1}{\sqrt{2\pi}}e^{-\frac12\frac{t^2}{|(\sqrt{A})^{-1}\omega|^2}}\left(\frac{A^{-1}\omega}{|(\sqrt{A})^{-1}\omega|}\right).
\end{align*}
The proof of the second part is the same as (3); hence, we omit the proof.
\end{proof}

\subsection{Two anisotropic Gaussian isoperimetric inequalities}\mbox{}\vspace{-.22cm}\\

We are ready to prove the anisotropic Gaussian isoperimetric inequality ($\ep$-enlargement version). As mentioned in the introduction, Bakry and Ledoux proved a general result about isoperimetric inequality for log-concave measures (see \cite{Ledoux_LogSobolev}, Theorem 1.1 and \cite{Ledoux-Bakry}). However, for completeness, we provide a proof for the case of anisotropic Gaussian measure using an argument inspired by Latała's paper \cite{Latala_Ehrhard}, Section 3. The proof is based on Ehrhard's inequality and the regularity of Radon measures.

\begin{theorem}[Anisotropic Gaussian Isoperimetric Inequality ($\ep$-enlargement version)]\label{Gaussian_Iso_epsilon}\mbox{}
\begin{enumerate}
\item For any measurable set $E$ in $\R^n$, 
$$\phi^{-1}(\gamma_A(E_\ep))\geq \phi^{-1}(\gamma_A(E))+\frac{\ep}{\|(\sqrt{A})^{-1}\|},$$
where $\|\cdot\|$ is the matrix norm induced by the Euclidean norm. The set
$$E_\ep:= E+\ep \bar{B}(0,1)=\{x\in \R^n:\dist(x,E)\leq\ep\}$$
is called the $\ep$-{\bf (Minkowski) enlargement} of $E$. Here $\bar{B}(0,1)$ is the closed unit ball in $\R^n$.
\item Let $E$ be a measurable set in $\mathbb{R}^{n}$ and let $H(\omega,t)$ be a half-space such that 
$$\gamma_A(E) \geq \gamma_A(H(\omega,t)) .$$ 
Then, for every $\ep >0$,
$$
\gamma_{A}\left(E_{\ep}\right) \geq \gamma_{A}\left(H\left(\omega,t+\ep \frac{|(\sqrt{A})^{-1}\omega|}{\|(\sqrt{A})^{-1}\|}\right)\right).
$$
\end{enumerate}
\end{theorem}
\begin{proof} (1) Let $E$ be a Borel set. Applying Ehrhard's inequality (see Theorem \ref{Ehrhard's Inequality}) with $\gamma_A(E)=\gamma(\sqrt{A}(E))$, we have the following: if $B,C$ are Borel sets in $\R^n$, then
$$
\phi^{-1}(\gamma_A(\lambda C+(1-\lambda) B)) \geq \lambda \phi^{-1}(\gamma_A(C))+(1-\lambda) \phi^{-1}(\gamma_A(B)), \text {\qquad for } \lambda \in(0,1),
$$
where
$$\phi(x)=\frac{1}{\sqrt{2\pi}}\int_{-\infty}^xe^{-t^2/2}\ dt.$$
Let  $C=E$ and $B=\bar{B}(0,1)$. Then we have
\begin{align*}
&\phi^{-1}(\gamma_A(E_\ep))=\phi^{-1}(\gamma_A(E+\ep B))=\phi^{-1}\Big[\gamma_A\Big(\lambda \Big(\frac{E}{\lambda}\Big)+(1-\lambda)\Big(\frac{\ep}{1-\lambda} B\Big)\Big)\Big]\\
&\geq \lambda\phi^{-1}\Big[\gamma_A\Big( \frac{E}{\lambda}\Big)\Big]+(1-\lambda)\phi^{-1}\Big[\gamma_A\Big(\frac{\ep}{1-\lambda} B\Big)\Big]= \lambda\phi^{-1}\Big[\gamma_A\Big( \frac{E}{\lambda}\Big)\Big]+(1-\lambda)\phi^{-1}\Big[\gamma\Big(\frac{\ep}{1-\lambda} \sqrt{A}(B)\Big)\Big]\\
&\geq \lambda\phi^{-1}\Big[\gamma_A\Big( \frac{E}{\lambda}\Big)\Big]+(1-\lambda)\phi^{-1}\Big[\gamma\Big(\frac{\ep}{1-\lambda} \frac{1}{\|(\sqrt{A})^{-1}\|}B\Big)\Big],
\end{align*}
using that $\phi^{-1}$ is increasing and $\sqrt{A}(B)\supset \frac{1}{\|(\sqrt{A})^{-1}\|}B$.
Taking $\lambda\to 1^{-}$, by (\ref{limit_of_GaussianBall}), we have
\begin{align*}
(1-\lambda)\phi^{-1}\Big[\gamma\Big(\frac{\ep}{1-\lambda} \frac{1}{\|(\sqrt{A})^{-1}\|}B\Big)\Big]&=\left(\frac{\phi^{-1}\Big[\gamma\Big(\frac{\ep}{1-\lambda} \frac{1}{\|(\sqrt{A})^{-1}\|}B\Big)\Big]}{\frac{\ep}{1-\lambda} \frac{1}{\|(\sqrt{A})^{-1}\|}}\right) \frac{\ep}{\|(\sqrt{A})^{-1}\|}\to  \frac{\ep}{\|(\sqrt{A})^{-1}\|}.
\end{align*}
That is, for any Borel set $E$,
$$\phi^{-1}(\gamma_A(E_\ep))\geq \phi^{-1}(\gamma_A(E))+\frac{\ep}{\|(\sqrt{A})^{-1}\|}.$$
A standard regularity argument ensures the above is true for all Lebesgue measurable sets.\\

\noindent (2) Using (1) with our assumption, we have
\begin{align*}
&\phi^{-1}(\gamma_A(E_\ep))\geq \phi^{-1}(\gamma_A(E))+\frac{\ep}{\|(\sqrt{A})^{-1}\|}\geq \phi^{-1}(\gamma_A(H(\omega,t)))+\frac{\ep}{\|(\sqrt{A})^{-1}\|}=\frac{t}{|(\sqrt{A})^{-1}\omega|}+\frac{\ep}{\|(\sqrt{A})^{-1}\|},
\end{align*}
since $\phi^{-1}$ is increasing. Applying $\phi$ on both sides and using Proposition \ref{mass_perimeter_barycenter},
$$\gamma_A(E_\ep)\geq \phi\left(\frac{t}{|(\sqrt{A})^{-1}\omega|}+\frac{\ep}{\|(\sqrt{A})^{-1}\|}\right)=\phi\left(\frac{t+\ep \frac{|(\sqrt{A})^{-1}\omega|}{\|(\sqrt{A})^{-1}\|}}{|(\sqrt{A})^{-1}\omega|}\right)=\gamma_{A}\left(H\left(\omega,t+\ep \frac{|(\sqrt{A})^{-1}\omega|}{\|(\sqrt{A})^{-1}\|}\right)\right).$$
\end{proof}

Suppose the boundary of $E$ is ``nice'' enough. Intuitively, we have the following
$$\frac{\gamma_A(E_\ep)-\gamma_A(E)}{\ep}\to \frac{1}{\sqrt{2\pi}}P_{\gamma_A}(E),$$
where the extra factor $\frac{1}{\sqrt{2\pi}}$ appears in front of $P_{\gamma_A}$ since we define $P_{\gamma_A}$ with coefficient $\frac{1}{(2\pi)^{(n-1)/2}}$. In order to use this idea, we need to introduce the signed distance function. Let $E$ be a subset of $\mathbb{R}^{n}$. Define $d_{E}: \mathbb{R}^{n} \rightarrow \mathbb{R}$ to be the {\bf signed distance function} from $E$:
$$
d_{E}(x):= \dist(x, E)-\dist\left(x, E^{c}\right)= \begin{cases}-\dist(x, \partial E), & x \in E \\ \dist(x, \partial E), & x \notin E\end{cases}.
$$
Moreover, $d_{E^{c}}(x)=-d_{E}(x)$,
$$d_E(x)=0\iff x\in \partial E,$$
and
$$ \left\{x \in \mathbb{R}^{n}: x \in(\partial E)_{\ep}\right\}=\left\{x \in \mathbb{R}^{n}:\left|d_{E}(x)\right| \leq \ep\right\} ,$$ with $\left|\nabla d_{E}(x)\right|=1$ for any point $x$ where it is differentiable. In particular, $d_E$ is Lipschitz, by Rademacher's theorem, $\nabla d_E$ is differentiable a.e. and hence $|\nabla d_E|=1$ a.e. (see \cite{Ambrosio-Dancer}, Theorem 1 and Remark 3).

\begin{theorem}[Anisotropic Gaussian Isoperimetric Inequality (perimeter version)]\label{Gaussian_IsoIneq_perimeter}\mbox{}\\
Let $E$ be a measurable set in $\mathbb{R}^{n}$. Then
$$P_{\gamma_{A}}\left(E\right) \geq e^{-[\phi^{-1}(\gamma_A(E))]^2/2}\frac{1}{\|(\sqrt{A})^{-1}\|}.$$
In particular, if $A=I_n$, we have the standard Gaussian isoperimetric inequality,
$$P_{\gamma_{I_{n}}}\left(E\right) \geq e^{-[\phi^{-1}(\gamma_{I_n}(E))]^2/2}.$$
\end{theorem}
\begin{proof} {\bf Step 1}: We first assume that $E$ is a bounded open set with smooth boundary in $\R^n$ and $P_{\gamma_A}(E)<\infty$. Since $E$ is a set of finite anisotropic Gaussian perimeter, by Proposition \ref{locally finite anisotropic Gaussian perimeter}, $E$ is a set of locally finite perimeter. Moreover, since $E$ is bounded and open with smooth boundary,
$$d_E\mbox{ is smooth in a tubular neighborhood of $\partial E$ and $\nu_E=\nabla d_E$ on $\partial E$}$$
(see \cite{Ambrosio-Dancer}, Theorem 2). By co-area formula, for any Borel function $g: \mathbb{R}^{n} \rightarrow[0, \infty]$ and Lipschitz function $u:\R^n\to \R$, we have
$$
\int_{\mathbb{R}^{n}} g(x)\left|\nabla u(x)\right| \ d x=\int_{-\infty}^{+\infty}\left(\int_{\left\{u=t\right\}} g(y) \ d \mathcal{H}^{n-1}(y)\right) \ d t 
$$
(see \cite{Ambrosio_Fusco_Pallara}, Remark 2.97). Consider $u=d_E$,  
$$g = \frac{\sqrt{\det A}}{(2\pi)^{n/2}}\ e^{-|\sqrt{A}x|^2/2}\chi_{\{0\leq d_E\leq \ep\}},$$
and define the smooth function $\psi_t$ as
$$\psi_t(x)=x+t\nabla d_E(x)\mbox{\quad in the tubular neighborhood of $\partial E$}.$$
Then we have
\begin{align*}
\frac{\gamma_A(E_{\ep})-\gamma_A(E)}{\ep}
&=\frac{1}{\ep}\int_{E_{\ep}\slash E}\frac{\sqrt{\det A}}{(2\pi)^{n/2}} e^{-|\sqrt{A}x|^2/2}\ dx\\
&=\frac{1}{\ep}\int_0^{\ep}\int_{\{d_E=t\}}\frac{\sqrt{\det A}}{(2\pi)^{n/2}} e^{-|\sqrt{A}x|^2/2}\ d\mathcal{H}^{n-1}(x)\ dt\\
&=\frac{\sqrt{\det A}}{(2\pi)^{n/2}} \frac{1}{\ep}\int_0^{\ep}\int_{\psi_t(\partial E)}e^{-|\sqrt{A}x|^2/2}\ d\mathcal{H}^{n-1}(x)\ dt.
\end{align*}
Taking $\ep\to 0^+$ on both sides, by the fundamental theorem of calculus,
$$\lim_{\ep\to 0^+}\frac{\gamma_A(E_{\ep})-\gamma_A(E)}{\ep}=\frac{\sqrt{\det A}}{(2\pi)^{n/2}} \int_{\partial E}e^{-|\sqrt{A}x|^2/2}\ d\mathcal{H}^{n-1}(x)=\frac{1}{\sqrt{2\pi}}P_{\gamma_A}(E),$$
where $\psi_0(\partial E)=\partial E$. On the other hand, by Theorem \ref{Gaussian_Iso_epsilon}, we have
\begin{align*}
\frac{\gamma_A(E_{\ep})-\gamma_A(E)}{\ep}&\geq \frac{\phi\left(\phi^{-1}(\gamma_A(E))+\frac{\ep}{\|(\sqrt{A})^{-1}\|}\right)-\gamma_A(E)}{\ep}\\
&\to \phi'\left(\phi^{-1}(\gamma_A(E))\right)\frac{1}{\|(\sqrt{A})^{-1}\|}=\frac{1}{\sqrt{2\pi}}e^{-[\phi^{-1}(\gamma_A(E))]^2/2}\frac{1}{\|(\sqrt{A})^{-1}\|},
\end{align*}
as $\ep\to 0^+$. That is,
$$P_{\gamma_A}(E)\geq e^{-[\phi^{-1}(\gamma_A(E))]^2/2}\frac{1}{\|(\sqrt{A})^{-1}\|}.$$
{\bf Step 2}: Now for any measurable set $E$, we may again assume that $P_{\gamma_A}(E)<\infty$. By Proposition \ref{approximation_gaussian_perimeter}, there exists a sequence $\{E_k\}$ of bounded open sets with smooth boundary such that
$$\chi_{E_k}\to \chi_{E}\mbox{ in $L^1(\R^n,\gamma_A)$},\quad P_{\gamma_A}(E_k)\to P_{\gamma_A}(E).$$
Applying (1) on $E_k$, we have
$$P_{\gamma_A}(E_k)\geq e^{-[\phi^{-1}(\gamma_A(E_k))]^2/2}\frac{1}{\|(\sqrt{A})^{-1}\|}.$$
Taking $k\to \infty$, we have finished the proof.
\end{proof}

In the paper of Cianchi-Fusco-Maggi-Pratelli \cite{Cianchi-Fusco-Maggi-Pratelli}, Proposition 3.1 and Theorem 4.1, they have characterized the equality cases for the standard Gaussian measure $\gamma_{I_n}$. The result reads as follows: let $E$ be a measurable subset of $\mathbb{R}^n$. Then
\begin{align}\label{case_of_equality_1}
P_{\gamma_{I_n}}(E) \geq e^{-[\phi^{-1}(\gamma_{I_n}(E))]^2/2}.
\end{align}
Moreover, 
\begin{enumerate}
\item if $n=1$, equality holds if and only if either $\gamma_{1}(E)=0$ or $\gamma_{1}(E)=1$, or up to a set of measure zero and for some $\sigma \in \mathbb{R}, E=(-\infty,-\sigma)$ or $E=(\sigma, \infty)$.
\item if $n\geq 2$, equality holds if and only if either $\gamma_{I_n}(E)=0$ or $\gamma_{I_n}(E)=1$, or $E$ is equivalent to a half-space.
\end{enumerate}
Notice that we can also derive Theorem \ref{Gaussian_IsoIneq_perimeter} from Proposition \ref{Gaussian_perimeter} and equation (\ref{case_of_equality_1}),
\begin{align}\label{case_of_equality_2}
P_{\gamma_{A}}(E)\geq P_{\gamma_{I_n}}(\sqrt{A}E)\frac{1}{\|(\sqrt{A})^{-1}\|}\geq e^{-[\phi^{-1}(\gamma_{I_n}(\sqrt{A}E))]^2/2}\frac{1}{\|(\sqrt{A})^{-1}\|}=e^{-[\phi^{-1}(\gamma_A(E))]^2/2}\frac{1}{\|(\sqrt{A})^{-1}\|}.\\
\notag
\end{align}

\subsection{Proof of Theorem \ref{AnisotropicGaussainIso} (cases of equality)} Notice that (1) follows directly from Cianchi-Fusco-Maggi-Pratelli \cite{Cianchi-Fusco-Maggi-Pratelli}, Proposition 3.1. We just need to prove (2) here. Suppose the equality holds and assume that $\gamma_A(E)=\gamma_{I_n}(\sqrt{A}E)\in (0,1)$. By equation (\ref{case_of_equality_2}), we have
$$P_{\gamma_{A}}(E)= P_{\gamma_{I_n}}(\sqrt{A}E)\frac{1}{\|(\sqrt{A})^{-1}\|}= e^{-[\phi^{-1}(\gamma_{I_n}(\sqrt{A}E))]^2/2}\frac{1}{\|(\sqrt{A})^{-1}\|}=e^{-[\phi^{-1}(\gamma_A(E))]^2/2}\frac{1}{\|(\sqrt{A})^{-1}\|}.$$
That is,
$$
P_{\gamma_{I_n}}(\sqrt{A}E) =e^{-[\phi^{-1}(\gamma_{I_n}(\sqrt{A}E))]^2/2}.
$$
By \cite{Cianchi-Fusco-Maggi-Pratelli}, Theorem 4.1, $\sqrt{A}E$ is equivalent to a half-space, say $H(\omega,t)$, where $\omega\in \SS^{n-1}$. Then
$$E\mbox{ is equivalent to }(\sqrt{A})^{-1}(H(\omega,t))=H\left(\frac{\sqrt{A}\omega}{|\sqrt{A}\omega|}, \frac{t}{|\sqrt{A}\omega|}\right).$$
Moreover, by Proposition \ref{mass_perimeter_barycenter},
$$\gamma_A(E)=\gamma_A\left(H\left(\frac{\sqrt{A}\omega}{|\sqrt{A}\omega|}, \frac{t}{|\sqrt{A}\omega|}\right)\right)=\phi\left(t\right)\implies t=\phi^{-1}(\gamma_A(E))$$
and
$$P_{\gamma_A}(E)=P_{\gamma_A}\left(H\left(\frac{\sqrt{A}\omega}{|\sqrt{A}\omega|}, \frac{t}{|\sqrt{A}\omega|}\right)\right)=e^{-\frac12 t^2}|\sqrt{A}\omega|=e^{-[\phi^{-1}(\gamma_A(E))]^2/2}|\sqrt{A}\omega|.$$
By our assumption, we have
\begin{align}\label{AnisotropicGaussainIso_eq1}
e^{-[\phi^{-1}(\gamma_A(E))]^2/2}\frac{1}{\|(\sqrt{A})^{-1}\|}=P_{\gamma_A}(E)=e^{-[\phi^{-1}(\gamma_A(E))]^2/2}|\sqrt{A}\omega|\implies |\sqrt{A}\omega|=\frac{1}{\|(\sqrt{A})^{-1}\|}=d_{\min},
\end{align}
where $d_{\min}$ is the smallest eigenvalue of $\sqrt{A}$ and we have used
$$\|(\sqrt{A})^{-1}\|=\mbox{ the largest eigenvalue of $(\sqrt{A})^{-1}$}=\frac{1}{d_{\min}}.$$
Now we claim that 
$$\omega\in V_{d_{\min}}(\sqrt{A})\cap \SS^{n-1}.$$
Notice that we can decompose $A$ into
$$A=O^{\mathsf{T}}D O,$$
with an orthogonal matrix $O$ and a diagonal matrix $D$. If all eignevalues of $D$ are the same, i.e., $D=d^2_{\min}I_n$, then $\sqrt{A}=d_{\min}I_n$ and $\omega\in V_{d_{\min}}(\sqrt{A})\cap \SS^{n-1}$. Hence, we may assume that $D$ has the following form:
$$D=\left[\begin{array}{cc}   D_1&0\\  0& D_2\end{array}\right],\quad D_1=\lambda_{\min}I_{1},\mbox{\quad and \quad $D_2$ has eigenvalues strictly greater than $\lambda_{\min}$},$$
where $\lambda_{\min}=d_{\min}^2$ is the smallest eigenvalue of $A$. Let
$$O=\left[\begin{array}{cc}   O_1&O_2\\  O_3& O_4\end{array}\right],$$
then $\sqrt{A}=O^{\mathsf{T}}D^{1/2} O$ and
\begin{align}\label{AnisotropicGaussainIso_eq2}
|\sqrt{A}\omega|^2=|O^{\mathsf{T}}D^{1/2}O\omega|^2=|D^{1/2}y|^2=|D^{1/2}_1y_1|^2+|D_2^{1/2}y_2|^2=d^2_{\min}|y_1|^2+|D_2^{1/2}y_2|^2
\end{align}
where
$$\left[\begin{array}{c}   y_1\\   y_2\end{array}\right]=y:=O\omega.$$
On the other hand,
\begin{align}\label{AnisotropicGaussainIso_eq3}
d_{\min}^2=d_{\min}^2|\omega|^2=d_{\min}^2|y|^2=d^2_{\min}|y_1|^2+d^2_{\min}|y_2|^2.
\end{align}
Therefore, by (\ref{AnisotropicGaussainIso_eq1}), (\ref{AnisotropicGaussainIso_eq2}), and (\ref{AnisotropicGaussainIso_eq3}),
$$|D_2^{1/2}y_2|^2=d^2_{\min}|y_2|^2\implies y_2=0$$
since $D_2$ has eigenvalues strictly greater than $\lambda_{\min}=d^2_{\min}$. Thus,
\begin{align*}
\sqrt{A}\omega=O^{\mathsf{T}}D^{1/2}O\omega&=O^{\mathsf{T}}D^{1/2}y=\left[\begin{array}{cc}   O^{\mathsf{T}}_1&O^{\mathsf{T}}_3\\  O^{\mathsf{T}}_2& O^{\mathsf{T}}_4\end{array}\right]\left[\begin{array}{cc}   d_{\min}I_1&0\\ 0& D_2^{1/2}\end{array}\right]\left[\begin{array}{c}   y_1\\   0\end{array}\right]=\left[\begin{array}{c}   d_{\min}O^{\mathsf{T}}_1y_1\\   d_{\min}O^{\mathsf{T}}_2y_1\end{array}\right],
\end{align*}
and
$$d_{\min}\omega=d_{\min}O^{\mathsf{T}}y=d_{\min}\left[\begin{array}{cc}   O^{\mathsf{T}}_1&O^{\mathsf{T}}_3\\  O^{\mathsf{T}}_2& O^{\mathsf{T}}_4\end{array}\right]\left[\begin{array}{c}   y_1\\   0\end{array}\right]=\left[\begin{array}{c}   d_{\min}O^{\mathsf{T}}_1y_1\\   d_{\min}O^{\mathsf{T}}_2y_1\end{array}\right]=\sqrt{A}\omega.$$
Hence,
$$\sqrt{A}\omega-d_{\min}\omega=0\implies \omega\in V_{d_{\min}}(\sqrt{A}).$$
We conclude that
$$E\mbox{ is equivalent to }(\sqrt{A})^{-1}(H(\omega,t))=H\left(\frac{\sqrt{A}\omega}{|\sqrt{A}\omega|}, \frac{t}{|\sqrt{A}\omega|}\right)=H\left(\omega,\frac{\phi^{-1}(\gamma_A(E))}{d_{\min}}\right).$$
Now we prove the converse of (2) in Theorem \ref{AnisotropicGaussainIso}. It is clear that the equality holds when $\gamma_A(E)=0$ or $\gamma_A(E)=1$. Hence, we may assume that $\gamma_A(E)\in (0,1)$, i.e., $\phi^{-1}(\gamma_A(E))\in \R$. Since $\omega\in V_{d_{\min}}(\sqrt{A})\cap \SS^{n-1}$, $\sqrt{A}\omega=d_{\min}\omega$. By Proposition \ref{mass_perimeter_barycenter}, we have
\begin{align*}
P_{\gamma_A}(E)=P_{\gamma_A}\left(H\left(\omega,\frac{\phi^{-1}(\gamma_A(E))}{d_{\min}}\right)\right)&=e^{-\frac12\frac{\frac{[\phi^{-1}(\gamma_A(E))]^2}{d^2_{\min}}}{|(\sqrt{A})^{-1}\omega|^2}}\frac{1}{|(\sqrt{A})^{-1}\omega|}\\
&=e^{-[\phi^{-1}(\gamma_A(E))]^2/2}d_{\min}=e^{-[\phi^{-1}(\gamma_A(E))]^2/2}\frac{1}{\|(\sqrt{A})^{-1}\|}.
\end{align*}
\qed

\section{Anisotropic Gaussian Perimeter Inequality under Ehrhard Symmetrization}\label{Anisotropic Gaussian Inequality under Ehrhard Symmetrization}

\subsection{Ehrhard symmetrization}\label{Ehrhard_symmetrization_notation_I}\mbox{}\vspace{-.22cm}\\

In this section, we will use the following notations:
$$\mbox{$x=(z, y)$ for $x \in \mathbb{R}^{n}, z \in \mathbb{R}^{n-1}$ and $y \in \mathbb{R} .$}$$ 
Similar to (\ref{Hausdorff_measure}), we define two (outer) measures $\mu_z$ and $\mathcal{H}^{0}_{z}$ on $\R^1$ such that
\begin{align}\label{def_with_z_section}
\mu_z(F_1)=\int_{F_1}e^{-|\sqrt{A}x|^2/2}dy,\mbox{\quad$\forall F_1\in \mathcal{L}(\R^1)$,}\quad \mathcal{H}^{0}_{z}(F_2)=\int_{F_2}e^{-|\sqrt{A}x|^2/2}d\mathcal{H}^{0}(y),\quad \mbox{$\forall F_2\subset \R^1$,}
\end{align}
where $\mathcal{H}^{0}$ is the counting measure. Moreover, we define
$$P_z(F)=\mathcal{H}^{0}_{z}(\partial^M F),\quad F\subset \R^1,$$
where $\partial^M F$ is the essential boundary of $F$. We also define an auxiliary function $\phi_{z}$ as
\begin{align*}
\phi_{z}(t)&=\int_{-\infty}^{t}e^{-|\sqrt{A}x|^2/2}  dy, \quad z\in\R^{n-1}.
\end{align*}
Let $E$ be a measurable set in $\mathbb{R}^{n}$ with $n \geq 2$. The {\bf section} $E_{z} \subseteq \mathbb{R}$ of $E$ is defined as
$$
E_{z}=\{y \in \mathbb{R}:(z, y) \in E\}, \quad \mbox{where $z \in \mathbb{R}^{n-1}$.}
$$
Define $v_{E}: \mathbb{R}^{n-1} \rightarrow\R$ as
$$
v_{E}(z)=\mu_z\left(E_{z}\right),\mbox{\quad $\forall z\in \R^{n-1}$.}
$$
Notice that $e^{-|\sqrt{A}x|^2/2}\leq e^{-\|(\sqrt{A})^{-1}\|^{-2}|x|^2/2}$ (see Lemma \ref{computational_lemma}) and $x\mapsto e^{-\|(\sqrt{A})^{-1}\|^{-2}|x|^2/2}\in L^1(\R^n)$. By Fubini theorem, we have
$$v_E\in L^1(\R^{n-1}).$$
The {\bf Ehrhard symmetrization} of $E$ with respect to the $y$-direction is defined as
\begin{align}\label{old_def}
E^s:= E^{s}_{A,-e_n}:=\left\{(z, y) \in \mathbb{R}^{n}: y<\phi^{-1}_{z}\left(v_{E}(z)\right)\right\},
\end{align}
and the {\bf essential projection} of $E$ with respect to the $y$-direction is defined as
$$
\pi_{+}(E):=\pi_{+,A,-e_n}(E):=\left\{z \in \mathbb{R}^{n-1}: v_E(z)=\mu_{z}\left(E_{z}\right)>0\right\} .
$$
We now define
$$p_E(z)=\mathcal{H}_z^0\left[(\partial^M E)_z\right].$$
Roughly speaking, the set $\pi_+(E)$ captures the set in $\R^{n-1}$ over which the one-dimensional vertical slices in $E$ have positive mass. We recall the co-area formula for sets of finite perimeter (see \cite{Cianchi-Fusco-Maggi-Pratelli}, equation (4.1)): for any non-negative Borel function $g: \mathbb{R}^{n} \rightarrow[0, \infty]$, we have
\beq\label{co-area formula for sets of finite perimeter_1}
\int_{\partial^{M} E} g(x)\left|\nu_{n}^{E}(x)\right| d \mathcal{H}^{n-1}(x)=\int_{\mathbb{R}^{n-1}} \int_{\left(\partial^{M} E\right)_{z}} g(z, y)\ d \mathcal{H}^{0}(y)\ dz,
\eeq
where $\nu_n^E$ means $\langle \nu^E,e_n\rangle$. We also recall the following theorem by Vol'port  from \cite{Vol'pert} (see also \cite{Ambrosio_Fusco_Pallara}, Theorem 3.108 and \cite{Chlebik-Cianchi-Fusco}, Theorem G):

\begin{theorem}[Vol'pert Theorem]\label{Vol'pert}\mbox{}\\
 Let $E \subseteq \mathbb{R}^{n}$ be a set of locally finite perimeter with $n \geq 2$. Then there exists a Borel set $B_{E} \subseteq \pi_{+}(E)$ with $\mathcal{L}^{n-1}\left(\pi_{+}(E) \slash B_{E}\right)=0$ such that for every $z \in B_{E}$,
\begin{enumerate}
\item[(i)] $E_{z}$ is a set of locally finite perimeter in $\mathbb{R}$;
\item[(ii)] $\left(\partial^{M} E\right)_{z}=\partial^{M}\left(E_{z}\right)=\partial^{*}\left(E_{z}\right)=\left(\partial^{*} E\right)_{z}$;
\item[(iii)] $\nu_{n}^{E}(z, y) \neq 0$ for every $y$ such that $(z, y) \in \partial^{*} E$.
\end{enumerate}
We will call $B_E$ the {\bf Vol'pert set}.
\end{theorem}

In order to understand the Ehrhard symmetrization set $E^s$, our first goal is to analyze the regularities of the mappings $z\mapsto v_E(z)$ and $z\mapsto \phi_z^{-1}(v_E(z))$. For the isotropic Gaussian case, the mapping $z\mapsto \phi^{-1}(\gamma_1(E_z))$ is in $BV_{loc}(\R^{n-1})$ since $z\mapsto \gamma_1(E_z)$ is in $BV(\R^{n-1})$ and $\omega\mapsto \phi^{-1}(\omega)$ is $C^1(\R)$. Here we have used a fact proven by Vol'pert \cite{Vol'pert} that the composition of a $C^1$ map with a BV function is again a BV function. In fact, Ambrosio-Dal Maso \cite{Ambrosio_DalMaso} showed that this is also true if we compose a BV function with a Lipschitz map. However, in our setting, the function $\phi_z^{-1}$ is also depending on the variable $z\in \R^{n-1}$ which required a different proof for the regularity of $z\mapsto \phi_z^{-1}(v_E(z))$.

\subsection{A regularity result for the map $z\mapsto \phi_z^{-1}(v_E(z))$}\mbox{}\vspace{-.22cm}\\

 Our first goal is to show that $v_E\in BV(\R^{n-1})$ if $E$ is a set of finite anisotropic Gaussian perimeter. The proof is similar to Chlebík-Cianchi-Fusco's paper \cite{Chlebik-Cianchi-Fusco}, Lemma 3.1 and Lemma 3.2. Before doing that, we need the following preliminary result for the integrand $e^{-|\sqrt{A}|^2/2}$. The cross term $a_{ij}x_ix_j$ in $\langle Ax,x\rangle=|\sqrt{A}x|^2$ also plays an important role in the calculation. We will need those estimates throughout this section.

\begin{lemma}[Computational lemma]\label{computational_lemma}\mbox{}\\ 
Let $n\geq 2$ and let $\sqrt{A}$ be a symmetric positive definite matrix. Then
\begin{enumerate}
\item {\rm \bf (Derivative for the integrand)}\mbox{}\\
Let $\nabla'=(\partial_1,\ldots, \partial_{n-1})$ and $x=(z,y)$. Then 
$$\partial_{z_k}e^{-|\sqrt{A}x|^2/2}=-e^{-|\sqrt{A}x|^2/2}\langle \row_k(A),x\rangle,\quad \partial^2_{yy}|\sqrt{A}x|^2=2\sum_{i=1}^na_{in}^2,$$
$$\nabla' \left(e^{-|\sqrt{A}x|^2/2}\right)=-e^{-|\sqrt{A}x|^2/2}A'x,\quad\mbox{and}\quad \nabla \left(e^{-|\sqrt{A}x|^2/2}\right)=-e^{-|\sqrt{A}x|^2/2}Ax,$$
where $\sqrt{A}=(a_{ij})$ and $A'\in M_{(n-1)\times n}(\R)$ is the first $n-1$ rows of matrix from $A$.
\item {\rm \bf (Regularity estimates)}
\begin{enumerate}
\item For any $z_0\in \R^{n-1}$ and for any measurable set $F$,
$$\lim_{z\to z_0}\int_F \left(e^{-|\sqrt{A}(z,y)|^2/2}-e^{-|\sqrt{A}(z_0,y)|^2/2}\right)dy=0.$$
In particular, the mapping
$$v:z\mapsto \int_{0}^{\infty}e^{-|\sqrt{A}(z,y)|^2/2}\ dy\mbox{\ is continuous.}$$
\item For any $z\in \R^{n-1}$ and for any measurable set $F$,
\begin{align}\label{computational_lemma_(2)(b)}
\lim_{k\to 0}\int_{F}\Bigg(\frac{e^{-|\sqrt{A}(z+k,y)|^2/2}-e^{-|\sqrt{A}(z,y)|^2/2}-\nabla'\left(e^{-|\sqrt{A}(z,y)|^2/2}\right)\cdot k}{|k|}\Bigg)dy=0.
\end{align}
In particular, the mapping $v$ in (a) is differentiable and
$$\nabla'v(z)=\int_{0}^{\infty}\nabla'\left(e^{-|\sqrt{A}(z,y)|^2/2}\right) dy.$$
\item Let $K$ be a convex compact set in $\R^{n-1}$ and $h\in C^1(K)$. Then for any $z_0,z\in K$,
\begin{align*}
&\left|e^{|\sqrt{A}(z,h(z))|^2/2}-e^{|\sqrt{A}(z_0,h(z_0))|^2/2}\right|\leq C(K,h,A)\left|z-z_0\right|
\end{align*}
for some constant $C(K,h,A)>0$.
\end{enumerate}
\item {\rm \bf (Integral bounds)}
\begin{enumerate}
\item There exists a constant $C_1(A)>0$ such that
$$\sup_{1\leq k\leq n-1}\sup_{z\in \R^{n-1}}\left(\int_{-\infty}^{\infty}\left|\frac{\partial}{\partial z_k}\left(e^{-|\sqrt{A}x|^2/2}\right)\right| dy\right)\leq C_1(A)<\infty.$$
\item There exists a constant $C_2(A)>0$ such that
$$\sup_{1\leq k\leq n-1}\int_{\R^n}\left|\frac{\partial}{\partial z_k}\left(e^{-|\sqrt{A}x|^2/2}\right)\right| dx\leq C_2(A)<\infty.$$
\end{enumerate}
\end{enumerate}
\end{lemma}
\begin{proof} 
We will denote $x=(z,y)$ in the following calculations. When we do matrix multiplication, the notation $Ax=A(z,y)$ means
$$A(z,y)^{\mathsf{T}}\in M_{n\times 1}(\R).$$
\begin{enumerate}
\item {\rm \bf (Derivative for the integrand)}\\
Let $A=(A_{ij})$, $\sqrt{A}=(a_{ij})$. Then
\begin{align*}
|\sqrt{A}x|^2=\sum_{i=1}^n\left(\sum_{j=1}^na_{ij}x_j\right)^2=\sum_{i=1}^n\left(\sum_{j=1}^{n-1}a_{ij}z_j+a_{in}y\right)^2,
\end{align*}
\begin{align*}
\partial^2_{yy}|\sqrt{A}x|^2&=\partial_{y}\left(\sum_{i=1}^n2\left(\sum_{j=1}^{n-1}a_{ij}z_j+a_{in}y\right)a_{in}\right)=2\sum_{i=1}^na_{in}^2,
\end{align*}
and
\begin{align*}
\partial_{z_k}|\sqrt{A}x|^2&=\sum_{i=1}^n\sum_{j=1}^{n-1}2a_{ik}a_{ij}z_j+\sum_{i=1}^n2a_{ik}a_{in}y
\end{align*}
for $k=1,2,\cdots, n-1$. Since $\sqrt{A}$ is symmetric, i.e., $a_{ik}=a_{ki}$, we have
\begin{align*}
\partial_{z_k}\left(e^{-|\sqrt{A}x|^2/2}\right)
&=-e^{-|\sqrt{A}x|^2/2}\left(\sum_{i=1}^n\sum_{j=1}^{n-1}a_{ik}a_{ij}z_j+\sum_{i=1}^na_{ik}a_{in}y\right)\\
&=-e^{-|\sqrt{A}x|^2/2}\left(\sum_{j=1}^{n-1}A_{kj}z_j+A_{kn}y\right)=-e^{-|\sqrt{A}x|^2/2}\langle \row_k(A),x\rangle
\end{align*}
Therefore,
\begin{align*}
\nabla' \left(e^{-|\sqrt{A}x|^2/2}\right)&=-e^{-|\sqrt{A}x|^2/2}\begin{pmatrix}
\langle \row_1(A),x\rangle\\
\vdots\\
\langle \row_{n-1}(A),x\rangle
\end{pmatrix}=-e^{-|\sqrt{A}x|^2/2}A'x
\end{align*}
and $\nabla \left(e^{-|\sqrt{A}x|^2/2}\right)=-e^{-|\sqrt{A}x|^2/2}Ax$.
\item {\rm \bf (Regularity estimates)}
\begin{enumerate}
\item Let $K$ be a compact set with $z_0,z\in K$. By using (1) and the mean value theorem,
\begin{align}\label{exp_estimate}
&\left|e^{-|\sqrt{A}(z,y)|^2/2}-e^{-|\sqrt{A}(z_0,y)|^2/2}\right|\notag\leq |z-z_0|\left|e^{-|\sqrt{A}(\zeta,y)|^2/2}A'(\zeta,y)\right|\notag\\
&\leq \left|A'(\zeta,y)\right|e^{-\|(\sqrt{A})^{-1}\|^{-2}|y|^2/2} |z-z_0|\notag\\
&\leq \sqrt{\lambda_{\max}(A^{\prime\mathsf{T}}A^{\prime})}\Big(|\zeta|+|y|\Big)e^{-\|(\sqrt{A})^{-1}\|^{-2}|y|^2/2} |z-z_0|\notag\\
&\leq \sqrt{\lambda_{\max}(A^{\prime\mathsf{T}}A^{\prime})}\Big(r(K)+|y|\Big)e^{-\|(\sqrt{A})^{-1}\|^{-2}|y|^2/2} |z-z_0|
\end{align}
where $\zeta$ lies between $z$ and $z_0$, $\lambda_{\max}(A^{\prime\mathsf{T}}A^{\prime})$ is the largest eigenvalue of $A^{\prime\mathsf{T}}A^{\prime}$, and $r(K)=\sup\limits_{\zeta\in K}|\zeta|$. We now claim that
$$\lim_{z\to z_0}\int_{F}e^{-|\sqrt{A}(z,y)|^2/2}-e^{-|\sqrt{A}(z_0,y)|^2/2}dy=0.$$
Since $z\to z_0$, we may assume that $z\in K:=\overline{B}(z_0,1)$. Thus, as $z\to z_0$,
\begin{align*}
&\int_{F}\left|e^{-|\sqrt{A}(z,y)|^2/2}-e^{-|\sqrt{A}(z_0,y)|^2/2}\right|dy\\
&\leq \int_F \sqrt{\lambda_{\max}(A^{\prime\mathsf{T}}A^{\prime})}\Big(r(K)+|y|\Big)e^{-\|(\sqrt{A})^{-1}\|^{-2}|y|^2/2} |z-z_0|\ dy\\
&=|z-z_0|\sqrt{\lambda_{\max}(A^{\prime\mathsf{T}}A^{\prime})}\int_F\Big(r(K)+|y|\Big)e^{-\|(\sqrt{A})^{-1}\|^{-2}|y|^2/2}dy\to 0.
\end{align*}
\item By Taylor expansion, if $f\in C^1(\R^{n-1})$ and $x_1,x_2\in \R^{n-1}$,
\begin{align*}
f(x_2)&=f(x_1)+\langle \nabla' f(x_1),x_2-x_1\rangle\\
&\quad+\int_0^1\langle\nabla' f(x_1+t(x_2-x_1))-\nabla' f(x_1), x_2-x_1\rangle\ d t
\end{align*}
(see, for example, \cite{Dmitriy}, Theorem 1.14). Fix $y\in \R$, $k\in \R^{n-1}$, and set
$$f(z)=e^{-|\sqrt{A}(z,y)|^2/2}.$$
Let $K$ be a convex compact set with $z,z+k\in K$. Then by a similar argument as $(\ref{exp_estimate})$,
\begin{align*}
&\left|e^{-|\sqrt{A}(z+k,y)|^2/2}-e^{-|\sqrt{A}(z,y)|^2/2}-\nabla'\left(e^{-|\sqrt{A}(z,y)|^2/2}\right)\cdot k\right|\\
&=\left|\int_0^1\left(-e^{-|\sqrt{A}(z+tk,y)|^2/2}A'(z+tk,y)+e^{-|\sqrt{A}(z,y)|^2/2}A'(z,y)\right)\cdot k\ d t \right|\\
&\leq \int_0^1\left|\left(-e^{-|\sqrt{A}(z+tk,y)|^2/2}A'(z+tk,y)+e^{-|\sqrt{A}(z+tk,y)|^2/2}A'(z,y)\right)\cdot k\right| d t\\
&\quad+\int_0^1\left|\left(-e^{-|\sqrt{A}(z+tk,y)|^2/2}A'(z,y)+e^{-|\sqrt{A}(z,y)|^2/2}A'(z,y)\right)\cdot k\right| d t \\
&\leq\frac12|k|^2e^{-\|(\sqrt{A})^{-1}\|^{-2}|y|^2/2}\sqrt{\lambda_{\max}(A^{\prime\mathsf{T}}A^{\prime})}\\
&\quad +\frac12|k|^2e^{-\|(\sqrt{A})^{-1}\|^{-2}|y|^2/2}\lambda_{\max}(A^{\prime\mathsf{T}}A^{\prime})\Big(r(K)+|y|\Big)\left(|z|+|y|\right)
\end{align*}
where $z+tk\in K$ since $K$ is convex. In particular, for any $z\in \R^{n-1}$ and for any measurable set $F$, we have the following
$$\lim_{k\to 0}\int_{F}\Bigg(\frac{e^{-|\sqrt{A}(z+k,y)|^2/2}-e^{-|\sqrt{A}(z,y)|^2/2}-\nabla'\left(e^{-|\sqrt{A}(z,y)|^2/2}\right)\cdot k}{|k|}\Bigg)dy=0.$$
\item Recall that
$$\left|e^x-1\right| \leq e^{|x|}-1 \leq|x| e^{|x|}\mbox{\quad for all $x\in \mathbb{R}$}.$$
Then by $h\in C^1(K)$, the convexity of $K$, and a similar argument as $(\ref{exp_estimate})$,
\begin{align*}
&\left|e^{|\sqrt{A}(z,h(z))|^2/2}-e^{|\sqrt{A}(z_0,h(z_0))|^2/2}\right|\\
&= e^{\frac{|\sqrt{A}(z_0,h(z_0))|^2}{2}}\left|e^{\frac{|\sqrt{A}(z,h(z))|^2-|\sqrt{A}(z_0,h(z_0))|^2}{2}}-1\right|\\
&\leq e^{\frac{|\sqrt{A}(z_0,h(z_0))|^2}{2}}\left|\frac{|\sqrt{A}(z,h(z))|^2-|\sqrt{A}(z_0,h(z_0))|^2}{2}\right|e^{\left|\frac{|\sqrt{A}(z,h(z))|^2-|\sqrt{A}(z_0,h(z_0))|^2}{2}\right|}\\
&\leq C(K,h,A)|z-z_0|
\end{align*}
for some constant $C(K,h,A)>0$.
\end{enumerate}
\item {\rm \bf (Integral bounds)}\\
We will only prove (3)(a) since the proof of (3)(b) is similar. By (1), we have
\begin{align*}
\left|\frac{\partial}{\partial z_k}\left(e^{-|\sqrt{A}x|^2/2}\right)\right|&=\left|e^{-|\sqrt{A}x|^2/2}\langle \row_k(A),x\rangle \right|=\left|e^{-|\sqrt{A}x|^2/2}\langle A^\mathsf{T}e_k,x\rangle \right|\\
&\leq e^{-|\sqrt{A}x|^2/2}\|A\||x|\leq  \|A\|e^{-\|(\sqrt{A})^{-1}\|^{-2}|x|^2/2}|x|\\
&\leq  \|A\|e^{-\|(\sqrt{A})^{-1}\|^{-2}|z|^2/2}e^{-\|(\sqrt{A})^{-1}\|^{-2}|y|^2/2}\left(|z|+|y|\right).
\end{align*}
Hence,
\begin{align*}
\int_{-\infty}^{\infty}\left|\frac{\partial}{\partial z_k}\left(e^{-|\sqrt{A}x|^2/2}\right)\right| dy&\leq \|A\|e^{-\|(\sqrt{A})^{-1}\|^{-2}|z|^2/2}|z|\int_{-\infty}^{\infty}e^{-\|(\sqrt{A})^{-1}\|^{-2}|y|^2/2}dy\\
&\quad+\|A\|e^{-\|(\sqrt{A})^{-1}\|^{-2}|z|^2/2}\int_{-\infty}^{\infty}e^{-\|(\sqrt{A})^{-1}\|^{-2}|y|^2/2}|y|dy.
\end{align*}
Therefore, there exists a constant $C_1(A)>0$ such that
$$\sup_{1\leq k\leq n-1}\sup_{z\in \R^{n-1}}\left(\int_{-\infty}^{\infty}\left|\frac{\partial}{\partial z_k}\left(e^{-|\sqrt{A}x|^2/2}\right)\right| dy\right)\leq C_1(A)<\infty.$$
\end{enumerate}
\end{proof}

We are now ready to show that $v_E$ is in $BV(\R^{n-1})$. Moreover, we prove a relation between $|Dv_E|$ and $P_{\gamma_A}(E;G\times\R)$ and a weak derivative formula for $D_iv_E$. These two ingredients play an important role in proving our main result (see Theorem \ref{Ehrhard_Sym_Ineq_I}).

\begin{lemma}[Regularity of $v_E$ and its distributional derivative formula]\label{distributional_derivative_formula}\mbox{}\\
Let $n \geq 2$ and let $E$ be a set of finite anisotropic Gaussian perimeter in $\mathbb{R}^{n}$. Then $v_{E} \in$ $BV\left(\mathbb{R}^{n-1}\right)$, i.e., $|Dv_E|(\R^{n-1})<\infty$, and
\begin{align*}
\frac{\det\sqrt{A}}{(2\pi)^{(n-1)/2}}|Dv_E|(G)&\leq P_{\gamma_A}(E ; G \times \mathbb{R})+\frac{\det\sqrt{A}}{(2\pi)^{(n-1)/2}}\int_G\left|\int_{E_z}\nabla^{\prime}\left(e^{-|\sqrt{A}x|^2/2}\right)dy\right| dz
\end{align*}
for every open set $G \subseteq \mathbb{R}^{n-1}$. Moreover, let $D_iv_E(z):=\frac{dD_iv_E{\tiny \mres} B_E}{d\mathcal{L}^{n-1} {\tiny\mres} B_E}(z)$,
$$
D_{i} v_{E}(z)=\int_{\left(\partial^{*} E\right)_{z}} \frac{\nu_{i}^{E}(z, y)}{\left|\nu_{n}^{E}(z, y)\right|} d \mathcal{H}_{z}^{0}(y) +\int_{E_z} \frac{\partial}{\partial x_i}\left(e^{-|\sqrt{A}x|^2/2}\right) dy\mbox{\qquad for $i=1,2,...,n-1$},
$$
for $\mathcal{L}^{n-1}$-a.e. $z \in B_{E}$, where $B_{E}$ is the set appearing in Vol'pert Theorem.
\end{lemma}
\begin{proof} Since $v_E \in L^1(\R^{n-1})$, our goal is to show that $v_E\in BV(\R^{n-1})$.\\
\noindent{\bf Step 1}: Let $\varphi \in C_{c}^{1}\left(\mathbb{R}^{n-1}\right)$ and $\psi_{j}\in C_{c}^{1}(\mathbb{R})$ with $0 \leq \psi_{j}(y) \leq 1$ for $y \in \mathbb{R}$, $j \in \mathbb{N}$, and such that $\lim _{j \rightarrow \infty} \psi_{j}(y)=1$ for every $y \in \mathbb{R}$. For any $i =1, \ldots, n-1$, by the dominated convergence theorem,
\begin{align}
&\int_{\mathbb{R}^{n-1}} \frac{\partial \varphi}{\partial z_{i}}\left(z\right) v_E\left(z\right) d z= \int_{\mathbb{R}^{n-1}} \frac{\partial \varphi}{\partial z_{i}}\left(z\right)\int_{E_{z}}e^{-|\sqrt{A}x|^2/2}dy\ d z\nonumber=\int_{\mathbb{R}^{n-1}} \left( \int_{\mathbb{R}} \frac{\partial \varphi}{\partial z_{i}}\left(z\right) \chi_{E}\left(z, y\right)e^{-|\sqrt{A}x|^2/2}d y\right) dz\nonumber \\ 
&=\lim _{j \rightarrow \infty} \int_{\mathbb{R}^{n}} \frac{\partial \varphi}{\partial z_{i}}\left(z\right) \psi_{j}(y) \chi_{E}\left(z, y\right)e^{-|\sqrt{A}x|^2/2} d z d y\mbox{\qquad (insert $\lim_{j}\psi_j=1$)}\nonumber \\
&=\lim _{j \rightarrow \infty}\Bigg\{\int_{\mathbb{R}^{n}} \div\left(\varphi\left(z\right) \psi_{j}(y) e^{-|\sqrt{A}x|^2/2}e_i\right) \chi_{E}\left(z, y\right) d z d y-\int_{\R^n}\frac{\partial}{\partial z_i}\left(e^{-|\sqrt{A}x|^2/2}\right)\varphi(z)\psi_j(y)\chi_{E}\left(z, y\right)\ dydz\Bigg\}\nonumber\\
&=\lim _{j \rightarrow \infty}\Bigg\{ \int_{E} \div\left(\varphi\left(z\right) \psi_{j}(y) e^{-|\sqrt{A}x|^2/2}e_i\right) d z d y-\int_{E}\frac{\partial}{\partial z_i}\left(e^{-|\sqrt{A}x|^2/2}\right)\varphi(z)\psi_j(y)\ dydz\Bigg\}\nonumber\\
&=-\lim _{j \rightarrow \infty}\Bigg\{ \int_{\rb E} \varphi\left(z\right) \psi_{j}(y) e^{-|\sqrt{A}x|^2/2}e_i \cdot  \nu^E\ d\mathcal{H}^{n-1}(z, y)+\int_{E}\frac{\partial}{\partial z_i}\left(e^{-|\sqrt{A}x|^2/2}\right)\varphi(z)\psi_j(y)\ dydz\Bigg\}\nonumber\\
&=-\lim _{j \rightarrow \infty} \Bigg\{\int_{\mathbb{R}^{n}} \varphi\left(z\right) \psi_{j}(y) e^{-|\sqrt{A}x|^2/2} d D_{i} \chi_{E}+\int_{E}\frac{\partial}{\partial z_i}\left(e^{-|\sqrt{A}x|^2/2}\right)\varphi(z)\psi_j(y)\ dydz\Bigg\}\mbox{\qquad (by (\ref{DeGiorgi_inner_normal}))}\nonumber\\
&=-\int_{\mathbb{R}^{n}} \varphi\left(z\right) e^{-|\sqrt{A}x|^2/2} d D_{i} \chi_{E}-\int_{E}\frac{\partial}{\partial z_i}\left(e^{-|\sqrt{A}x|^2/2}\right)\varphi(z)\ dydz .\notag
\end{align}
(a) Notice that for $|\varphi|\leq 1$,
\begin{align*}
\int_{\mathbb{R}^{n}} \left(-\varphi\left(z\right) \psi_{j}(y)\right) e^{-|\sqrt{A}x|^2/2} d D_{i} \chi_{E}
&\leq\int_{\rb E}e^{-|\sqrt{A}x|^2/2}\ d\mathcal{H}^{n-1}(x)=P_{\gamma_A}(E),
\end{align*}
and hence
$$-\lim_{j\to \infty}\int_{\mathbb{R}^{n}} \varphi\left(z\right) \psi_{j}(y) e^{-|\sqrt{A}x|^2/2} d D_{i} \chi_{E}\leq P_{\gamma_A}(E)<\infty.$$
(b) By Lemma \ref{computational_lemma}, there exists a constant $C>0$ such that
\begin{align*}
-\int_{E}\frac{\partial}{\partial z_i}\left(e^{-|\sqrt{A}x|^2/2}\right)\varphi(z)\psi_j(y)\ dydz&\leq  \int_{\R^n}\left|\frac{\partial}{\partial z_i}\left(e^{-|\sqrt{A}x|^2/2}\right)\right| dydz\leq C
\end{align*}
and hence
$$-\lim_{j\to \infty}\int_{E}\frac{\partial}{\partial z_i}\left(e^{-|\sqrt{A}x|^2/2}\right)\varphi(z)\psi_j(y)\ dydz\leq C<\infty.$$
Taking the sup over $\varphi\in C_{c}^{1}\left(\mathbb{R}^{n-1}\right)$ with $|\varphi| \leq 1$, (a), and (b), we conclude that $v_E \in BV\left(\mathbb{R}^{n-1}\right)$ since for any vector function $\varphi=(\varphi_1,..,\varphi_{n-1})\in C_{c}^{1}\left(\mathbb{R}^{n-1};\R^{n-1}\right)$ with $|\varphi|\leq 1$, the above argument works for each $\varphi_i$, $i=1,...,n-1$. Moreover, we have
\beq\label{distributional_derivative_formula_eq1}
\int_{\R^{n-1}}\varphi(z)\ dD_iv_E(z)=\int_{\mathbb{R}^{n}} \varphi\left(z\right) e^{-|\sqrt{A}x|^2/2} d D_{i} \chi_{E}+\int_{E}\frac{\partial}{\partial z_i}\left(e^{-|\sqrt{A}x|^2/2}\right)\varphi(z)\ dydz
\eeq
for every $\varphi\in C^1_c(\R^{n-1})$. \\

Now we consider $\eta=(\eta_1,...,\eta_{n-1})\in C^1_c(G;\R^{n-1})$ with $|\eta|\leq 1$, where $G$ is open in $\R^{n-1}$. By an approximation argument, we may set $\varphi$ in (\ref{distributional_derivative_formula_eq1}) as
$$\varphi(z)=\eta_i(z)\chi_G(z).$$
That is, for $i=1,...,n-1$,
$$
\int_{G}\eta_i(z)\ dD_iv_E\left(z\right) =\int_{G\times \R}  e^{-|\sqrt{A}x|^2/2} \eta_i(z)\ d D_i \chi_{E} (x)+\int_{E\cap (G\times\R)}\frac{\partial}{\partial z_i}\left(e^{-|\sqrt{A}x|^2/2}\right)\eta_i(z)\ dydz.
$$
Therefore, 
\begin{align}\label{distributional_derivative_formula_eq2}
\int_{G}\eta(z)\cdot  dDv_E(z)&=\int_{G\times \R}  e^{-|\sqrt{A}x|^2/2}\eta(z)\cdot  d \left(D_1 \chi_{E},D_2 \chi_{E},...,D_{n-1} \chi_{E}\right) (x)\\
&\quad+\int_{E\cap(G\times \R)}\eta(z)\cdot \nabla^{\prime}\left(e^{-|\sqrt{A}x|^2/2}\right)\ dydz\notag
\end{align}
for any $\eta\in C_c^{1}(G;\R^{n-1})$ with $|\eta|\leq 1$. Let $\psi_{j}\in C_{c}^{1}(\mathbb{R})$ with $0 \leq \psi_{j}(y) \leq 1$ for $y \in \mathbb{R}$, $j \in \mathbb{N}$, and such that $\lim _{j \rightarrow \infty} \psi_{j}(y)=1$ for every $y \in \mathbb{R}$. By the dominated convergence theorem and Proposition \ref{total_variation_2},
\begin{align*}
&\int_{G}\eta(z)\cdot  dDv_E(z)-\int_{E\cap(G\times \R)}\eta(z)\cdot \nabla^{\prime}\left(e^{-|\sqrt{A}x|^2/2}\right)\ dydz\\
&=\lim_{j\to \infty}\int_{G\times \R}  e^{-|\sqrt{A}x|^2/2}\left(\eta(z)\psi_j(y)\right)\cdot  d \left(D_1 \chi_{E},D_2 \chi_{E},...,D_{n-1} \chi_{E}\right) (x)\\
&\leq \sup\left\{\int_{G\times \R}  e^{-|\sqrt{A}x|^2/2}\tilde{\eta}(x)\cdot  d \left(D_1 \chi_{E},D_2 \chi_{E},...,D_{n-1} \chi_{E}\right) (x): \tilde{\eta}\in C_c(G\times \R;\R^{n-1}),|\tilde{\eta}|\leq 1\right\}\\
&= \int_{G\times \R}  e^{-|\sqrt{A}x|^2/2}\ d \left|\left(D_1 \chi_{E},D_2 \chi_{E},...,D_{n-1} \chi_{E}\right)\right| (x)\\
&\leq \int_{G\times \R}  e^{-|\sqrt{A}x|^2/2}\ d \left|D\chi_{E}\right| (x)=\int_{G\times \R}  e^{-|\sqrt{A}x|^2/2}\ d \mathcal{H}^{n-1}\mres \rb E (x)=\frac{(2\pi)^{(n-1)/2}}{\det\sqrt{A}} P_{\gamma_A}(E ; G \times \mathbb{R}).
\end{align*}
Thus,
\begin{align*}
\int_{G}\eta(z)\cdot  dDv_E(z)&\leq \frac{(2\pi)^{(n-1)/2}}{\det\sqrt{A}} P_{\gamma_A}(E ; G \times \mathbb{R})+\int_{E\cap(G\times \R)}\eta(z)\cdot \nabla^{\prime}\left(e^{-|\sqrt{A}x|^2/2}\right)\ dydz\\
&=\frac{(2\pi)^{(n-1)/2}}{\det\sqrt{A}} P_{\gamma_A}(E ; G \times \mathbb{R})+\int_G \eta(z)\cdot\int_{E_z} \nabla^{\prime}\left(e^{-|\sqrt{A}x|^2/2}\right)\ dy\ dz\\
&\leq  \frac{(2\pi)^{(n-1)/2}}{\det\sqrt{A}} P_{\gamma_A}(E ; G \times \mathbb{R})+\int_G\left|\int_{E_z} \nabla^{\prime}\left(e^{-|\sqrt{A}x|^2/2}\right)\ dy\right|dz.
\end{align*}
Taking the sup over $\eta$ on both sides, we have
\begin{align*}
|Dv_E|(G)&\leq\frac{(2\pi)^{(n-1)/2}}{\det\sqrt{A}} P_{\gamma_A}(E ; G \times \mathbb{R})+\int_G\left|\int_{E_z} \nabla^{\prime}\left(e^{-|\sqrt{A}x|^2/2}\right)\ dy\right|dz
\end{align*}
for every open set $G \subseteq \mathbb{R}^{n-1}$.\\

\noindent{\bf Step 2}: Let $B_{E}$ be the set given by Vol'pert (Theorem \ref{Vol'pert}). Applying equation (\ref{outer unit vector of reduced boundary}), we have
$$
\frac{\nu_{i}^{E}\left(z, y\right)}{\left|\nu_{n}^{E}\left(z, y\right)\right|}=\lim _{r \rightarrow 0} \frac{D_{i} \chi_{E}\left(B_{r}\left(z, y\right)\right)}{\left|D_{n} \chi_{E}\right|\left(B_{r}\left(z, y\right)\right)},
$$
for every $z \in B_E$ and every $y$ such that $(z, y)\in \rb E$. By the Besicovitch differentiation theorem,
$$D_{i} \chi_{E} \mres\left(B_{E} \times \mathbb{R}\right)=\frac{\nu_{i}^{E}}{\left|\nu_{n}^{E}\right|}\left|D_{n} \chi_{E}\right| \mres\left(B_{E} \times \mathbb{R}\right).$$
Now, let $g$ be any function in $C_{c}\left(\mathbb{R}^{n-1}\right)$. We can set $\varphi\left(z\right)=g\left(z\right) \chi_{B_{E}}\left(z\right)$ in (\ref{distributional_derivative_formula_eq1}) since we can first approximate opens sets then Borel sets. Therefore,
\begin{align}\label{weak eq 1}
\int_{B_{E}} g\left(z\right) d D_{i} v_E&=\int_{B_{E} \times \mathbb{R}} g\left(z\right)e^{-|\sqrt{A}x|^2/2}  d D_{i} \chi_{E}(x)+\int_{E\cap (B_E\times \R)}\frac{\partial}{\partial z_i}\left(e^{-|\sqrt{A}x|^2/2}\right)g(z)\ dydz\notag\\
&=\int_{B_{E} \times \mathbb{R}} \frac{\nu_{i}^{E}\left(z, y\right)}{\left|\nu_{n}^{E}\left(z, y\right)\right|} g\left(z\right)e^{-|\sqrt{A}x|^2/2} d\left|D_{n} \chi_{E}\right|+\int_{B_E}g(z)\left(\int_{E_z}\frac{\partial}{\partial z_i}\left(e^{-|\sqrt{A}x|^2/2}\right) dy\right)dz.
\end{align}
Moreover, by $|D_n\chi_E|=|\nu^E_n|\mathcal{H}^{n-1}\mres \partial^*E$ and co-area formula (\ref{co-area formula for sets of finite perimeter_1}),
\begin{align}\label{weak eq 2}
\int_{B_{E} \times \mathbb{R}} \frac{\nu_{i}^{E}\left(z, y\right)}{\left|\nu_{n}^{E}\left(z, y\right)\right|} g\left(z\right)e^{-|\sqrt{A}x|^2/2} d\left|D_{n} \chi_{E}\right| &=\int_{\partial^{*} E \cap\left(B_{E} \times \mathbb{R}\right)} g\left(z\right)e^{-|\sqrt{A}x|^2/2} \nu_{i}^{E}\left(z, y\right) d \mathcal{H}^{n-1}\notag \\
&=\int_{B_{E}} g\left(z\right)  \int_{\left(\partial^{*} E\right)_{z}} \frac{\nu_{i}^{E}\left(z, y\right)}{\left|\nu_{n}^{E}\left(z, y\right)\right|}e^{-|\sqrt{A}x|^2/2} d \mathcal{H}^{0}(y)\ dz\notag\\
&=\int_{B_{E}} g\left(z\right)  \int_{\left(\partial^{*} E\right)_{z}} \frac{\nu_{i}^{E}\left(z, y\right)}{\left|\nu_{n}^{E}\left(z, y\right)\right|} \ d \mathcal{H}^{0}_{z}(y)\ dz
\end{align}
(see (\ref{def_with_z_section}) for the definition of $\mathcal{H}_z^0$). Combining (\ref{weak eq 1}) and (\ref{weak eq 2}) together,
$$\int_{B_{E}} g\left(z\right) d D_{i} v_E=\int_{B_{E}} g\left(z\right)  \Bigg(\int_{\left(\partial^{*} E\right)_{z}} \frac{\nu_{i}^{E}\left(z, y\right)}{\left|\nu_{n}^{E}\left(z, y\right)\right|} \ d \mathcal{H}^{0}_{z}(y)+\int_{E_z}\frac{\partial}{\partial z_i}\left(e^{-|\sqrt{A}x|^2/2}\right)dy\Bigg) d z.$$
Since $g$ is arbitrary, we have
$$D_iv_E\mres B_E=\left(\int_{\left(\partial^{*} E\right)_z} \frac{\nu_{i}^{E}(z,y)}{\left|\nu_{n}^{E}(z,y)\right|}\ d \mathcal{H}_{z}^{0}(y)+\int_{E_z}\frac{\partial}{\partial z_i}\left(e^{-|\sqrt{A}x|^2/2}\right)dy\right)\mathcal{L}^{n-1}\mres B_E.$$
 That is,
\beq\label{weak derivative formula_2}
D_{i} v_E(z)=\int_{\left(\partial^{*} E\right)_z} \frac{\nu_{i}^{E}(z,y)}{\left|\nu_{n}^{E}(z,y)\right|}\ d \mathcal{H}_{z}^{0}(y)+\int_{E_z}\frac{\partial}{\partial z_i}\left(e^{-|\sqrt{A}x|^2/2}\right)dy,
\eeq
for $\mathcal{L}^{n-1}$-a.e. $z\in B_E$.
\end{proof}

With Lemma \ref{computational_lemma} in hand, we first show that $h:z\mapsto \phi_z^{-1}(v_E(z))$ is $C^1$ when $v_E$ is $C^1$. The key idea is to use the integral equation (\ref{integral equation}) and the lower bound estimate of $e^{-|\sqrt{A}x|^2/2}$. We will prove the general case of $v_E$ in Theorem \ref{approximation_theorem_for_BV_maps}.

\begin{lemma}[Regularity Estimates for the map $z\mapsto \phi_z^{-1}(v(z))$]\label{regularity_estimates_for_C1}\mbox{}\\
Let $x=(z,y)\in \R^{n-1}\times \R$ and
$$\phi_z(t)=\int_{-\infty}^{t} e^{-|\sqrt{A}x|^2/2}dy.$$
Let $\Omega$ be an open set in $\R^{n-1}$ and $v\in C^1(\Omega)$. Then the map $h:z\mapsto \phi_z^{-1}(v(z))$ is also $C^1(\Omega)$,
and for all $z\in \Omega$, we have
$$\nabla^{\prime} h(z)=e^{\frac{|\sqrt{A}(z,h(z))|^{2}}{2}} \left(\nabla^{\prime} v(z)-\int_{-\infty}^{h(z)}\nabla^{\prime}\left(e^{-|\sqrt{A}(z,y)|^2/2}\right)dy\right).$$
Moreover, if we assume that $\nabla'v$ is locally Lipschitz on $\Omega$. Then
$$z\mapsto \nabla' h(z)\mbox{ is also locally Lipschitz on $\Omega$.}$$
\end{lemma}
\begin{proof}
Since $h(z)=\phi_z^{-1}(v(z))$, i.e., $\phi_z(h(z))=v(z)$, we have
\begin{align}\label{integral equation}
\int_{-\infty}^{h(z)}e^{-|\sqrt{A}(z,y)|^2/2}dy=v(z).
\end{align}
{\bf Step 1}: Assume that $v\in C^0(\Omega)$. We first show that $h\in C^0(\Omega)$. For any $z_0\in \Omega$,
\begin{align*}
v(z)-v(z_0)&=\int_{-\infty}^{h(z)}e^{-|\sqrt{A}(z,y)|^2/2}dy-\int_{-\infty}^{h(z_0)}e^{-|\sqrt{A}(z_0,y)|^2/2}dy\\
&=\int_{-\infty}^{h(z)}e^{-|\sqrt{A}(z,y)|^2/2}dy-\int_{-\infty}^{h(z_0)}e^{-|\sqrt{A}(z,y)|^2/2}dy+\int_{-\infty}^{h(z_0)}e^{-|\sqrt{A}(z,y)|^2/2}dy-\int_{-\infty}^{h(z_0)}e^{-|\sqrt{A}(z_0,y)|^2/2}dy\\
&=\int_{h(z_0)}^{h(z)}e^{-|\sqrt{A}(z,y)|^2/2}dy+\int_{-\infty}^{h(z_0)}e^{-|\sqrt{A}(z,y)|^2/2}-e^{-|\sqrt{A}(z_0,y)|^2/2}dy.
\end{align*}
By Lemma \ref{computational_lemma} (2)(a) and $v\in C^0(\Omega)$, as $z\to z_0$ in $\Omega$,
$$\lim_{z\to z_0}\int_{h(z_0)}^{h(z)}e^{-|\sqrt{A}(z,y)|^2/2}dy=0$$
Now we claim that
$$h(z)\to h(z_0)\mbox{\quad as $z\to z_0$.}$$
Suppose not, there exists $\ep_0>0$ and a sequence $z_k\to z_0$ in $\Omega$ such that 
$$|h(z_k)-h(z_0)|\geq \ep_0.$$
Then there exists a subsequence $z_{k_n}$ such that
$$\mbox{(1) $h(z_{k_n})\geq h(z_0)$ for all $n$, or  (2) $h(z_{k_n})< h(z_0)$ for all $n$}.$$
We will only prove (1) since the proof of (2) is similar.  For (1), we have
$$|h(z_{k_n})-h(z_0)|\geq \ep_0\implies h(z_{k_n})\geq h(z_0)+\ep_0.$$
Hence, by Lemma \ref{computational_lemma} (1)(a),
\begin{align*}
\int_{h(z_0)}^{h(z_{k_n})}e^{-|\sqrt{A}(z,y)|^2/2}dy&\geq e^{-\|\sqrt{A}\|^2|z_{k_n}|^2/2}\int_{h(z_0)}^{h(z_{k_n})}e^{-\|\sqrt{A}\|^2|y|^2/2}dy\geq e^{-\|\sqrt{A}\|^2|z_{k_n}|^2/2}\int_{h(z_0)}^{h(z_{0})+\ep_0}e^{-\|\sqrt{A}\|^2|y|^2/2}dy.
\end{align*}
Taking $k\to \infty$,
\begin{align*}
0=\lim_{n\to \infty}\int_{h(z_0)}^{h(z_{k_n})}e^{-|\sqrt{A}(z,y)|^2/2}dy&\geq \lim_{n\to \infty}e^{-\|\sqrt{A}\|^2|z_{k_n}|^2/2}\int_{h(z_0)}^{h(z_{0})+\ep_0}e^{-\|\sqrt{A}\|^2|y|^2/2}dy\\
&\geq e^{-\|\sqrt{A}\|^2|z_{0}|^2/2}\int_{h(z_0)}^{h(z_{0})+\ep_0}e^{-\|\sqrt{A}\|^2|y|^2/2}dy>0.
\end{align*}
This gives us a contradiction. Therefore, $h\in C^0(\Omega)$.\\

\noindent{\bf Step 2}: Now we assume that $v\in C^1(\Omega)$. Our goal is to show that $h\in C^1(\Omega)$.\\
Define
$$\ell(z)=e^{\frac{|\sqrt{A}(z,h(z))|^{2}}{2}} \left(\nabla' v(z)-\int_{-\infty}^{h(z)}\nabla'\left(e^{-|\sqrt{A}(z,y)|^2/2}\right)dy\right).$$
We show that $h$ is differentiable and
$$\nabla^{\prime} h(z)=\ell(z).$$
Using the mean value theorem, we have
\begin{align*}
&\frac{1}{|k|}\left\{v(z+k)-v(z)-\nabla^{\prime} v(z)\cdot k\right\}\\
&=\frac{1}{|k|}\Bigg\{\int_{-\infty}^{h(z+k)}e^{-|\sqrt{A}(z+k,y)|^2/2}dy-\int_{-\infty}^{h(z)}e^{-|\sqrt{A}(z,y)|^2/2}dy-\nabla^{\prime} v(z)\cdot k\Bigg\}\\
&=\frac{1}{|k|}\Bigg\{\int_{h(z)}^{h(z+k)}e^{-|\sqrt{A}(z+k,y)|^2/2}dy+\int_{-\infty}^{h(z)}\left(e^{-|\sqrt{A}(z+k,y)|^2/2}-e^{-|\sqrt{A}(z,y)|^2/2}\right)dy -\nabla^{\prime} v(z)\cdot k\Bigg\}\\
&=\frac{1}{|k|}\Bigg\{\left(h(z+k)-h(z)\right)e^{-|\sqrt{A}(z+k,y(k))|^2/2}+\int_{-\infty}^{h(z)}\left(e^{-|\sqrt{A}(z+k,y)|^2/2}-e^{-|\sqrt{A}(z,y)|^2/2}\right)dy -\nabla^{\prime} v(z)\cdot k\Bigg\}\\
&=e^{-|\sqrt{A}(z+k,y(k))|^2/2}\frac{1}{|k|}\Bigg(h(z+k)-h(z)-\ell(z)\cdot k\Bigg)+\frac{1}{|k|}e^{-|\sqrt{A}(z+k,y(k))|^2/2}\ell(z)\cdot k\\
&\quad+\frac{1}{|k|}\Bigg\{\int_{-\infty}^{h(z)}\left(e^{-|\sqrt{A}(z+k,y)|^2/2}-e^{-|\sqrt{A}(z,y)|^2/2}\right)dy -\nabla^{\prime} v(z)\cdot k\Bigg\}
\end{align*}
where $y(k)$ lies between $h(z)$ and $h(z+k)$, and the continuity of $h$ (Step 1) implies that
$$|y(k)-h(z)|\leq |h(z+k)-h(z)|\implies y(k)\to h(z)\mbox{ as $k\to 0$}.$$
By using the definition of $\ell(z)$ and (\ref{computational_lemma_(2)(b)}), we have
\begin{align*}
\frac{h(z+k)-h(z)-\ell(z)\cdot k}{|k|}&=e^{|\sqrt{A}(z+k,y(k))|^2/2}\left(\frac{v(z+k)-v(z)-\nabla^{\prime} v(z)\cdot k}{|k|}\right)-\frac{1}{|k|}\ell(z)\cdot k\\
&\quad-e^{|\sqrt{A}(z+k,y(k))|^2/2}\frac{1}{|k|}\Bigg\{\int_{-\infty}^{h(z)}\left(e^{-|\sqrt{A}(z+k,y)|^2/2}-e^{-|\sqrt{A}(z,y)|^2/2}\right)dy -\nabla^{\prime} v(z)\cdot k\Bigg\}\\
&=e^{|\sqrt{A}(z+k,y(k))|^2/2}\left(\frac{v(z+k)-v(z)-\nabla^{\prime} v(z)\cdot k}{|k|}\right)-\frac{1}{|k|}e^{|\sqrt{A}(z,h(z))|^{2}/2} \nabla^{\prime} v(z)\cdot k\\
&\quad+\frac{1}{|k|}e^{|\sqrt{A}(z,h(z))|^{2}/2}\int_{-\infty}^{h(z)}\nabla^{\prime}\left(e^{-|\sqrt{A}x|^2/2}\right)dy \cdot k\\
&\quad-e^{|\sqrt{A}(z+k,y(k))|^2/2}\frac{1}{|k|}\Bigg\{\int_{-\infty}^{h(z)}\left(e^{-|\sqrt{A}(z+k,y)|^2/2}-e^{-|\sqrt{A}(z,y)|^2/2}\right)dy -\nabla^{\prime} v(z)\cdot k\Bigg\}\\
&=e^{|\sqrt{A}(z+k,y(k))|^2/2}\left(\frac{v(z+k)-v(z)-\nabla^{\prime} v(z)\cdot k}{|k|}\right)\\
&\quad +\frac{1}{|k|}\left(e^{|\sqrt{A}(z+k,y(k))|^2/2}-e^{|\sqrt{A}(z,h(z))|^{2}/2}\right)\nabla^{\prime}v(z)\cdot k\\
&\quad -e^{|\sqrt{A}(z+k,y(k))|^{2}/2}\int_{-\infty}^{h(z)}\Bigg(\frac{e^{-|\sqrt{A}(z+k,y)|^2/2}-e^{-|\sqrt{A}(z,y)|^2/2}-\nabla^{\prime}\left(e^{-|\sqrt{A}x|^2/2}\right)\cdot k}{|k|}\Bigg)dy\\
&\quad-\left(e^{|\sqrt{A}(z+k,y(k))|^{2}/2}-e^{|\sqrt{A}(z,h(z))|^2/2}\right)\frac{1}{|k|}\int_{-\infty}^{h(z)}\nabla^{\prime}\left(e^{-|\sqrt{A}x|^2/2}\right)dy\cdot k\to 0
\end{align*}
as $k\to 0$. Therefore, $h$ is differentiable and
$$\nabla^{\prime} h(z)=e^{\frac{|\sqrt{A}(z,h(z))|^{2}}{2}} \left(\nabla^{\prime} v(z)-\int_{-\infty}^{h(z)}\nabla^{\prime}\left(e^{-|\sqrt{A}(z,y)|^2/2}\right)dy\right).$$
Now we claim that $\nabla' h\in C^0(\Omega)$. Since $\nabla'v \in C^0(\Omega)$ and $h\in C^0(\Omega)$, we just need to show that
\begin{align}\label{intermediate_term_estimate}
z\mapsto \int_{-\infty}^{h(z)}\nabla^{\prime}\left(e^{-|\sqrt{A}(z,y)|^2/2}\right)dy \mbox{\quad is in $C^0(\Omega)$}.
\end{align}
Without loss of generality,  we may assume that $|z-z_0|\leq 1$ and let $K=\bar{B}(z_0,1)$.
\begin{align}\label{regularity_estimates_for_C1_eq1}
&\int_{-\infty}^{h(z)}\nabla^{\prime}\left(e^{-|\sqrt{A}(z,y)|^2/2}\right)dy-\int_{-\infty}^{h(z_0)}\nabla^{\prime}\left(e^{-|\sqrt{A}(z,y)|^2/2}\right)\Big|_{z=z_0}dy\\
&=-\int_{-\infty}^{h(z)}e^{-|\sqrt{A}x|^2/2}A'(z,y)\ dy+\int_{-\infty}^{h(z_0)}e^{-|\sqrt{A}(z_0,y)|^2/2}A'(z_0,y)\ dy\notag\\
&=-A'\left(\int_{-\infty}^{h(z)}e^{-|\sqrt{A}(z,y)|^2/2}(z,y)\ dy-\int_{-\infty}^{h(z_0)}e^{-|\sqrt{A}(z_0,y)|^2/2}(z_0,y)\ dy\right)\notag\\
&=-A'\Bigg[\left(\int_{-\infty}^{h(z)}e^{-|\sqrt{A}(z,y)|^2/2}(z,y)\ dy-\int_{-\infty}^{h(z_0)}e^{-|\sqrt{A}(z,y)|^2/2}(z,y)\ dy\right)\notag\\
&\quad+\left(\int_{-\infty}^{h(z_0)}e^{-|\sqrt{A}(z,y)|^2/2}(z,y)\ dy-\int_{-\infty}^{h(z_0)}e^{-|\sqrt{A}(z_0,y)|^2/2}(z_0,y)\ dy\right)\Bigg]:=\text{(I)}+\text{(II)}.\notag
\end{align}
\noindent (i) We first estimate (II).
\begin{align*}
&\left|\int_{-\infty}^{h(z_0)}e^{-|\sqrt{A}(z,y)|^2/2}A'(z,y)\ dy-\int_{-\infty}^{h(z_0)}e^{-|\sqrt{A}(z_0,y)|^2/2}A'(z_0,y)\ dy\right|\\
&\leq \int_{-\infty}^{h(z_0)}\left|e^{-|\sqrt{A}(z,y)|^2/2}-e^{-|\sqrt{A}(z_0,y)|^2/2}\right| \left|A'(z,y)\right| dy+  \int_{-\infty}^{h(z_0)}e^{-|\sqrt{A}(z_0,y)|^2/2}\left|A'(z-z_0,0)\right| dy\\
&\leq  \lambda_{\max}(A^{\prime\mathsf{T}}A^{\prime})\int_{-\infty}^{\infty}\Big(r(K)+|y|\Big)^2e^{-\|(\sqrt{A})^{-1}\|^{-2}|y|^2/2} |z-z_0|
 dy+ \sqrt{\lambda_{\max}(A^{\prime\mathsf{T}}A^{\prime})}\int_{-\infty}^{\infty}e^{-\|(\sqrt{A})^{-1}\|^{-2}|y|^2/2}\left|z-z_0\right| dy\\
&= C(K,A)\left|z-z_0\right|
\end{align*}
where we have used (\ref{exp_estimate}) and recall that $r(K)=\sup_{\zeta\in K}|\zeta|$.\\
(ii) We now estimate (I). By the mean value theorem,
\begin{align}\label{locally_lipschitz_estimate}
\left|\int_{h(z_0)}^{h(z)}e^{-|\sqrt{A}(z,y)|^2/2}(z,y)\ dy\right|&=\left|\left(\int_{h(z_0)}^{h(z)}e^{-|\sqrt{A}(z,y)|^2/2}z\ dy,\int_{h(z_0)}^{h(z)}e^{-|\sqrt{A}(z,y)|^2/2}y\ dy\right)\right|\notag\\
&\leq \left|\int_{h(z_0)}^{h(z)}e^{-|\sqrt{A}(z,y)|^2/2}z\ dy\right|+\left|\int_{h(z_0)}^{h(z)}e^{-|\sqrt{A}(z,y)|^2/2}y\ dy\right|\notag\\
&=\left|e^{-|\sqrt{A}(z,y_z)|^2/2}z\right||h(z)-h(z_0)|+\left|e^{-|\sqrt{A}(z,\tilde{y}_z)|^2/2}\tilde{y}_z\right||h(z)-h(z_0)|\to 0
\end{align}
where $y_z$ and $\tilde{y}_{z}$ lies between $h(z)$ and $h(z_0)$, and by the continuity of $h$, $y_z,\tilde{y}_z\to h(z_0)$ as $z\to z_0$. Thus, $\text{(I)}\to 0$ and $\text{(II)}\to 0$. Therefore, $h\in C^1(\Omega).$\\

\noindent{\bf Step 3}: Finally, we claim that if $\nabla'v$ is locally Lipschitz on $\Omega$,
$$\nabla'h\mbox{ is also locally Lipschitz on $\Omega$.}$$
By Lemma \ref{computational_lemma} (2)(c), we just need to show that (\ref{intermediate_term_estimate}) is locally Lipschitz on $\Omega$. Let $z^*\in \Omega$. Since $\Omega$ is open, there exists $B(z^*,r)\subset \Omega$ s.t. $\bar{B}(z^*,r)\subset \Omega$. Then $K:=\bar{B}(z^*,r)$ is a convex compact set. Thanks to the estimates in Step 2, we are left to estimate (\ref{locally_lipschitz_estimate}) with $z,z_0\in B(z^*,r)\subset K$. Since $h\in C^1(\Omega)$ and $|\tilde{y}_z|\leq \|h\|_{L^{\infty}(K)}$,
\begin{align*}
(\ref{locally_lipschitz_estimate})\leq r(K)\|\nabla h\|_{L^{\infty}(K)}|z-z_0|+\|h\|_{L^{\infty}(K)}\|\nabla h\|_{L^{\infty}(K)}|z-z_0|,
\end{align*}
i.e., (\ref{intermediate_term_estimate}) is Lipschitz on $B(z^*,r)$.
\end{proof}

By Lemma \ref{distributional_derivative_formula}, we know that $v_E\in BV(\R^{n-1})$. Now we are ready to show that $h(z)=\phi_z^{-1}(v_E(z))$ is $\mathcal{L}^{n-1}$-measurable and $E^s$ is a set of locally finite perimeter in $\R^{n}$. The key idea is to approximate $v_E\in BV(\R^{n-1})$ by $C^1$ functions. Then we can apply Lemma \ref{regularity_estimates_for_C1} on each $C^1$ function.

\begin{theorem}[Approximation theorem for the map $z\mapsto \phi^{-1}_z(v_E(z))$]\label{approximation_theorem_for_BV_maps}\mbox{}\\
Let $n \geq 2$ and let $E$ be a set of finite $A$-anisotropic Gaussian perimeter in $\mathbb{R}^{n}$. Define $h:\R^{n-1}\to \R$ as $h(z)=\phi_z^{-1}(v_E(z))$. Then
\begin{enumerate}
\item for any open set $\Omega\subset \R^{n-1}$, there exists a sequence of functions $v_{k} \in C^{1}(\Omega)$, such that $v_{k} \rightarrow v_{E}$ in $L^{1}(\Omega)$ and a.e. in $\Omega$, $Dv_k\wkly Dv_E$ in $\Omega$, and
$$
\lim _{k \rightarrow \infty} \int_{\Omega}\left|\nabla' v_{k}(z)\right| d z=\left|D v_{E}\right|(\Omega).
$$
Moreover,
$$h_k \to h\mbox{\ a.e. in $\Omega$},\quad \chi_{F_k}\to \chi_{E^s}\mbox{\ a.e. in $\Omega\times\R$}$$
where $h_k(z):=\phi^{-1}_z(v_k(z))$ and
$$F_k:= \left\{(z, y) \in \Omega \times \mathbb{R}: y<h_k(z):=\phi_{z}^{-1}(v_k(z))\right\}.$$
In particular, $h$ is $\mathcal{L}^{n-1}$-measurable and the set $E^s$ is $\mathcal{L}^n$-measurable.
\item $E^s=E^s_{A,-e_n}$ is a set of locally finite perimeter in $\mathbb{R}^{n}$. 
\end{enumerate}
\end{theorem}
\begin{proof}
(1) Since $E$ is a set of finite anisotropic Gaussian perimeter, by Lemma \ref{distributional_derivative_formula}, $v_{E} \in BV(\Omega)$. Hence there exists a sequence of functions $v_{k} \in C^{1}(\Omega)$, such that $v_{k} \rightarrow v_{E}$ in $L^{1}(\Omega)$ and a.e. in $\Omega$, $Dv_k\wkly Dv_E$ in $\Omega$, and
$$
\lim _{k \rightarrow \infty} \int_{\Omega}\left|\nabla' v_{k}(z)\right| d z=\left|D v_{E}\right|(\Omega)
$$
(see \cite{Ambrosio_Fusco_Pallara}, Theorem 3.9 and Proposition 3.13). Since $v_k\to v_E$ a.e. in $\Omega$, there exists a measure zero set $Z\subset \Omega$ such that for any $z\in \Omega\slash Z$, $v_k(z)\to v_E(z)$. Notice that
\begin{align*}
v_k(z)-v_E(z)&=\int_{-\infty}^{h_k(z)}e^{-|\sqrt{A}(z,y)|^2/2}dy-\int_{-\infty}^{h(z)}e^{-|\sqrt{A}(z,y)|^2/2}dy=\int_{h(z)}^{h_k(z)}e^{-|\sqrt{A}(z,y)|^2/2}dy.
\end{align*}
As $k\to \infty$, we have
$$\lim_{k\to \infty}\int_{h(z)}^{h_k(z)}e^{-|\sqrt{A}(z,y)|^2/2}dy=0\quad\mbox{for any $z\in \Omega\slash Z.$}$$
By using the same argument as in Lemma \ref{regularity_estimates_for_C1} Step 1, we have for any $z\in \Omega\slash Z$, $h_k(z)\to h(z)$ as $k\to \infty$. By Lemma \ref{regularity_estimates_for_C1}, $h_k\in C^1(\Omega)$ since $v_k\in C^1(\Omega)$. In particular, $h_k$ is $\cL^{n-1}\mres\Omega$ -measurable on $\Omega$ and hence the limit function $h|_{\Omega}$ is also $\cL^{n-1}\mres\Omega$ -measurable on $\Omega$. Since $\Omega$ is arbitrary, $h$ is also $\cL^{n-1}$-measurable. Next we show that
$$\mbox{$\chi_{F_{k}} \rightarrow \chi_{E^{s}}$\ a.e. in $\Omega\times \R$}.$$
Let $\Gamma(h)$ be the graph of $h$, let $Z'=(Z\times \R)\cup \Gamma(h;\Omega)\subset \Omega\times \R$, where $Z$ is the measure zero set from above and $\Gamma(h;\Omega)$ is the graph of $h$ over $\Omega$, i.e.,
$$\Gamma(h;\Omega)=\{(z,y)\in \Omega\times \R:y=h(z)\}=\Gamma(h)\cap (\Omega\times\R).$$
We first check that $\cL^{n}(Z')=0$. Notice that $\cL^{n}(Z\times \R)=0$. It is enough to show that $\cL^{n}(\Gamma(h))=0$. Since $h$ is $\cL^{n-1}$-measurable, we define
$$g(z,y):=f_2\circ f_1(z,y)=h(z)-y$$
where $f_1:(z,y)\mapsto (h(z),y)$ is $\cL^{n}$-measurable, $f_2:(x,y)\mapsto x-y$ is continuous. Therefore, $g$ is also $\cL^{n}$-measurable and hence $\Gamma(h)=g^{-1}(\{0\})$ is $\cL^{n}$-measurable. By Fubini's theorem, $\cL^{n}(\Gamma(h))=0$.
Next, for any $x=(z,y)\in \Omega\times\R\slash Z'$, we have $z\not\in Z$. If $\chi_{E^s}(x)=1$, then $y<h(z)$ and hence there exists $k$ such that $y<h_k(z)$ since $h_k(z)\to h(z)$. That is, $x=(z,y)\in F_k$, i.e., $\chi_{F_k}(x)=1$ and $\chi_{F_k}\to \chi_{E^s}$. If $\chi_{E^s}(x)=0$, then $y\geq h(z)$. However, $x\not\in \Gamma(h;\Omega)$, so $y\not=h(z)$ and there exists $k$ such that $y>h_k(z)$. That is, $x=(z,y)\not\in F_k$, i.e., $\chi_{F_k}(x)=0$ and $\chi_{F_k}\to \chi_{E^s}$. Therefore, $\chi_{F_{k}} \rightarrow \chi_{E^{s}}$  a.e. in $\Omega\times \R$. In particular, $E^s=\{g>0\}$ is $\cL^{n}$-measurable.\\

\noindent (2) Now we claim that $E^s$ is a set of locally finite perimeter in $\R^n$.
Since $E^s$ is $\cL^{n}$-measurable, we have $\chi_{E^s}\in L^1_{loc}(\R^n)$. In order to show $\chi_{E^s}\in BV_{loc}(\R^{n-1})$, by Proposition \ref{lower_semicontinuity} and \ref{locally finite anisotropic Gaussian perimeter}, we just need to prove that for any open set $V\subset \subset \R^{n}$,
$$
\sup \left\{\int_{E^s} \operatorname{div} \varphi(x)-\langle \varphi(x), Ax\rangle\ d \gamma_{A}(x): \varphi \in C_{c}^{1}\left(V ; \mathbb{R}^{n}\right), |\varphi| \leq 1\right\} <\infty.
$$
Since $\bar{V}\subset \R^n$ is a compact set, there exists an open bounded set $\Omega\subset \R^{n-1}$ such that $V\subset \Omega\times \R$. In fact, we claim that
\begin{align}\label{perimeter_inside_cylinder}
\sup \left\{\int_{E^s} \operatorname{div} \varphi(x)-\langle \varphi(x), Ax\rangle\ d \gamma_{A}(x): \varphi \in C_{c}^{1}\left(\Omega\times \R ; \mathbb{R}^{n}\right), |\varphi| \leq 1\right\} <\infty.
\end{align}
For convenience, let $\varphi_z=(\varphi_1,\varphi_2,...,\varphi_{n-1})$
and
\begin{align}\label{(I) and (II)}
\int_{E^s} \operatorname{div} \varphi-\langle\varphi , Ax\rangle\ d \gamma_{A}(x)&=\int_{E^s} \operatorname{div} \varphi+\langle\varphi , \nabla (-|\sqrt{A}x|^2/2)\rangle\ d \gamma_{A}(x)\notag\\
&=\sum_{i=1}^{n-1} \int_{E^s} \frac{\partial \varphi_{i}}{\partial z_{i}}+ \varphi_{i}\frac{\partial (-|\sqrt{A}x|^2/2)}{\partial z_i}\ d \gamma_{A}(x)+\int_{E^s} \frac{\partial \varphi_{n}}{\partial y}+\varphi_n\frac{\partial (-|\sqrt{A}x|^2/2)}{\partial y}\ d \gamma_{A}(x)\notag\\
&:= \text{(I)}+\text{(II)} .
\end{align}
{\bf Step 1}: We estimate (I) using approximation of $v_E$ by $C^1(\Omega)$ functions. Consider $\widetilde{v} \in C^{1}(\Omega)$ and let $\varphi \in C_{c}^{1}(\Omega \times \R;\R^n)$ such that $|\varphi| \leq 1$. Let
$$
F:= \{x=(z, y) \in \Omega \times \mathbb{R}: y<y(z)\}, \quad y(z)=\phi_{z}^{-1}(\widetilde{v}(z)).
$$
Since $y(z)=\phi_{z}^{-1}(\widetilde{v}(z))$, i.e., 
$$\tilde{v}(z)=\phi_z(y(z))=\int_{-\infty}^{y(z)}e^{-|\sqrt{A}x|^2/2}dy,$$
and by Lemma \ref{regularity_estimates_for_C1}, the mapping $z\mapsto y(z)$ is in $C^1(\Omega)$. Hence
$$\nabla^{\prime} \widetilde{v}(z)=e^{-\frac{|\sqrt{A}(z,y(z))|^{2}}{2}} \nabla^{\prime} y(z)+\int_{-\infty}^{y(z)}\nabla^{\prime}\left(e^{-|\sqrt{A}x|^2/2}\right)dy,$$
$$\frac{\partial \widetilde{v}}{\partial z_i}=e^{-\frac{|\sqrt{A}(z,y(z))|^{2}}{2}}\frac{\partial y(z)}{\partial z_i}+\int_{-\infty}^{y(z)}\frac{\partial}{\partial z_i}\left(e^{-|\sqrt{A}(z,y)|^2/2}\right) dy.$$
We now compute the equivalent of (I).
\begin{align}\label{locally_finite_perimeter_estimate_eq1}
&\sum_{i=1}^{n-1} \int_{F} \frac{\partial \varphi_{i}}{\partial z_{i}} +\varphi_{i}\frac{\partial (-|\sqrt{A}x|^2/2)}{\partial z_i}\ d \gamma_{A}(x)\notag \\
&=\sum_{i=1}^{n-1} \frac{\det\sqrt{A}}{(2\pi)^{n/2}}\int_{\Omega}  \int^{y(z)}_{-\infty} \left(\frac{\partial \varphi_{i}}{\partial z_{i}}+\varphi_{i}\frac{\partial (-|\sqrt{A}x|^2/2)}{\partial z_i}\right)e^{-|\sqrt{A}x|^2/2} d y\ dz\notag \\
&=\sum_{i=1}^{n-1} \frac{\det\sqrt{A}}{(2\pi)^{n/2}}\int_{\Omega} \left(\frac{\partial}{\partial z_i} \int^{y(z)}_{-\infty} \varphi_{i}e^{-|\sqrt{A}x|^2/2} d y-\frac{\partial y(z)}{\partial z_i}\varphi_i(z,y(z))e^{-|\sqrt{A}(z,y(z)|^2/2}\right) dz\notag \\
&=-\sum_{i=1}^{n-1}\frac{\det \sqrt{A}}{(2 \pi)^{n / 2}} \int_{\Omega}\frac{\partial y(z)}{\partial z_i}\varphi_i(z,y(z))e^{-|\sqrt{A}(z,y(z)|^2/2}\  d z\notag \\
&=-\sum_{i=1}^{n-1}\frac{\det \sqrt{A}}{(2 \pi)^{n / 2}} \int_{\Omega}\varphi_i(z,y(z))\left(\frac{\partial \widetilde{v}}{\partial z_i}-\int_{-\infty}^{y(z)}\frac{\partial}{\partial z_i}\left(e^{-|\sqrt{A}(z,y)|^2/2}\right) dy\right) d z\notag \\
&=-\frac{\det \sqrt{A}}{(2 \pi)^{n / 2}}\left(\int_{\Omega}\varphi_z(z,y(z))\cdot \nabla^{\prime}\widetilde{v}(z)\ d z-\int_{\Omega}\int_{-\infty}^{y(z)}\varphi_z(z,y(z))\cdot \nabla^{\prime}\left(e^{-|\sqrt{A}(z,y)|^2/2}\right)dydz\right)
\end{align}
where we have used the divergence theorem and $\varphi$ has compact support in $\Omega\times \R$.\\

\noindent Now we approximate $v_E$ by Theorem \ref{approximation_theorem_for_BV_maps} (1), i.e., there exists a sequence of functions $v_{k} \in C^{1}(\Omega)$, such that $v_{k}\rightarrow v_{E}$ in $L^{1}(\Omega)$ and a.e. in $\Omega$, $Dv_k\wkly Dv_E$ in $\Omega$, and
$$
\lim _{k \rightarrow \infty} \int_{\Omega}\left|\nabla' v_{k}(z)\right| d z=\left|D v_{E}\right|(\Omega).
$$
Moreover, $\mbox{$\chi_{F_{k}} \rightarrow \chi_{E^{s}}$\ a.e. in $\Omega\times \R$}$ and $h_k(z)\to h(z):=\phi_{z}^{-1}(v_E(z))$ a.e. in $\Omega$, where
$$
F_k:= \left\{(z, y) \in \Omega \times \mathbb{R}: y<h_k(z):=\phi_{z}^{-1}(v_k(z))\right\}.
$$
Replacing $F$ as $F_k$ and $y(z)$ as $h_k(z)$ in equation (\ref{locally_finite_perimeter_estimate_eq1}), we have the following estimate
\begin{align}
&\text{(I)}=\sum_{i=1}^{n-1} \int_{E^s} \frac{\partial \varphi_{i}}{\partial z_{i}}+\varphi_{i}\frac{\partial (-|\sqrt{A}x|^2/2)}{\partial z_i}\ d \gamma_{A}(x)\notag\\
&=\lim _{k \rightarrow \infty}\Big( \sum_{i=1}^{n-1} \int_{F_{k}} \frac{\partial \varphi_{i}}{\partial z_{i}}+\varphi_{i}\frac{\partial (-|\sqrt{A}x|^2/2)}{\partial z_i}\ d \gamma_{A}(x)\Big)\notag\\
&=\frac{\det \sqrt{A}}{(2 \pi)^{n / 2}} \lim _{k \rightarrow \infty}\Bigg(\int_{\Omega}\left(-\varphi_z(z,h_k(z))\right)\cdot \nabla^{\prime} v_k(z) d z-\int_{\Omega_j}\int_{-\infty}^{h_k(z)}\left(-\varphi_z(z,h_k(z))\right)\cdot \nabla^{\prime}\left(e^{-|\sqrt{A}x|^2/2}\right)dydz\Bigg)\notag\\
&\leq  \frac{\det \sqrt{A}}{(2 \pi)^{n / 2}}\limsup_{k\to \infty}\int_{\Omega}\left| \nabla^{\prime} v_k(z) \right| d z+ \frac{\det \sqrt{A}}{(2 \pi)^{n / 2}}\int_{\Omega}\int_{-\infty}^{h(z)}\varphi_z(z,h(z))\cdot \nabla^{\prime}\left(e^{-|\sqrt{A}x|^2/2}\right)dydz\notag\\
&\leq \frac{\det \sqrt{A}}{(2 \pi)^{n / 2}}\left|D v_{E}\right|(\Omega)+ \frac{\det \sqrt{A}}{(2 \pi)^{n / 2}}\int_{\R^{n-1}}\int_{-\infty}^{\infty}\left|e^{-|\sqrt{A}x|^2/2}A'x\right|dydz\notag\\
&\leq \frac{\det \sqrt{A}}{(2 \pi)^{n / 2}}\left|D v_{E}\right|(\Omega)+ \frac{\det \sqrt{A}}{(2 \pi)^{n / 2}}\sqrt{\lambda_{\max}(A^{\prime\mathsf{T}}A^{\prime})}\int_{\R^{n}}e^{-\|\sqrt{A}\|^{-2}|x|^2/2}|x| \ dx<\infty.\notag
\end{align}

\noindent{\bf Step 2}: Now we estimate (II).
\begin{align*}
&\text{(II)}=\int_{E^s} \frac{\partial \varphi_n}{\partial y}+\varphi_n \frac{\partial (-|\sqrt{A}x|^2/2)}{\partial y} d \gamma_A(x)=\frac{\operatorname{det} \sqrt{A}}{(2 \pi)^{n / 2}} \int_{E^s} \frac{\partial}{\partial y}\left(\varphi_n e^{-|\sqrt{A}(z,y)|^2/2}\right) d x\\
&=\frac{\operatorname{det} \sqrt{A}}{(2 \pi)^{n / 2}} \lim_{k\to \infty}\int_{F_k} \frac{\partial}{\partial y}\left(\varphi_n e^{-|\sqrt{A}(z,y)|^2/2}\right) d x=\frac{\operatorname{det} \sqrt{A}}{(2 \pi)^{n / 2}} \lim_{k\to \infty}\int_{\Omega}\int_{-\infty}^{h_k(z)} \frac{\partial}{\partial y}\left(\varphi_n e^{-|\sqrt{A}(z,y)|^2/2}\right) d ydz\\
&=\frac{\operatorname{det} \sqrt{A}}{(2 \pi)^{n / 2}} \lim_{k\to \infty}\int_{\Omega}\varphi_n(z,h_k(z)) e^{-|\sqrt{A}(z,h_k(z))|^2/2}\ dz=\frac{\operatorname{det} \sqrt{A}}{(2 \pi)^{n / 2}} \int_{\Omega}\varphi_n(z,h(z)) e^{-|\sqrt{A}(z,h(z))|^2/2}\ dz\\
&\leq \frac{\operatorname{det} \sqrt{A}}{(2 \pi)^{n / 2}}\mathcal{L}^{n-1}(\Omega)<\infty
\end{align*}
since $\varphi$ is $C^1$ with compact support, $|\varphi|\leq 1$, and recall that $h_k\to h$ a.e. in $\Omega$ by Theorem \ref{regularity_estimates_for_C1} (1).\\
\end{proof}

\subsection{Ehrhard symmetrization in 1D and an example in 2D}\mbox{}\vspace{-.22cm}\\

In this section, we show that the Ehrhard symmetrization preserves the mass under any one-dimensional slices in the $y$-direction, and hence it preserves the total mass. Moreover, the anisotropic Gaussian perimeters of the one-dimensional sections decrease under Ehrhard symmetrization. Geometrically, we will rearrange the mass in each one-dimensional section of $E$ to a half-line with the same mass. The resulting new shape $E^s$ is the Ehrhard symmetrization of $E$.

\begin{proposition}\label{basic_properties}
Let $n \geq 2$ and let $E$ be a set of finite anisotropic Gaussian perimeter in $\mathbb{R}^{n}$.
\begin{enumerate}
\item  {\rm \bf (Properties for $E^s$)}
$$v_{E^s}(z)= v_E(z),\qquad \mu_z\left((E^s)_{z}\right)=\mu_z\left(E_{z}\right),$$
for all $z\in \R^{n-1}$. Hence $\pi_+(E^s)=\pi_+(E)$
and
$$\nabla'v_E(z)=\nabla'v_{E^s}(z)$$
for a.e. $z\in B_E\cap B_{E^s}$, where 
$$\nabla'v_E(z):=\left(D_1v_E(z),\cdots,D_{n-1}v_E(z)\right),\quad D_iv_E:=\frac{dD_iv_E\mres B_E}{d\mathcal{L}^{n-1}\mres B_E}.$$
Moreover, $\gamma_A(E)=\gamma_A(E^s)$
and
$$\gamma_A(E_1\Delta E_2)\geq \gamma_A(E^s_1\Delta E^s_2).$$
In particular, for any sequence of sets of finite anisotropic Gaussian perimeter $E_{k}$ with $\chi_{E_{k}} \rightarrow \chi_{E}$ in $L^{1}\left(\mathbb{R}^{n},\gamma_A\right)$, we have
$$\mbox{$\chi_{E^s_{k}} \rightarrow \chi_{E^s}$ in $L^{1}\left(\mathbb{R}^{n},\gamma_A\right)$},$$
and
$$
P_{\gamma_A}(E^s ; U) \leq \liminf _{k \rightarrow \infty} P_{\gamma_A}\left(E^s_{k} ; U\right)\mbox{\quad for any open set $U\subset \R^{n}$.}
$$
\item {\rm \bf (Cross terms estimate)}\\
For any $z\in \R^{n-1}$,
$$\int_{E_z}y\ d\mu_z(y)\geq \int_{E^s_z}y\ d\mu_z(y),$$
and
\begin{align*}
&\Bigg|\int_{E_z}\nabla^{\prime}e^{-|\sqrt{A}x|^2/2}dy-\int_{E_z^s} \nabla^{\prime}e^{-|\sqrt{A}x|^2/2}dy\Bigg|\\
&=\left(\int_{E_z}y\ d\mu_z(y)-\int_{E^s_z}y\ d\mu_z(y)\right)\|Ae_n-\langle Ae_n,e_n\rangle e_n\|.
\end{align*}
\item {\rm \bf (Ehrhard symmetrization in 1D)}\\
The anisotropic Gaussian perimeter of almost every one-dimensional section in $y$-direction decreases under Ehrhard symmetrization, i.e.,
$$p_{E^s}(z)\leq p_E(z),$$
for all $z\in B_E\cap B_{E^s}$, where $p_{E}: \mathbb{R}^{n-1} \rightarrow[0, \infty]$ is defined as
$$p_E(z)=\mathcal{H}_z^0\left[(\partial^M E)_z\right].$$
\end{enumerate}
\end{proposition}
\begin{proof}
(1) Notice that
$$(E^s)_{z}=\big\{y:y<\phi^{-1}_{z}\left(v_{E}(z)\right)\big\}=\big(-\infty,\phi^{-1}_{z}\left(v_{E}(z)\right)\big)\mbox{ is a measurable set}.$$
Therefore, for any $z\in \R^{n-1}$,
\begin{align*}
v_{E^s}(z)=\mu_z\left((E^s)_{z}\right)&=\mu_z\left(-\infty, \phi^{-1}_{z}\left(v_{E}(z)\right)\right)=\int_{-\infty}^{\phi^{-1}_{z}\left(v_{E}(z)\right)}e^{-|Ax|^2/2} \ dy=\phi_{z}\left(\phi^{-1}_{z}\left(v_{E}(z)\right)\right)=v_E(z).
\end{align*}
Thus,
$$\pi_+(E)=\{z\in \R^{n-1}: v_E(z)>0\}=\{z\in \R^{n-1}: v_{E^s}(z)>0\}=\pi_+(E^s).$$
Moreover, for any measurable set $A$,
$$D_iv_E\mres B_E(A)=\int_A D_iv_E(z)\ d\mathcal{L}^{n-1}\mres B_E(z),$$
and
$$D_iv_{E^s}\mres B_{E^s}(A)=\int_A D_iv_{E^s}(z)\ d\mathcal{L}^{n-1}\mres B_{E^s}(z).$$
Therefore, for any measurable set $B\subset B_E\cap B_{E^s}$, setting $A=B$ in the above equations, we have
$$\int_B D_iv_E(z)\ d\mathcal{L}^{n-1}(z)=D_iv_E(B)=D_{i}v_{E^s}(B)=\int_B D_iv_{E^s}(z)\ d\mathcal{L}^{n-1}(z)$$
where the second equality holds since $v_E=v_{E^s}$. By the arbitrariness of $B$, for a.e. $z\in B_E\cap B_{E^s}$,
$$D_iv_E(z)=D_iv_{E^s}(z).$$
By Fubini's theorem,
\begin{align*}
&\gamma_A(E)=\frac{|\det \sqrt{A}|}{(2\pi)^{n/2}}\int_E e^{-|\sqrt{A}x|^2/2}\ dx=\frac{|\det \sqrt{A}|}{(2\pi)^{n/2}}\int_{\R^{n-1}}\int_{E_z} e^{-|\sqrt{A}x|^2/2}\ dy\ dz\\
&=\frac{|\det \sqrt{A}|}{(2\pi)^{n/2}}\int_{\R^{n-1}}v_E(z)\ dz=\frac{|\det \sqrt{A}|}{(2\pi)^{n/2}}\int_{\R^{n-1}}v_{E^s}(z)\ dz=\frac{|\det \sqrt{A}|}{(2\pi)^{n/2}}\int_{\R^{n-1}}\int_{E^s_z} e^{-|\sqrt{A}x|^2/2}\ dy\ dz=\gamma_A(E^s).
\end{align*}
Similarly, we have
\begin{align*}
&\gamma_A(E_1\Delta E_2)=\frac{|\det \sqrt{A}|}{(2\pi)^{n/2}}\int_{\R^{n-1}}\int_{(E_1\Delta E_2)_z} e^{-|\sqrt{A}x|^2/2}\ dy\ dz=\frac{|\det \sqrt{A}|}{(2\pi)^{n/2}}\int_{\R^{n-1}}\int_{(E_1)_z\Delta (E_2)_z} e^{-|\sqrt{A}x|^2/2}\ dy\ dz\\
&\geq \frac{|\det \sqrt{A}|}{(2\pi)^{n/2}}\int_{\R^{n-1}}\left|\int_{(E_1)_z} e^{-|\sqrt{A}x|^2/2}\ dy-\int_{(E_2)_z} e^{-|\sqrt{A}x|^2/2}\ dy\right|\ dz= \frac{|\det \sqrt{A}|}{(2\pi)^{n/2}}\int_{\R^{n-1}}\left|v_{E_1}(z)-v_{E_2}(z)\right|\ dz\\
&= \frac{|\det \sqrt{A}|}{(2\pi)^{n/2}}\int_{\R^{n-1}}\left|v_{(E_1)^s}(z)-v_{(E_2)^s}(z)\right|\ dz=\frac{|\det \sqrt{A}|}{(2\pi)^{n/2}}\int_{\R^{n-1}}\int_{(E_1^s)_z\Delta (E_2^s)_z} e^{-|\sqrt{A}x|^2/2}\ dy\ dz=\gamma_A(E_1^s\Delta E_2^s).
\end{align*}
For any sequence of measurable sets $E_{k}$ with $\chi_{E_{k}} \rightarrow \chi_{E}$ in $L^{1}\left(\mathbb{R}^{n},\gamma_A\right)$, we have
\begin{align*}
\int_{\R^n} \left|\chi_{E^s_k}-\chi_{E^s}\right|d\gamma_A&=\gamma_A(E_k^s\Delta E^s)\leq \gamma_A(E_k\Delta E)=\int_{\R^n} \left|\chi_{E_k}-\chi_{E}\right|d\gamma_A\to 0.
\end{align*}
By Theorem \ref{approximation_theorem_for_BV_maps} and Proposition \ref{lower_semicontinuity}, $E_k^s$ and $E^s$ are sets of locally finite perimeter and hence
$$
P_{\gamma_A}(E^s ; U) \leq \liminf _{k \rightarrow \infty} P_{\gamma_A}\left(E^s_{k} ; U\right)\mbox{\quad for any open set $U\subset \R^{n}$.}
$$
\noindent(2) Notice that
\begin{align*}
\mu_z(E^s_z\slash E_z)+\mu_z(E^s_z\cap E_z)=\mu_z(E^s_z)&=v_{E^s}(z)=v_E(z)=\mu_z(E_z)=\mu_z(E_z\slash E^s_z)+\mu_z(E_z\cap E^s_z).
\end{align*}
That is, $\mu_z(E^s_z\slash E_z)=\mu_z(E_z\slash E^s_z)$. Let $y(z)=\phi^{-1}_z(v_E(z))$, we have
$$y- y(z)< 0\mbox{\quad if $y\in E^s_z$,}\quad y- y(z)\geq 0\mbox{\quad if $y\not\in E^s_z$}.$$
Now we are ready to show that
$$\int_{E^s_z}y\ d\mu_z(y)\leq \int_{E_z}y\ d\mu_z(y).$$
Notice that
\begin{align*}
&\int_{E^s_z}y\ d\mu_z(y)-\int_{E_z}y\ d\mu_z(y)=\left(\int_{E^s_z\slash E_z}y\ d\mu_z(y)+\int_{E^s_z\cap E_z}y\ d\mu_z(y)\right)-\left(\int_{E_z\slash E_z^s}y\ d\mu_z(y)+\int_{E_z\cap E^s_z}y\ d\mu_z(y)\right)\\
&=\int_{E^s_z\slash E_z}y\ d\mu_z(y)-\int_{E_z\slash E_z^s}y\ d\mu_z(y)\\
&=\left(\int_{E^s_z\slash E_z}y-y(z)\ d\mu_z(y)+\int_{E^s_z\slash E_z}y(z)\ d\mu_z(y)\right)+\left(\int_{E_z\slash E_z^s}y(z)-y\ d\mu_z(y)-\int_{E_z\slash E^s_z}y(z)\ d\mu_z(y)\right)\\
&\leq\left(\int_{E^s_z\slash E_z}y(z)\ d\mu_z(y)-\int_{E_z\slash E^s_z}y(z)\ d\mu_z(y)\right)= y(z)\Big(\mu_z(E^s_z\slash E_z)-\mu_z(E_z\slash E^s_z)\Big)=0.
\end{align*}
On the other hand, by Lemma \ref{computational_lemma} (1), we have
$$\partial_{z_k}\left(e^{-|\sqrt{A}x|^2/2}\right)=-e^{-|\sqrt{A}x|^2/2}\langle \row_k(A),x\rangle,$$
and hence
\begin{align*}
\int_{E_z}\partial_{z_k}\left(e^{-|\sqrt{A}x|^2/2}\right)dy=-\int_{E_z}e^{-|\sqrt{A}x|^2/2}\left(\sum_{j=1}^{n-1}A_{kj}z_j+A_{kn}y\right)dy=-\sum_{j=1}^{n-1}A_{kj}z_jv_E(z)-A_{kn}\left(\int_{E_z}y\ d\mu_z(y)\right).
\end{align*}
Similarly, we have
\begin{align*}
\int_{E^s_z}\partial_{z_k}\left(e^{-|\sqrt{A}x|^2/2}\right)dy&=-\sum_{j=1}^{n-1}A_{kj}z_jv_{E^s}(z)-A_{kn}\left(\int_{E^s_z}y\ d\mu_z(y)\right).
\end{align*}
Since $v_E(z)=v_{E^s}(z)$,
$$\int_{E_z}\partial_{z_k}\left(e^{-|\sqrt{A}x|^2/2}\right)dy-\int_{E^s_z}\partial_{z_k}\left(e^{-|\sqrt{A}x|^2/2}\right)dy=\left(\int_{E^s_z}y\ d\mu_z(y)-\int_{E_z}y\ d\mu_z(y)\right)A_{kn.}$$
Therefore,
\begin{align*}
&\Bigg|\int_{E_z}\nabla^{\prime}e^{-|\sqrt{A}x|^2/2}dy-\int_{E_z^s} \nabla^{\prime}e^{-|\sqrt{A}x|^2/2}dy\Bigg|=\left|\int_{E_z}y\ d\mu_z(y)-\int_{E^s_z}y\ d\mu_z(y)\right|\|(A_{1n},A_{2n},\cdots, A_{(n-1)n})\|\\
&=\left|\int_{E_z}y\ d\mu_z(y)-\int_{E^s_z}y\ d\mu_z(y)\right|\|Ae_n-\langle Ae_n,e_n\rangle e_n\|=\left(\int_{E_z}y\ d\mu_z(y)-\int_{E^s_z}y\ d\mu_z(y)\right)\|Ae_n-\langle Ae_n,e_n\rangle e_n\|.
\end{align*}
(3) In order to work with probability measures, we first normalize our definitions of $\mu_z$ and $\phi_z$ as
$$\widetilde{\mu}_z(F):=\frac{1}{\mu_z(\R)}\mu_z(F)=\int_{F}\frac{1}{\mu_z(\R)}e^{-|\sqrt{A}(z,y)|^2/2}dy\mbox{\quad and \quad} \widetilde{\phi}_z(s):=\frac{1}{\mu_z(\R)}\int_{-\infty}^se^{-|\sqrt{A}(z,y)|^2/2}dy.$$
Notice that
$$(\widetilde{\phi}_z)^{-1}(\widetilde{\mu}_z(F))=\phi^{-1}_z(\mu_z(F)).$$
Now we claim that $\widetilde{\mu}_z$ is a log-concave measure on $\R$. We just need to show that 
$$\log \left(\frac{1}{\mu_z(\R)}e^{-|\sqrt{A}(z,y)|^2/2}\right)=-\log(\mu_z(\R))-\frac12 |\sqrt{A}(z,y)|^2\mbox{\quad is a concave function in $y$.}$$
Applying Lemma \ref{computational_lemma} (1), we have
$$\partial_{yy}^2\left(-\log(\mu_z(\R))-\frac12 |\sqrt{A}(z,y)|^2\right)\leq 0.$$
Thus, $\widetilde{\mu}_z$ is a log-concave measure on $\R$. Let $F$ be a Borel set in $\R$ with $p:=\widetilde{\mu}_z(F)\in (0,1)$.
Since $\widetilde{\mu}_z(F^s)=\widetilde{\mu}_z(F)=p$, and $F^s=(-\infty,\phi^{-1}_z(\mu_z(F)))$ is a half-line, with the help of the one-dimensional log-concave isoperimetric inequality (see {\cite{Bobkov}}, Proposition 2.1), we have
$$\inf_{\widetilde{\mu}_z(A)\geq p}\widetilde{\mu}_z(A+B_h)= \widetilde{\mu}_z(F^s+B_h)$$
for all $h>0$, where $B_h=[-h,h]$. In particular, for any Borel set $F\subset \R$,
\begin{align}\label{log-concave_isoperimetric_inequality}
\mu_z(F+B_h)\geq \mu_z(F^s+B_h)\mbox{\quad for all $h>0$.}
\end{align}
Let $z\in B_E\cap B_{E^s}$, by Vol'pert Theorem (Theorem \ref{Vol'pert}), $E_z$ is a set of locally finite perimeter in $\R$. Moreover, by \cite{maggi}, Lemma 15.12, $E_z\cap (-R,R)$ is also a set of locally finite perimeter in $\R$ for any $R>0$. Applying \cite{maggi}, Proposition 12.13 on $E_z\cap (-R,R)$, we may assume that $E_z\cap (-R,R)$ is a disjoint union of open intervals with positive distance, i.e., $a_i<b_i<a_{i+1}<b_{i+1}$ for all $i$ and
\begin{align}\label{disjoint_union_of_open_intervals}
E_z\cap (-R,R)=\bigcup_{i\in S_R}(a_i,b_i),
\end{align}
where $S_R$ is a countable set. First we claim that $S_R$ is a finite set. By using equation (\ref{disjoint_union_of_open_intervals}),
\begin{align*}
\infty&>p_E(z)+e^{-|\sqrt{A}(z,R)|^2/2}+e^{-|\sqrt{A}(z,-R)|^2/2}=P_z(E_z)+e^{-|\sqrt{A}(z,R)|^2/2}+e^{-|\sqrt{A}(z,-R)|^2/2}\\
&\geq P_z(E_z\cap (-R,R))=P_z\left(\bigcup_{i\in S_R}(a_i,b_i)\right)=\sum_{i\in S_R}P_z\left((a_i,b_i)\right)=\sum_{i\in S_R}e^{-|\sqrt{A}(z,a_i)|^2/2}+e^{-|\sqrt{A}(z,b_i)|^2/2}\\
&\geq \sum_{i\in S_R}e^{-\|\sqrt{A}\|^2(|z|^2+R^2)/2}+e^{-\|\sqrt{A}\|^2(|z|^2+R^2)/2}\geq 2e^{-\|\sqrt{A}\|^2(|z|^2+R^2)/2}|S_R|,
\end{align*}
i.e., $|S_R|<\infty$, where we have used
$$|\sqrt{A}(z,a_i)|^2\leq\|\sqrt{A}\|^2(|z|^2+a_i^2)\leq \|\sqrt{A}\|^2(|z|^2+R^2).$$
Moreover, by the definition of $\mu_z$ and the fundamental theorem of calculus, for any $-\infty\leq a<b\leq \infty$,
\begin{align}\label{fundamental_theorem_of_calculus}
\lim_{h\to 0}\frac{\mu_z((a,b)+B_h)-\mu_z((a,b))}{h}&=e^{-|\sqrt{A}(z,a)|^2/2}+e^{-|\sqrt{A}(z,b)|^2/2}.
\end{align}
Next we claim that
$$P_z(E_z\cap (-R,R))\geq e^{-|\sqrt{A}(z,y_R(z))|^2/2},$$
where 
$$y_R(z):=\phi^{-1}_z\bigg(\mu_z\Big(\big(E\cap (\R^{n-1}\times(-R,R))\big)_z\Big)\bigg).$$
Notice that
$$\big(E\cap (\R^{n-1}\times(-R,R))\big)_z=E_z\cap (-R,R),\quad \big(E\cap (\R^{n-1}\times(-R,R))\big)^s_z=(-\infty,y_R(z)),$$
$$\mu_z\Big(\big(E\cap (\R^{n-1}\times(-R,R))\big)^s_z\Big)=\mu_z\Big(\big(E\cap (\R^{n-1}\times(-R,R))\big)_z\Big),$$
$$y_R(z)=\phi^{-1}_z\bigg(\mu_z\Big(E_z\cap (-R,R)\Big)\bigg)\to y(z):=\phi^{-1}_z(\mu_z(E_z)),\mbox{\quad as $R\to \infty$,}$$
and
$$\limsup_{R\to \infty}P_z(E_z\cap (-R,R))\leq \limsup_{R\to \infty}\Big(P_z(E_z)+e^{-|\sqrt{A}(z,R)|^2/2}+e^{-|\sqrt{A}(z,-R)|^2/2}\Big)= p_E(z).$$
By using equation (\ref{fundamental_theorem_of_calculus}), and the fact that the intervals $(a_i,b_i)$ are disjoint, 
\begin{align*}
&P_z(E_z\cap (-R,R))=\sum_{i\in S_R} e^{-|\sqrt{A}(z,a_i)|^2/2}+e^{-|\sqrt{A}(z,b_i)|^2/2}\\
&=\sum_{i\in S_R} \lim_{h\to 0^+}\frac{\mu_z((a_i,b_i)+B_h)-\mu_z((a_i,b_i))}{h}\\
&=\lim_{h\to 0^+}\sum_{i\in S_R} \frac{\mu_z((a_i,b_i)+B_h)-\mu_z((a_i,b_i))}{h}\mbox{\qquad ($S_R$ is finite)}\\
&\geq\lim_{h\to 0^+} \frac{1}{h}\left(\mu_z\bigg(\bigcup_{i\in S_R}((a_i,b_i)+B_h)\bigg)-\mu_z\bigg(\bigcup_{i\in S_R}(a_i,b_i)\bigg)\right)\\
&\geq \lim_{h\to 0^+}\frac{1}{h}\left(\mu_z\Big(E_z\cap (-R,R)+B_h\Big)-\mu_z\Big(E_z\cap (-R,R)\Big)\right)\\
&= \lim_{h\to 0^+}\frac{1}{h}\left(\mu_z\Big(\big(E\cap (\R^{n-1}\times(-R,R))\big)_z+B_h\Big)-\mu_z\Big(\big(E\cap ( \R^{n-1}\times(-R,R))\big)_z\Big)\right)\\
&\geq \lim_{h\to 0^+}\frac{1}{h}\left(\mu_z\Big(\big(E\cap (\R^{n-1}\times(-R,R))\big)^s_z+B_h\Big)-\mu_z\Big(\big(E\cap (\R^{n-1}\times(-R,R))\big)^s_z\Big)\right)\\
&=\lim_{h\to 0^+}\frac{1}{h}\int_{y_R(z)}^{y_R(z)+h}e^{-|\sqrt{A}(z,y)|^2/2}dy =e^{-|\sqrt{A}(z,y_R(z))|^2/2}
\end{align*}
in which we have used (\ref{log-concave_isoperimetric_inequality}) where $F=\big(E\cap (\R^{n-1}\times(-R,R))\big)_z$ in the last inequality.
Taking $R\to \infty$, we have $p_E(z)\geq e^{-|\sqrt{A}(z,y(z))|^2/2}=p_{E^s}(z).$
\end{proof}

We have seen that the anisotropic Gaussian perimeter of the one-dimensional section decreases by Ehrhard symmetrization. Intuitively, this gives us hope that the higher dimensional anisotropic Gaussian perimeter might also decrease after doing the Ehrhard symmetrization. However, our next example shows that this is not true in general. The main idea is to understand the asymptotic behavior of the quantity $h(z)=\phi^{-1}_z(v_E(z))$ via the equation $v_{E^s}(z)= v_E(z)$.

\begin{example}\label{counterexample_1}
Let $n=2$ and
$$A=2\left[ \begin{array}{rr} a & b \\ b & c \end{array} \right]\succ 0$$
with $b\neq 0$. Consider
$$E_\alpha=[-\alpha,\alpha]\times (0,\infty).$$
Then there exists some $\delta>0$ such that for any $0<\alpha<\delta$, we have
$$P_{\gamma_A}(E_\alpha)<P_{\gamma_A}(E^s_\alpha)$$
and
$$P_{\gamma_A}(E^s_\alpha)<P_{\gamma_A}(E_\alpha)+ \sqrt{2\pi}\|Ae_2-\langle Ae_2,e_2\rangle e_2\|\langle b_{\gamma_A}(E)-b_{\gamma_A}(E^s),e_2\rangle.$$
\end{example}
\begin{proof}
Notice that both $a$ and $c$ are positive since $A$ is a positive definite matrix. Additionally, notice that
$$e^{-|\sqrt{A}(x,y)|^2/2}=e^{-\langle A(x,y),(x,y)\rangle/2}=e^{-ax^2-2bxy-cy^2}. $$
Let $K=[-1,1]$ and $\Omega$ be an open set in $\R^1$ such that $K\subset \Omega$. 
Let $E=\Omega\times (0,\infty)$. Then
$$E_x=(0,\infty)$$
for all $x\in \Omega$. By Lemma \ref{computational_lemma} (2)(b),
$$v_E(x)=\int_{E_x}e^{-|\sqrt{A}(x,y)|^2/2}\ dy=\int_{0}^{\infty}e^{-ax^2-2bxy-cy^2}\ dy\mbox{\quad is differentiable on $\Omega$.}$$
By Lemma \ref{regularity_estimates_for_C1}, $h:x\mapsto \phi^{-1}_x(v_E(x))$ is also differentiable on $\Omega$. Notice that
\begin{align}\label{counterexample_1_eq0}
\int_{0}^{\infty}e^{-ax^2-2bxy-cy^2}\ dy=v_E(x)=v_{E^s}(x)=\int_{-\infty}^{h(x)}e^{-ax^2-2bxy-cy^2}\ dy.
\end{align}
Setting $x=0$ in equation (\ref{counterexample_1_eq0}), we have
$$v_E(0)=v_{E^s}(0)\implies \int_{0}^{\infty}e^{-cy^2}\ dy=\int_{-\infty}^{h(0)}e^{-cy^2}\ dy\implies h(0)=0.$$
Taking derivative on equation (\ref{counterexample_1_eq0}), we also have
\begin{align*}
\int_{0}^{\infty}e^{-ax^2-2bxy-cy^2}(-2ax-2by) dy=e^{-ax^2-2bxh(x)-ch^2(x)}h'(x)+\int_{-\infty}^{h(x)}e^{-ax^2-2bxy-cy^2}(-2ax-2by) dy.
\end{align*}
In particular for $x=0$,
$$h'(0)=\int_{0}^{\infty}e^{-cy^2}(-2by) dy-\int_{-\infty}^0e^{-cy^2}(-2by)dy=-4b\int_0^{\infty}e^{-cy^2}ydy=-\frac{2b}{c}.$$
That is,
\begin{align}\label{counerexample_1_eq1}
h(0)=0,\qquad h'(0)=-\frac{2b}{c}.
\end{align}
Notice that the Ehrhard symmetrization of $E_\alpha$ has the form
$$E^s_\alpha=\{(x,y):x\in [-\alpha,\alpha], y< h(x)\}\mbox{\quad for $0<\alpha<1$}.$$
We now claim that
\begin{align}\label{counerexample_1_eq2}
P_{\gamma_A}(E_\alpha)-P_{\gamma_A}(E_\alpha^s)&=2\frac{\sqrt{\det A}}{\sqrt{2\pi}}\left(1-\sqrt{1+\frac{4b^2}{c^2}}\right)\alpha+o(\alpha).
\end{align}
By using the Taylor expansion, we have
\begin{align*}
\frac{\sqrt{2\pi}}{\sqrt{\det A}}P_{\gamma_A}(E_\alpha)&=\int_{0}^{\infty}e^{-a\alpha^2-2b\alpha y-cy^2}\ dy+\int_{0}^{\infty}e^{-a\alpha^2+2b\alpha y-cy^2}\ dy+\int_{-\alpha}^{\alpha}e^{-ax^2}\ dx\\
&=\left(\int_{0}^{\infty}e^{-cy^2}\ dy+\frac{d}{d\alpha}\int_{0}^{\infty}e^{-a\alpha^2-2b\alpha y-cy^2}\ dy\Bigg|_{\alpha=0}\alpha+o(\alpha)\right)\\
&\quad+\left(\int_{0}^{\infty}e^{-cy^2}\ dy+\frac{d}{d\alpha}\int_{0}^{\infty}e^{-a\alpha^2+2b\alpha y-cy^2}\ dy\Bigg|_{\alpha=0}\alpha+o(\alpha)\right)+2\int_{0}^{\alpha}e^{-ax^2}\ dx\\
&=\left(\frac{\sqrt{\pi}}{2\sqrt{c}}+\left(\frac{-b}{c}\right)\alpha+o(\alpha)\right)+\left(\frac{\sqrt{\pi}}{2\sqrt{c}}+\left(\frac{b}{c}\right)\alpha+o(\alpha)\right)+\left(2\alpha+o(\alpha)\right)\\
&=\frac{\sqrt{\pi}}{\sqrt{c}}+2\alpha+o(\alpha).
\end{align*}
Moreover, by (\ref{counerexample_1_eq1}), we have
\begin{align*}
\frac{\sqrt{2\pi}}{\sqrt{\det A}}P_{\gamma_A}(E^s_\alpha)&=\int_{-\infty}^{h(\alpha)}e^{-a\alpha^2-2b\alpha y-cy^2}\ dy+\int_{-\infty}^{h(-\alpha)}e^{-a\alpha^2+2b\alpha y-cy^2}\ dy+\int_{-\alpha}^{\alpha}e^{-at^2-2bt h(t)-ch^2(t)}\sqrt{1+h'(t)^2}\ dt\\
&=\left(\int_{-\infty}^{h(0)}e^{-cy^2}\ dy+\frac{d}{d\alpha}\int_{-\infty}^{h(\alpha)}e^{-a\alpha^2-2b\alpha y-cy^2}\ dy\Bigg|_{\alpha=0}\alpha+o(\alpha)\right)\\
&\quad+\left(\int_{-\infty}^{h(0)}e^{-cy^2}\ dy+\frac{d}{d\alpha}\int_{-\infty}^{h(-\alpha)}e^{-a\alpha^2+2b\alpha y-cy^2}\ dy\Bigg|_{\alpha=0}\alpha+o(\alpha)\right)\\
&\quad+\left(\frac{d}{d\alpha}\int_{-\alpha}^{\alpha}e^{-at^2-2bt h(t)-ch^2(t)}\sqrt{1+h'(t)^2}\ dt\Bigg|_{\alpha=0}\alpha+o(\alpha)\right)\\
&=\left(\frac{\sqrt{\pi}}{2\sqrt{c}}+\alpha\left[-\frac{2b}{c}+\frac{b}{c}\right]+o(\alpha)\right)+\left(\frac{\sqrt{\pi}}{2\sqrt{c}}+\alpha\left[\frac{2b}{c}-\frac{b}{c}\right]+o(\alpha)\right)+\left(2\alpha\sqrt{1+\frac{4b^2}{c^2}}+o(\alpha)\right)\\
&=\frac{\sqrt{\pi}}{\sqrt{c}}+2\alpha\sqrt{1+\frac{4b^2}{c^2}}+o(\alpha),
\end{align*}
where the third term is the anisotropic Gaussian perimeter of the graph of $h$. Therefore,
\begin{align*}
P_{\gamma_A}(E_\alpha)-P_{\gamma_A}(E_\alpha^s)&=2\frac{\sqrt{\det A}}{\sqrt{2\pi}}\left(1-\sqrt{1+\frac{4b^2}{c^2}}\right)\alpha+o(\alpha).
\end{align*}
Next we show that
\begin{align}\label{counerexample_1_eq3}
\sqrt{2\pi}\|Ae_2-\langle Ae_2,e_2\rangle e_n\|\langle b_{\gamma_A}(E)-b_{\gamma_A}(E^s),e_2\rangle=\frac{\sqrt{\det A}}{\sqrt{2\pi}}2\left(\frac{2b}{c}\right)\alpha+o(\alpha).
\end{align}
Notice that
$$\|Ae_2-\langle Ae_2,e_2\rangle e_2\|=2|b|.$$
Now we compute the barycenter of $E_\alpha$ and $E_\alpha^s$, i.e., 
\begin{align*}
b_{\gamma_A}(E_\alpha)&=\frac{\sqrt{\det A}}{2\pi}\int_{-\alpha}^{\alpha}\int_0^{\infty}(x,y)e^{-|\sqrt{A}(x,y)|^2/2}dy dx,
\end{align*}
and
\begin{align*}
b_{\gamma_A}(E^s_\alpha)&=\frac{\sqrt{\det A}}{2\pi}\int_{-\alpha}^{\alpha}\int_{-\infty}^{h(x)}(x,y)e^{-|\sqrt{A}(x,y)|^2/2}dy dx.
\end{align*}
Then
\begin{align*}
b_{\gamma_A}(E_\alpha)-b_{\gamma_A}(E^s_\alpha)&=\frac{\sqrt{\det A}}{2\pi}\int_{-\alpha}^{\alpha}\left(\int_{0}^{\infty}(x,y)e^{-|\sqrt{A}(x,y)|^2/2}dy-\int_{-\infty}^{h(x)}(x,y)e^{-|\sqrt{A}(x,y)|^2/2}dy\right) dx
\end{align*}
and hence
\begin{align*}
\langle b_{\gamma_A}(E_\alpha)-b_{\gamma_A}(E^s_\alpha),e_2\rangle&=\frac{\sqrt{\det A}}{2\pi}\int_{-\alpha}^{\alpha}\left(\int_{0}^{\infty}ye^{-|\sqrt{A}(x,y)|^2/2}dy-\int_{-\infty}^{h(x)}ye^{-|\sqrt{A}(x,y)|^2/2}dy\right) dx\\
&=\frac{\sqrt{\det A}}{2\pi}\left[\frac{d}{d\alpha}\int_{-\alpha}^{\alpha}\left(\int_{0}^{\infty}ye^{-|\sqrt{A}(x,y)|^2/2}dy-\int_{-\infty}^{h(x)}ye^{-|\sqrt{A}(x,y)|^2/2}dy\right) dx\Bigg|_{\alpha=0}\alpha+o(\alpha)\right]\\
&=\frac{\sqrt{\det A}}{2\pi}\left[2\left(\int_{0}^{\infty}ye^{-|\sqrt{A}(0,y)|^2/2}dy-\int_{-\infty}^{h(0)}ye^{-|\sqrt{A}(0,y)|^2/2}dy\right)\alpha+o(\alpha)\right]\\
&=\frac{\sqrt{\det A}}{2\pi}\left[4\left(\int_{0}^{\infty}ye^{-cy^2}dy\right)\alpha+o(\alpha)\right]=\frac{\sqrt{\det A}}{2\pi}\left(\frac{2}{c}\right)\alpha+o(\alpha).\\
\end{align*}
Thus,
$$\sqrt{2\pi}\|Ae_2-\langle Ae_2,e_2\rangle e_2\|\langle b_{\gamma_A}(E)-b_{\gamma_A}(E^s),e_2\rangle=2\frac{\sqrt{\det A}}{\sqrt{2\pi}}\left(\frac{2|b|}{c}\right)\alpha+o(\alpha).$$
Thanks to (\ref{counerexample_1_eq2}) and (\ref{counerexample_1_eq3}), we have
$$
\lim_{\alpha\to 0^+}\frac{P_{\gamma_A}(E_\alpha)-P_{\gamma_A}(E_\alpha^s)}{\alpha}=2\frac{\sqrt{\det A}}{\sqrt{2\pi}}\left(1-\sqrt{1+\frac{4b^2}{c^2}}\right)<0
$$
and
\begin{align*}
&\lim_{\alpha\to 0^+}\frac{\sqrt{2\pi}\|Ae_2-\langle Ae_2,e_2\rangle e_2\|\langle b_{\gamma_A}(E)-b_{\gamma_A}(E^s),e_2\rangle +P_{\gamma_A}(E_\alpha)-P_{\gamma_A}(E_\alpha^s)}{\alpha}\\
&=2\frac{\sqrt{\det A}}{\sqrt{2\pi}}\left(\left(1+\frac{2|b|}{c}\right)-\sqrt{1+\frac{4b^2}{c^2}}\right)>0.
\end{align*}
Since $b\not=0$, let
$$\ep=\min\left\{\frac{\sqrt{\det A}}{\sqrt{2\pi}}\left(\sqrt{1+\frac{4b^2}{c^2}}-1\right),\frac{\sqrt{\det A}}{\sqrt{2\pi}}\left(\left(1+\frac{2|b|}{c}\right)-\sqrt{1+\frac{4b^2}{c^2}}\right)\right\}>0.$$
There exists $\delta>0$ such that for all $0<\alpha<\delta$,
$$\left|\frac{P_{\gamma_A}(E_\alpha)-P_{\gamma_A}(E_\alpha^s)}{\alpha}-2\frac{\sqrt{\det A}}{\sqrt{2\pi}}\left(1-\sqrt{1+\frac{4b^2}{c^2}}\right)\right|<\ep\leq \frac{\sqrt{\det A}}{\sqrt{2\pi}}\left(\sqrt{1+\frac{4b^2}{c^2}}-1\right)$$
and
\begin{align*}
&\left|\frac{\sqrt{2\pi}\|Ae_2-\langle Ae_2,e_2\rangle e_2\|\langle b_{\gamma_A}(E)-b_{\gamma_A}(E^s),e_2\rangle +P_{\gamma_A}(E_\alpha)-P_{\gamma_A}(E_\alpha^s)}{\alpha}-2\frac{\sqrt{\det A}}{\sqrt{2\pi}}\left(\left(1+\frac{2|b|}{c}\right)-\sqrt{1+\frac{4b^2}{c^2}}\right)\right|\\
&<\ep\leq \frac{\sqrt{\det A}}{\sqrt{2\pi}}\left(\left(1+\frac{2|b|}{c}\right)-\sqrt{1+\frac{4b^2}{c^2}}\right).
\end{align*}
Therefore, for any $0<\alpha<\delta$,
$$
P_{\gamma_A}(E_\alpha)-P_{\gamma_A}(E_\alpha^s)<\frac{\sqrt{\det A}}{\sqrt{2\pi}}\left(1-\sqrt{1+\frac{4b^2}{c^2}}\right)\alpha<0
$$
and
$$
\sqrt{2\pi}\|Ae_2-\langle Ae_2,e_2\rangle e_2\|\langle b_{\gamma_A}(E)-b_{\gamma_A}(E^s),e_2\rangle +P_{\gamma_A}(E_\alpha)-P_{\gamma_A}(E_\alpha^s)>\frac{\sqrt{\det A}}{\sqrt{2\pi}}\left(\left(1+\frac{2|b|}{c}\right)-\sqrt{1+\frac{4b^2}{c^2}}\right)\alpha>0,
$$
i.e.,
$$P_{\gamma_A}(E^s_\alpha)<P_{\gamma_A}(E_\alpha)+\sqrt{2\pi}\|Ae_2-\langle Ae_2,e_2\rangle e_2\|\langle b_{\gamma_A}(E)-b_{\gamma_A}(E^s),e_2\rangle.$$
\end{proof}

\noindent{\bf Remark}: From Example \ref{counterexample_1}, we see that there exists some $E$ such that
$$P_{\gamma_A}(E)<P_{\gamma_A}(E^s).$$
Although the anisotropic Gaussian perimeter does not always decrease under Ehrhard symmetrization, a natural question to ask here is whether there still exists an upper bound for $P_{\gamma_A}(E^s)$ in terms of $P_{\gamma_A}(E)$. In our next subsection, we will show that there exists an upper bound and
\begin{align*}
P_{\gamma_A}\left(E^{s} \right) &\leq P_{\gamma_A}(E)+ \sqrt{2\pi}\|Ae_n-\langle Ae_n,e_n\rangle e_n\|\langle b_{\gamma_A}(E)-b_{\gamma_A}(E^s),e_n\rangle,
\end{align*}
for any set of finite anisotropic Gaussian perimeter $E$ in $\mathbb{R}^{n}$ (see Theorem \ref{Ehrhard_Sym_Ineq_I}).

\subsection{Ehrhard symmetrization on anisotropic Gaussian measures}\mbox{}\vspace{-.22cm}\\

Our next goal is to show that the perimeter of a Ehrhard symmetrization set $E^s$ can still be controlled by the perimeter of $E$ plus an error term with a form like $A-\lambda I_n$. In particular, the Ehrhard symmetrization along the eigendirections decreases the anisotropic Gaussian perimeter. We will break this into several lemmas. Our next three lemmas are modifications of Cianchi-Fusco-Maggi-Pratelli's paper \cite{Cianchi-Fusco-Maggi-Pratelli} Lemma 4.5 and Lemma 4.6. We will prove the ``dust estimate'', ``cylindrical estimate'', and ``graphical estimate''. Starting from this section, the notation $C$ means a constant that depends only on $n$ and $A$, which may change from line to line.

\begin{lemma}[Dust estimate for $E$]\label{dust_estimate}\mbox{}\\
Let $E$ be a set of finite anisotropic Gaussian perimeter in $\mathbb{R}^{n}$ and $B$ be a measurable set such that
$$v_E(z)=0\mbox{\quad for all $z\in B$.}$$
Then
\begin{align*}
P_{\gamma_A}(E^s;B\times \R)&\leq P_{\gamma_A}\left(E; B \times \mathbb{R}\right)\\
&\quad+\frac{\det\sqrt{A}}{(2\pi)^{(n-1)/2}}\int_{B}\left|\int_{E_z} \nabla^{\prime}\left(e^{-|\sqrt{A}x|^2/2}\right)dy-\int_{E^s_z} \nabla^{\prime}\left(e^{-|\sqrt{A}x|^2/2}\right)dy\right| dz.
\end{align*}
In particular, if we assume that $B$ is open, then
$$P_{\gamma_A}(E^s;B\times \R)=0.$$
\end{lemma}
\begin{proof}
Without loss of generality, we may assume that $B$ is bounded since we can consider $B\cap B(0,R)$. Given $\ep>0$ and let $\Omega$ be an open set with $\Omega\supset B$ such that
$$\mathcal{L}^{n-1}(\Omega\slash B)<\ep.$$
By \cite{Daners}, Proposition 8.2.1, there exists a sequence of bounded open smooth sets $\Omega_j\nearrow\Omega$, i.e., $\Omega_{j}\subset\subset  \Omega_{j+1}\subset \subset  \Omega$ and $\bigcup_{j}\Omega_j=\Omega$. Then
$$P_{\gamma_A}(E^s;\Omega_j\times \R)\nearrow P_{\gamma_A}(E^s;\Omega\times \R)$$
and
$$P_{\gamma_A}(E^s;B\times \R)\leq P_{\gamma_A}(E^s;\Omega\times \R)$$
since $B\subset \Omega$. Recall from Proposition \ref{lower_semicontinuity} that in order to compute $P_{\gamma_A}(E^s,\Omega_j\times \R)$ it is enough to look at (\ref{perimeter_inside_cylinder}) and equation (\ref{(I) and (II)}) as we have seen in Theorem \ref{approximation_theorem_for_BV_maps}, i.e., let $\varphi\in C^1_c(\Omega_j\times\R;\R^n)$ with $|\varphi|\leq 1$ and let $\varphi_z=(\varphi_1,\varphi_2,...,\varphi_{n-1})$, we estimate integrals (I) and (II) from (\ref{(I) and (II)}).

Applying Theorem \ref{approximation_theorem_for_BV_maps} on $\Omega_j$ and $y(z):=\phi_{z}^{-1}(v_E(z))$, there exists a sequence of functions $v_{k}^j \in C^{1}(\Omega_j)$, such that $v_{k}^j \rightarrow v_{E}$ in $L^{1}(\Omega_j)$ and a.e. in $\Omega_j$, $Dv_k^j\wkly Dv_E$ in $\Omega_j$, and
$$
\lim _{k \rightarrow \infty} \int_{\Omega_j}\left|\nabla' v^j_{k}(z)\right| d z=\left|D v_{E}\right|(\Omega_j).
$$
Moreover,
$$\mbox{$\mbox{$\chi_{F_{k}^j} \rightarrow \chi_{E^{s}}$\ a.e. in $\Omega_j\times \R$}$ and $y_k^j(z)\to y(z)=\phi_{z}^{-1}(v_E(z))$ a.e. in $\Omega_j$}$$
where $y_k^j(z):=\phi_{z}^{-1}(v^j_k(z))$ and
$$
F_k^j:= \left\{(z, y) \in \Omega_j \times \mathbb{R}: y<y_k^j(z):=\phi_{z}^{-1}(v^j_k(z))\right\}.
$$
Since $v_E(z)=0$ for all $z\in B$, we have
$$v_k^j(z)\to v_E(z)=0\mbox{\ a.e. in $\Omega_j\cap B$,}$$
$$y_k^j(z)\to y(z)=\phi_{z}^{-1}(v_E(z))=-\infty\mbox{\ a.e. in $\Omega_j\cap B$,}$$
and
$$p_{E^s}(z)=e^{-[\phi_z^{-1}(v_E(z))]^2/2}=0\mbox{\ on $B$.}$$
Therefore, by equation (\ref{locally_finite_perimeter_estimate_eq1}) where $F$ as $F_k^j$, $y(z)$ as $y_k^j(z)$, and $\Omega$ as $\Omega_j$,
\begin{align}
&\text{(I)}=\sum_{i=1}^{n-1} \int_{E^s} \frac{\partial \varphi_{i}}{\partial z_{i}}+\varphi_{i}\frac{\partial (-|\sqrt{A}x|^2/2)}{\partial z_i}\ d \gamma_{A}(x)=\lim _{k \rightarrow \infty}\Big( \sum_{i=1}^{n-1} \int_{F^j_{k}} \frac{\partial \varphi_{i}}{\partial z_{i}}+\varphi_{i}\frac{\partial (-|\sqrt{A}x|^2/2)}{\partial z_i}\ d \gamma_{A}(x)\Big)\notag\\
&=\frac{\det \sqrt{A}}{(2 \pi)^{n / 2}} \lim _{k \rightarrow \infty}\Bigg(\int_{\Omega_j}\left(-\varphi_z(z,y_k^j(z))\right)\cdot \nabla^{\prime} v_k^j(z) d z-\int_{\Omega_j}\int_{-\infty}^{y_k^j(z)}\left(-\varphi_z(z,y_k^j(z))\right)\cdot \nabla^{\prime}\left(e^{-|\sqrt{A}x|^2/2}\right)dydz\Bigg)\notag\\
&\leq  \frac{\det \sqrt{A}}{(2 \pi)^{n / 2}}\limsup_{k\to \infty}\int_{\Omega_j}\left| \nabla^{\prime} v_k^j(z) \right| d z\notag+ \frac{\det \sqrt{A}}{(2 \pi)^{n / 2}}\lim _{k \rightarrow \infty}\int_{\Omega_j\slash B}\int_{-\infty}^{y_k^j(z)}\varphi_z(z,y_k^j(z))\cdot \nabla^{\prime}\left(e^{-|\sqrt{A}x|^2/2}\right)dydz\notag\\
&\quad +\frac{\det \sqrt{A}}{(2 \pi)^{n / 2}}\lim _{k \rightarrow \infty}\int_{\Omega_j\cap B}\int_{-\infty}^{y_k^j(z)}\varphi_z(z,y_k^j(z))\cdot \nabla^{\prime}\left(e^{-|\sqrt{A}x|^2/2}\right)dydz\notag\\
&= \frac{\det \sqrt{A}}{(2 \pi)^{n / 2}}\left|D v_{E}\right|(\Omega_j)+ \frac{\det \sqrt{A}}{(2 \pi)^{n / 2}}\int_{\Omega_j\slash B}\int_{-\infty}^{y(z)}\varphi_z(z,y(z))\cdot \nabla^{\prime}\left(e^{-|\sqrt{A}x|^2/2}\right)dydz\notag\\
&\leq \frac{\det \sqrt{A}}{(2 \pi)^{n / 2}}\left|D v_{E}\right|(\Omega_j)+ \frac{\det \sqrt{A}}{(2 \pi)^{n / 2}}\int_{\Omega_j\slash B}\int_{-\infty}^{\infty}\left| \nabla^{\prime}\left(e^{-|\sqrt{A}x|^2/2}\right)\right|dydz\leq \frac{\det \sqrt{A}}{(2 \pi)^{n / 2}}\left|D v_{E}\right|(\Omega_j)+C\ep\notag
\end{align}
where we have used $\mathcal{L}^{n-1}(\Omega_j\slash B)<\ep$ and Lemma \ref{computational_lemma} (3), i.e.,
\begin{align*}
\sup_{z\in \R^{n-1}}\int_{-\infty}^{\infty}\left| \nabla^{\prime}\left(e^{-|\sqrt{A}x|^2/2}\right)\right| dy&\leq (n-1)C.
\end{align*}
Therefore,
$$\text{(I)}\leq \frac{\det \sqrt{A}}{(2 \pi)^{n / 2}}\left|D v_{E}\right|(\Omega_j)+C\ep.$$
Now we estimate (II). Recalling that $\mathcal{H}^{n-1}(\partial^M E^s\slash \rb E^s)=0$ and $(\partial^M E^s)_z=\partial^M((E^s)_z)$ for $\mathcal{L}^{n-1}$-a.e. $z\in \R^{n-1}$, we get
\begin{align}
&\text{(II)}=\int_{E^s} \frac{\partial \varphi_{n}}{\partial y}+\varphi_{n}\frac{\partial (-|\sqrt{A}x|^2/2)}{\partial y} \ d \gamma_{A}(x)=\frac{\det \sqrt{A}}{(2\pi)^{n/2}}\int_{E^{s}} \frac{\partial}{\partial y}\left(\varphi_{n} e^{-|\sqrt{A}x|^2/2}\right) d z\notag \\
&=-\frac{\det \sqrt{A}}{(2\pi)^{n/2}}\int_{\partial^{M} E^{s}} \varphi_{n}e^{-|\sqrt{A}x|^2/2} \nu_{n}^{E^s}\ d \mathcal{H}^{n-1}(x)\leq \frac{\det \sqrt{A}}{(2\pi)^{n/2}}\int_{\partial^{M} E^{s} \cap(\Omega_j \times \mathbb{R})}\left|\nu_{n}^{E^s}\right|e^{-|\sqrt{A}x|^2/2}\ d \mathcal{H}^{n-1}(x) \notag\\
&=\frac{\det \sqrt{A}}{(2\pi)^{n/2}}\int_{\Omega_j}  \int_{\left(\partial^{M} E^s\right)_{z}} e^{-|\sqrt{A}x|^2/2}\ d \mathcal{H}^{0}(y) \ d z=\frac{\det \sqrt{A}}{(2\pi)^{n/2}}\int_{\Omega_j}p_{E^s}(z)\ d z=\frac{\det \sqrt{A}}{(2\pi)^{n/2}}\int_{\Omega_j\slash B}p_{E^s}(z)\ d z\leq C\ep\notag
\end{align}
where we have used the co-area formula (\ref{co-area formula for sets of finite perimeter_1}), the definition of $p_{E^s}(z)$, $p_{E^s}(z)=0$ if $z\in B$, and
$$p_{E^s}(z)=e^{-[\phi_z^{-1}(v_E(z))]^2/2}\leq 1.$$
Combining (I) and (II) together,
\begin{align*}
\int_{E^{s}} \operatorname{div} \varphi-\langle\varphi, Ax\rangle\ d \gamma_{A}(x) &= \text{(I)}+\text{(II)}\leq \frac{\det \sqrt{A}}{(2 \pi)^{n / 2}}\left|D v_{E}\right|(\Omega_j)+C\ep.
\end{align*}
Taking the sup over $\varphi$ gives us
\begin{align*}
\frac{1}{\sqrt{2\pi}}P_{\gamma_A}\left(E^{s} ; \Omega_j \times \mathbb{R}\right)  &\leq \frac{\det \sqrt{A}}{(2 \pi)^{n / 2}}\left|D v_{E}\right|(\Omega_j)+C\ep.
\end{align*}
Applying Lemma \ref{distributional_derivative_formula} on $\Omega_j$,
\begin{align*}
P_{\gamma_A}\left(E^{s} ; \Omega_j \times \mathbb{R}\right) &\leq \frac{\det \sqrt{A}}{(2 \pi)^{(n-1) / 2}}\left|D v_{E}\right|(\Omega_j)+C\ep\\ 
&\leq P_{\gamma_A}\left(E ; \Omega_j \times \mathbb{R}\right)+\frac{\det\sqrt{A}}{(2\pi)^{(n-1)/2}}\int_{\Omega_j}\left|\int_{E_z}\nabla^{\prime}\left(e^{-|\sqrt{A}x|^2/2}\right)dy\right| dz+C\ep\\
&\leq P_{\gamma_A}\left(E ; \Omega_j \times \mathbb{R}\right)+\frac{\det\sqrt{A}}{(2\pi)^{(n-1)/2}}\int_{B}\left|\int_{E_z}\nabla^{\prime}\left(e^{-|\sqrt{A}x|^2/2}\right)dy\right| dz\\
&\quad +\frac{\det\sqrt{A}}{(2\pi)^{(n-1)/2}}\int_{\Omega_j\slash B}\left|\int_{E_z}\nabla^{\prime}\left(e^{-|\sqrt{A}x|^2/2}\right)dy\right| dz+C\ep\\
&\leq P_{\gamma_A}\left(E ; \Omega_j \times \mathbb{R}\right)+\frac{\det\sqrt{A}}{(2\pi)^{(n-1)/2}}\int_{B}\left|\int_{E_z}\nabla^{\prime}\left(e^{-|\sqrt{A}x|^2/2}\right)dy\right| dz+C\ep\\
&= P_{\gamma_A}\left(E ; \Omega_j \times \mathbb{R}\right)+\frac{\det\sqrt{A}}{(2\pi)^{(n-1)/2}}\int_{B}\left|\int_{E_z}\nabla^{\prime}\left(e^{-|\sqrt{A}x|^2/2}\right)dy-\int_{E^s_z}\nabla^{\prime}\left(e^{-|\sqrt{A}x|^2/2}\right)dy\right| dz+C\ep
\end{align*}
where we have used $E^s_z=\left(-\infty,\phi_z^{-1}(v_E(z))\right)=\emptyset$ if $z\in B$, and Lemma \ref{computational_lemma} (3)(a), i.e.,
\begin{align*}
\int_{\Omega_j\slash B}\left|\int_{E_z}\nabla^{\prime}\left(e^{-|\sqrt{A}x|^2/2}\right)dy\right| dz\leq \int_{\Omega_j\slash B}\int_{-\infty}^{\infty}\left|\nabla^{\prime}\left(e^{-|\sqrt{A}x|^2/2}\right)\right|dy dz\leq (n-1)C\ep.
\end{align*}
Taking $j\to \infty$ on both sides,
\begin{align*}
P_{\gamma_A}\left(E^{s} ; B \times \mathbb{R}\right)&\leq P_{\gamma_A}\left(E^{s} ; \Omega \times \mathbb{R}\right) \\
&\leq P_{\gamma_A}\left(E; \Omega \times \mathbb{R}\right)+\frac{\det\sqrt{A}}{(2\pi)^{(n-1)/2}}\int_{B}\left|\int_{E_z} \nabla^{\prime}\left(e^{-|\sqrt{A}x|^2/2}\right)dy-\int_{E^s_z} \nabla^{\prime}\left(e^{-|\sqrt{A}x|^2/2}\right)dy\right| dz+C\ep.
\end{align*}
Taking the inf over $\Omega\supset B$ with $\mathcal{L}^{n-1}(\Omega\slash B)<\ep$,
\begin{align*}
P_{\gamma_A}\left(E^{s} ; B \times \mathbb{R}\right)
&\leq P_{\gamma_A}\left(E; B \times \mathbb{R}\right) +\frac{\det\sqrt{A}}{(2\pi)^{(n-1)/2}}\int_{B}\left|\int_{E_z} \nabla^{\prime}\left(e^{-|\sqrt{A}x|^2/2}\right)dy-\int_{E^s_z} \nabla^{\prime}\left(e^{-|\sqrt{A}x|^2/2}\right)dy\right| dz+C\ep.
\end{align*}
Taking $\ep\to 0$,
\begin{align*}
P_{\gamma_A}\left(E^{s} ; B \times \mathbb{R}\right)\leq P_{\gamma_A}\left(E; B \times \mathbb{R}\right)+\frac{\det\sqrt{A}}{(2\pi)^{(n-1)/2}}\int_{B}\left|\int_{E_z} \nabla^{\prime}\left(e^{-|\sqrt{A}x|^2/2}\right)dy-\int_{E^s_z} \nabla^{\prime}\left(e^{-|\sqrt{A}x|^2/2}\right)dy\right| dz.
\end{align*}
Now we claim that
$$P_{\gamma_A}\left(E^{s} ; B \times \mathbb{R}\right)=0,\quad\mbox{if $B$ is open and $v_E(z)=0$ for all $z\in B$.}$$
In this situation, we can just apply all the previous estimates on $B$ instead of $\Omega_j$, which means that we can replace $v_{k}^j$ as $v_k$, $y_k^j$ as $y_k$, and $F_k^j$ as $F_k$ in the previous calculation and notice that
$$\mbox{$y_k(z) \rightarrow y(z)=\phi_z^{-1}\left(v_E(z)\right)=-\infty$ a.e. in $B$,}$$
and
$$\mbox{$\varphi_z\left(z, y_k(z)\right)=(0, \ldots, 0)$ for a.e. $z \in B$ and large enough $k$,}$$ 
since $\varphi$ has compact support and $y_k(z) \rightarrow-\infty$. Now (I) and (II) become
\begin{align*}
&(\mathrm{I})=\sum_{i=1}^{n-1} \int_{E^s} \frac{\partial \varphi_i}{\partial z_i}+\varphi_i \frac{\partial ((-|\sqrt{A}x|^2/2))}{\partial z_i} d \gamma_A(x)= \lim _{k \rightarrow \infty}\left(\sum_{i=1}^{n-1} \int_{F_k} \frac{\partial \varphi_i}{\partial z_i}+\varphi_i \frac{\partial ((-|\sqrt{A}x|^2/2))}{\partial z_i} d \gamma_A(x)\right) \\
&= \frac{\operatorname{det} \sqrt{A}}{(2 \pi)^{n / 2}} \lim _{k \rightarrow \infty}\left(\int_{B}\left(-\varphi_z\left(z, y_k(z)\right)\right) \cdot \nabla^{\prime} v_k(z) d z-\int_{B} \int_{-\infty}^{y_k(z)}\left(-\varphi_z\left(z, y_k(z)\right)\right) \cdot \left(\nabla^{\prime} e^{-|\sqrt{A}x|^2/2} \right)d y d z\right) =0
\end{align*}
and
$$(\mathrm{II})\leq\frac{\det \sqrt{A}}{(2\pi)^{n/2}}\int_{B}p_{E^s}(z)\ d z=\frac{\det \sqrt{A}}{(2\pi)^{n/2}}\int_{B}e^{-[\phi_z^{-1}(v_E(z))]^2/2}\ d z =0.$$
Thus, we conclude that
$$P_{\gamma_A}(E^s;B\times \R)=0.$$
\end{proof}

\begin{lemma}[Cylindrical estimate for $E$]\label{cylindrical_estimate}\mbox{}\\
Let $E$ be a set of finite anisotropic Gaussian perimeter in $\mathbb{R}^{n}$ and $B$ be a measurable set in $\R^{n-1}$ with $\mathcal{L}^{n-1}(B)=0$. Then
$$P_{\gamma_A}(E^s;B\times \R)\leq P_{\gamma_A}(E;B\times\R).$$
\end{lemma}
\begin{proof} Given $\ep>0$ and let $\Omega$ be an open set with $\Omega\supset B$ such that$$\mathcal{L}^{n-1}(\Omega)<\ep.$$
The proof of this lemma is similar to Lemma \ref{dust_estimate}, i.e., we approximate our set $\Omega$ by a sequence of bounded open subsets $\Omega_j$. Hence, we keep the same notation as we have seen in Lemma \ref{dust_estimate}, $\Omega_j$, $v_k^j$, and $y_k^j$ etc. The only difference here is the estimates of (I) and (II) from Lemma \ref{dust_estimate}.
\begin{align}
&\text{(I)}=\sum_{i=1}^{n-1} \int_{E^s} \frac{\partial \varphi_{i}}{\partial z_{i}}+\varphi_{i}\frac{\partial (-|\sqrt{A}x|^2/2)}{\partial z_i}\ d \gamma_{A}(x)=\lim _{k \rightarrow \infty}\Big( \sum_{i=1}^{n-1} \int_{F_{k}} \frac{\partial \varphi_{i}}{\partial z_{i}}+\varphi_{i}\frac{\partial (-|\sqrt{A}x|^2/2)}{\partial z_i}\ d \gamma_{A}(x)\Big)\notag\\
&=-\frac{\det \sqrt{A}}{(2 \pi)^{n / 2}} \lim _{k \rightarrow \infty}\Bigg(\int_{\Omega_j}\varphi_z(z,y_k^j(z))\cdot \nabla^{\prime} v_k^j(z) d z-\int_{\Omega_j}\int_{-\infty}^{y_k^j(z)}\varphi_z(z,y_k^j(z))\cdot \nabla^{\prime}\left(e^{-|\sqrt{A}x|^2/2}\right)dydz\Bigg)\notag\\
&\leq \frac{\det \sqrt{A}}{(2 \pi)^{n / 2}}\lim_{k\to \infty}\int_{\Omega_j}\left|\nabla' v^j_{k}(z)\right| d z + \frac{\det \sqrt{A}}{(2 \pi)^{n / 2}}\int_{\Omega_j}\int_{-\infty}^{y(z)}\varphi_z(z,y(z))\cdot \nabla^{\prime}\left(e^{-|\sqrt{A}x|^2/2}\right)dydz\notag\\
&\leq \frac{\det \sqrt{A}}{(2 \pi)^{n / 2}}\left|D v_{E}\right|(\Omega_j)+\frac{\det \sqrt{A}}{(2 \pi)^{n / 2}}\int_{\Omega_j}\int_{-\infty}^{\infty}\left| \nabla^{\prime}\left(e^{-|\sqrt{A}x|^2/2}\right)\right| dydz\leq \frac{\det \sqrt{A}}{(2 \pi)^{n / 2}}\left|D v_{E}\right|(\Omega_j)+C\mathcal{L}^{n-1}(\Omega_j)\notag\\
&\leq \frac{\det \sqrt{A}}{(2 \pi)^{n / 2}}\left|D v_{E}\right|(\Omega_j)+C\ep\notag
\end{align}
where we have used $\mathcal{L}^{n-1}(\Omega_j)<\ep$ and Lemma \ref{computational_lemma} (3)(a). Moreover, applying Lemma \ref{distributional_derivative_formula} on $\Omega_j$ and Lemma \ref{computational_lemma} (3)(a),
$$\frac{\det \sqrt{A}}{(2 \pi)^{(n-1)/ 2}}|Dv_E|(\Omega_j)\leq P_{\gamma_A}(E;\Omega_j\times \R)+C\ep.$$
Therefore,
$$\text{(I)}\leq \frac{1}{\sqrt{2\pi}}P_{\gamma_A}(E;\Omega_j\times \R)+C\ep.$$
The estimate of (II) is the same as Lemma \ref{dust_estimate}, i.e.,
\begin{align}
&\text{(II)}=\frac{\det \sqrt{A}}{(2\pi)^{n/2}}\int_{\Omega_j}p_{E^s}(z)\ d z\leq C\ep\notag
\end{align}
where we have used the definition of $p_{E^s}(z)$, and $\mathcal{L}^{n-1}(\Omega_j)<\ep$. Combining (I) and (II) together,
\begin{align*}
\int_{E^{s}} \operatorname{div} \varphi-\langle\varphi, Ax\rangle\ d \gamma_{A}(x) &= \text{(I)}+\text{(II)}\leq \frac{1}{\sqrt{2\pi}}P_{\gamma_A}(E;\Omega_j\times \R)+C\ep.
\end{align*}
Taking the sup over $\varphi$, $j\to \infty$, inf over $\Omega\supset B$ with $\mathcal{L}^{n-1}(\Omega)<\ep$, and $\ep\to 0$, we get
\begin{align*}
P_{\gamma_A}\left(E^{s} ; B \times \mathbb{R}\right)\leq P_{\gamma_A}(E;B\times \R).
\end{align*}
\end{proof}

\begin{lemma}[Graphical estimate for $E$]\label{graphical_estimate}\mbox{}\\
Let $E$ be a set of finite anisotropic Gaussian perimeter in $\mathbb{R}^{n}, n \geq 2$. Then
$$
P_{\gamma_A}(E ; B \times \mathbb{R}) \geq \frac{\det\sqrt{A}}{(2\pi)^{(n-1)/2}}\int_{B} \sqrt{p_{E}(z)^{2}+\left|\nabla^{\prime} v_{E}(z)-\int_{E_z}\nabla^{\prime}\left(e^{-|\sqrt{A}x|^2/2}\right)dy\right|^{2}}\ dz,
$$
for every Borel set $B\subset B_E$,
$$
P_{\gamma_A}(E^s ; B \times \mathbb{R}) = \frac{\det\sqrt{A}}{(2\pi)^{(n-1)/2}}\int_{B} \sqrt{p_{E^s}(z)^{2}+\left|\nabla^{\prime} v_{E^s}(z)-\int_{E^s_z}\nabla^{\prime}\left(e^{-|\sqrt{A}x|^2/2}\right)dy\right|^{2}}\ dz,
$$
for every Borel set $B\subset B_{E^s}$, and
$$
\begin{aligned}
P_{\gamma_A}(E^s;B\times \R)&\leq P_{\gamma_A}(E;B\times \R)+\frac{\det \sqrt{A}}{(2 \pi)^{(n-1) / 2}}\int_{B}\Bigg|\int_{E_z}\nabla^{\prime}e^{-|\sqrt{A}x|^2/2}dy-\int_{E_z^s} \nabla^{\prime}e^{-|\sqrt{A}x|^2/2}dy\Bigg|\ dz,
\end{aligned}
$$
for every Borel set $B \subseteq B_{E}\cap B_{E^s} .$ 
\end{lemma}
\begin{proof}
By the co-area formula (\ref{co-area formula for sets of finite perimeter_1}) and Vol'pert Theorem (Theorem \ref{Vol'pert}),
$$
P_{\gamma_A}(E ; B \times \mathbb{R})=\frac{\det\sqrt{A}}{(2\pi)^{(n-1)/2}}\int_{\partial^{*} E \cap(B \times \mathbb{R})}\ \frac{\left|\nu_{n}^{E}\right|}{\left|\nu_{n}^{E}\right|}e^{-|\sqrt{A}x|^2/2}\ d \mathcal{H}^{n-1}=\frac{\det\sqrt{A}}{(2\pi)^{(n-1)/2}}\int_{B} \int_{\partial^{*} E_{z}} \frac{1}{\left|\nu_{n}^{E}\right|}\ d \mathcal{H}_{z}^{0}\ d z.
$$
Applying Jensen's inequality to the convex function $\varphi:(x_1,...,x_{n-1})=x\mapsto\sqrt{1+|x|^2}$, we have
\begin{align}\label{Jensen}
&\sqrt{1+\sum_{i=1}^{n-1}\left(\fint_{\partial^{*} E_{z}} \frac{\nu_{i}^{E}}{\left|\nu_{n}^{E}\right|}\ d \mathcal{H}_{z}^{0}\right)^{2}}=\varphi\left(\fint_{\partial^{*} E_{z}} \frac{\nu_{1}^{E}}{\left|\nu_{n}^{E}\right|}\ d \mathcal{H}_{z}^{0},\cdots,\fint_{\partial^{*} E_{z}} \frac{\nu_{n-1}^{E}}{\left|\nu_{n}^{E}\right|}\ d \mathcal{H}_{z}^{0}\right)\\
&=\varphi\left(\fint_{\partial^{*} E_{z}} \left(\frac{\nu_{1}^{E}}{\left|\nu_{n}^{E}\right|},\cdots,\frac{\nu_{n-1}^{E}}{\left|\nu_{n}^{E}\right|} \right) d \mathcal{H}_{z}^{0}\right)\nonumber\leq \fint_{\partial^{*} E_{z}} \varphi \left(\frac{\nu_{1}^{E}}{\left|\nu_{n}^{E}\right|},\cdots,\frac{\nu_{n-1}^{E}}{\left|\nu_{n}^{E}\right|} \right)d \mathcal{H}_{z}^{0}\nonumber\\
&= \fint_{\partial^{*} E_{z}} \sqrt{1+\sum_{i=1}^{n-1}\frac{\left|\nu_{i}^{E}\right|^{2}}{\left|\nu_{n}^{E}\right|^{2}}}\ d \mathcal{H}_{z}^{0}.\nonumber
\end{align}
By Lemma \ref{distributional_derivative_formula}, (\ref{Jensen}), and the definition of $p_E(z)$,  
\begin{align*}
P_{\gamma_A}(E ; B \times \mathbb{R}) &=\frac{\det\sqrt{A}}{(2\pi)^{(n-1)/2}}\int_{B} \int_{\partial^{*} E_{z}}  \sqrt{1+\frac{\sum_{i=1}^{n-1}\left|\nu_{i}^{E}\right|^{2}}{\left|\nu_{n}^{E}\right|^{2}}}\ d \mathcal{H}_{z}^{0}\ d z\\
&=\frac{\det\sqrt{A}}{(2\pi)^{(n-1)/2}}\int_{B} p_{E}(z)  \fint_{\partial^{*} E_{z}} \sqrt{1+\frac{\sum_{i=1}^{n-1}\left|\nu_{i}^{E}\right|^{2}}{\left|\nu_{n}^{E}\right|^{2}}}\ d \mathcal{H}_{z}^{0} \ d z\\
& \geq \frac{\det\sqrt{A}}{(2\pi)^{(n-1)/2}}\int_{B} p_{E}(z) \sqrt{1+\sum_{i=1}^{n-1}\left(\fint_{\partial^{*} E_{z}} \frac{\nu_{i}^{E}}{\left|\nu_{n}^{E}\right|}\ d \mathcal{H}_{z}^{0}\right)^{2}}\ d z\\
& = \frac{\det\sqrt{A}}{(2\pi)^{(n-1)/2}}\int_{B}  \sqrt{p_E(z)^2+\sum_{i=1}^{n-1}\left(\fint_{\partial^{*} E_{z}} \frac{\nu_{i}^{E}}{\left|\nu_{n}^{E}\right|}\ d \mathcal{H}_{z}^{0}\right)^{2}p_E(z)^2}\ d z \\
& = \frac{\det\sqrt{A}}{(2\pi)^{(n-1)/2}}\int_{B}  \sqrt{p_E(z)^2+\sum_{i=1}^{n-1}\left(\int_{\partial^{*} E_{z}} \frac{\nu_{i}^{E}}{\left|\nu_{n}^{E}\right|}\ d \mathcal{H}_{z}^{0}\right)^{2}}\ d z \\
&=\frac{\det\sqrt{A}}{(2\pi)^{(n-1)/2}}\int_{B} \sqrt{p_{E}(z)^{2}+\left|\nabla^{\prime} v_{E}(z)-\int_{E_z}\nabla^{\prime}\left(e^{-|\sqrt{A}x|^2/2}\right)dy\right|^{2}}\ dz,
\end{align*}
where $\nabla'v_E(z)=(D_1v_E(z),\cdots, D_{n-1}v_E(z))$, $B\subset B_E$ and 
$$
D_{i} v_{E}(z)=\int_{\left(\partial^{*} E\right)_{z}} \frac{\nu_{i}^{E}(z, y)}{\left|\nu_{n}^{E}(z, y)\right|} d \mathcal{H}_{z}^{0}(y) +\int_{E_z} \frac{\partial}{\partial x_i}\left(e^{-|\sqrt{A}x|^2/2}\right) dy\mbox{\qquad for $i=1,2,...,n-1$}.
$$
Applying the same calculation on $E^s$ with $E^s_z=(-\infty, y(z))$, $y(z)=\phi_z^{-1}(v_E(z))$, and 
$$p_{E^s}(z)=e^{-|\sqrt{A}(z,y(z))|^2/2},$$
we have for any $B\subset B_{E^s}$,
\begin{align*}
P_{\gamma_A}(E^s ; B \times \mathbb{R})
&=\frac{\det\sqrt{A}}{(2\pi)^{(n-1)/2}}\int_{B} \sqrt{p_{E^s}(z)^{2}+\left|\nabla^{\prime} v_{E^s}(z)-\int_{E^s_z}\nabla^{\prime}\left(e^{-|\sqrt{A}x|^2/2}\right)dy\right|^{2}}\ dz.
\end{align*}
Notice that we have the following inequality
\begin{align}\label{algebraic inequality}
\sqrt{a^2+b^2}-\sqrt{a^2+c^2}\leq |b-c|, \mbox{\quad if $b,c\geq 0$}.
\end{align}
Plugging
$$a=p_E(z),\  b=\left|\nabla^{\prime} v_{E}(z)-\int_{E^s_z}\nabla^{\prime}\left(e^{-|\sqrt{A}x|^2/2}\right)dy\right|,\ \mbox{and}\ c=\left|\nabla^{\prime} v_{E}(z)-\int_{E_z}\nabla^{\prime}\left(e^{-|\sqrt{A}x|^2/2}\right)dy\right|$$
into (\ref{algebraic inequality}), we have for any $B\subset B_E\cap B_{E^s}$, 
\begin{align*}
P_{\gamma_A}(E^s;B\times \R)-P_{\gamma_A}(E;B\times \R)&\leq \frac{\det\sqrt{A}}{(2\pi)^{(n-1)/2}}\int_{B} \sqrt{p_{E^s}(z)^{2}+\left|\nabla^{\prime} v_{E^s}(z)-\int_{E^s_z}\nabla^{\prime}\left(e^{-|\sqrt{A}x|^2/2}\right)dy\right|^{2}}\ dz\\
&\quad-\frac{\det\sqrt{A}}{(2\pi)^{(n-1)/2}}\int_{B} \sqrt{p_{E}(z)^{2}+\left|\nabla^{\prime} v_{E}(z)-\int_{E_z}\nabla^{\prime}\left(e^{-|\sqrt{A}x|^2/2}\right)dy\right|^{2}}\ dz\\
&\leq \frac{\det\sqrt{A}}{(2\pi)^{(n-1)/2}}\int_{B} \sqrt{p_{E}(z)^{2}+\left|\nabla^{\prime} v_{E}(z)-\int_{E^s_z}\nabla^{\prime}\left(e^{-|\sqrt{A}x|^2/2}\right)dy\right|^{2}}\ dz\\
&\quad-\frac{\det\sqrt{A}}{(2\pi)^{(n-1)/2}}\int_{B} \sqrt{p_{E}(z)^{2}+\left|\nabla^{\prime} v_{E}(z)-\int_{E_z}\nabla^{\prime}\left(e^{-|\sqrt{A}x|^2/2}\right)dy\right|^{2}}\ dz\\
&\leq \frac{\det\sqrt{A}}{(2\pi)^{(n-1)/2}}\int_{B} \left| \left|\nabla^{\prime} v_{E}(z)-\int_{E_z}\nabla^{\prime}\left(e^{-|\sqrt{A}x|^2/2}\right)dy\right|\right.\\
&\quad -\left.\left|\nabla^{\prime} v_{E}(z)-\int_{E^s_z}\nabla^{\prime}\left(e^{-|\sqrt{A}x|^2/2}\right)dy\right|\right|\ dz\\
&\leq \frac{\det\sqrt{A}}{(2\pi)^{(n-1)/2}}\int_{B} \left| \int_{E_z}\nabla^{\prime}\left(e^{-|\sqrt{A}x|^2/2}\right) -\int_{E^s_z}\nabla^{\prime}\left(e^{-|\sqrt{A}x|^2/2}\right)dy\right|dz,
\end{align*}
where we have used Proposition \ref{basic_properties} (1) and (3), i.e., $\nabla'v_E(z)=\nabla'v_{E^z}(z)$ and $p_E(z)\geq p_{E^s}(z)$ for a.e. $z\in B\subset B_E\cap B_{E^s}$.
\end{proof}

Although the perimeter might not decrease in every direction when we do the Ehrhard Symmetrization, we are still able to give an upper bound for the perimeter of the Ehrhard Symmetrization with an error term involving $A-\lambda I_n$ and barycenters. Combining Lemma \ref{dust_estimate}, Lemma \ref{cylindrical_estimate}, and Lemma \ref{graphical_estimate}, we have the following estimate which tells us how the direction of Ehrhard Symmetrization affects the anisotropic Gaussian perimeter.

\begin{theorem}[Anisotropic Gaussian Perimeter Inequality under Ehrhard Symmetrization]\label{Ehrhard_Sym_Ineq_I}\mbox{}\\
Let $n \geq 2$ and let $E$ be a set of finite $A$-anisotropic Gaussian perimeter in $\mathbb{R}^{n}$. Then, for every Borel set $B \subseteq \mathbb{R}^{n-1}$ we have
\begin{align*}
P_{\gamma_A}\left(E^{s} ; B \times \mathbb{R}\right) &\leq P_{\gamma_A}(E ; B \times \mathbb{R})\\
&\quad+\frac{\det \sqrt{A}}{(2 \pi)^{(n-1) / 2}}\int_{B}\Bigg|\int_{E_z}\nabla^{\prime}e^{-|\sqrt{A}x|^2/2}dy-\int_{E_z^s} \nabla^{\prime}e^{-|\sqrt{A}x|^2/2}dy\Bigg|\ dz.
\end{align*}
Moreover,
\begin{align*}
P_{\gamma_A}\left(E^{s} ; B \times \mathbb{R}\right) &\leq P_{\gamma_A}(E ; B \times \mathbb{R})\\
&\quad+ \sqrt{2\pi}\|Ae_n-\langle Ae_n,e_n\rangle e_n\|\langle b_{\gamma_A}(E\cap (B\times \R))-b_{\gamma_A}(E^s\cap (B\times \R)),e_n\rangle,
\end{align*}
where
$$b_{\gamma_A}(E):=\int_Ex\ d\gamma_A(x).$$
\end{theorem}
\begin{proof} 
{\bf Step 1}: For any Borel set $B$,
$$
\begin{aligned}
B&=(B\cap (B_E\cap B_{E^s}))\cup (B\cap \pi_+(E)\slash (B_E\cap B_{E^s}))\cup(B\slash \pi_+(E))\\
&:= B_1\cup B_2\cup B_3.
\end{aligned}
$$
Recall that: $\pi_+(E)=\pi_+(E^s)$ and
$$B_E\subset \pi_+(E),\quad B_{E^s}\subset \pi_+(E^s)=\pi_+(E),\quad \mathcal{L}^{n-1}(\pi_+(E)\slash B_E)=0, \quad \mathcal{L}^{n-1}(\pi_+(E)\slash B_{E^s})=0.$$
Then
$$B_2\subset  \pi_+(E)\slash (B_E\cap B_{E^s})= (\pi_+(E)\slash B_E)\cup (\pi_+(E)\slash B_{E^s})\implies \mathcal{L}^{n-1}(B_2)=0.$$
Moreover, for any $z\in B_3$, $v_E(z)=0$. Thus, applying the dust estimate (Lemma \ref{dust_estimate}) on $B_3$, the cylindrical estimate (Lemma \ref{cylindrical_estimate}) on $B_2$, the graphical estimate (Lemma \ref{graphical_estimate}) on $B_1\subset B_E\cap B_{E^s}$, we have
$$
\begin{aligned}
P_{\gamma_A}(E^s;B\times \R)&=P_{\gamma_A}(E^s;B_1\times \R)+P_{\gamma_A}(E^s;B_2\times \R)+P_{\gamma_A}(E;B_3\times \R)\\
&\leq P_{\gamma_A}(E;B_1\times \R)+P_{\gamma_A}(E;B_2\times \R)+P_{\gamma_A}(E;B_3\times \R)\\
&\quad+\frac{\det \sqrt{A}}{(2 \pi)^{(n-1) / 2}}\int_{B_1\cup B_2\cup B_3}\Bigg|\int_{E_z}\nabla^{\prime}e^{-|\sqrt{A}x|^2/2}dy-\int_{E_z^s} \nabla^{\prime}e^{-|\sqrt{A}x|^2/2}dy\Bigg|\ dz\\
&=P_{\gamma_A}(E;B\times \R)+\frac{\det \sqrt{A}}{(2 \pi)^{(n-1) / 2}}\int_{B}\Bigg|\int_{E_z}\nabla^{\prime}e^{-|\sqrt{A}x|^2/2}dy-\int_{E_z^s} \nabla^{\prime}e^{-|\sqrt{A}x|^2/2}dy\Bigg|\ dz.
\end{aligned}
$$
{\bf Step 2}: Now we claim that
\begin{align*}
&\frac{\det \sqrt{A}}{(2 \pi)^{(n-1) / 2}}\int_{B}\Bigg|\int_{E_z}\nabla^{\prime}e^{-|\sqrt{A}x|^2/2}dy-\int_{E_z^s} \nabla^{\prime}e^{-|\sqrt{A}x|^2/2}dy\Bigg|\ dz\\
&= \sqrt{2\pi} \|Ae_n-\langle Ae_n,e_n\rangle e_n\|\langle b_{\gamma_A}(E\cap (B\times \R))-b_{\gamma_A}(E^s\cap (B\times \R)),e_n\rangle.
\end{align*}
By Proposition \ref{basic_properties} (2) and recall that $x=(z,y)$, i.e., $y=\langle x,e_n\rangle$, we have
\begin{align*}
&\frac{\det \sqrt{A}}{(2 \pi)^{(n-1) / 2}}\int_{B}\Bigg|\int_{E_z}\nabla^{\prime}e^{-|\sqrt{A}x|^2/2}dy-\int_{E_z^s} \nabla^{\prime}e^{-|\sqrt{A}x|^2/2}dy\Bigg|\ dz\\
&= \frac{\det \sqrt{A}}{(2 \pi)^{(n-1) / 2}}\int_B \left(\int_{E_z}y\ d\mu_z(y)-\int_{E^s_z}y\ d\mu_z(y)\right)\|Ae_n-\langle Ae_n,e_n\rangle e_n\|\ dz\\
&= \|Ae_n-\langle Ae_n,e_n\rangle e_n\|\frac{\det \sqrt{A}}{(2 \pi)^{(n-1) / 2}}\left(\int_B \int_{E_z}y\ d\mu_z(y)dz -\int_B\int_{E^s_z}y\ d\mu_z(y)dz\right)\\
&= \|Ae_n-\langle Ae_n,e_n\rangle e_n\|\frac{\det \sqrt{A}}{(2 \pi)^{(n-1) / 2}}\left(\int_{E\cap (E\times \R)} ye^{-|\sqrt{A}x|^2/2}\ dx -\int_{E^s\cap (B\times\R)}ye^{-|\sqrt{A}x|^2/2}\ dx\right)\\
&= \sqrt{2\pi}\|Ae_n-\langle Ae_n,e_n\rangle e_n\|\left\langle \left(\int_{E\cap (B\times \R)} x\ d\gamma_A(x) -\int_{E^s\cap (B\times\R)}x\ d\gamma_A(x)\right),e_n\right\rangle\\
&= \sqrt{2\pi}\|Ae_n-\langle Ae_n,e_n\rangle e_n\|\langle b_{\gamma_A}(E\cap (B\times \R))-b_{\gamma_A}(E^s\cap (B\times \R)),e_n\rangle.
\end{align*}
\end{proof}

Our main goal here is to define the Ehrhard symmetrization to any direction $u\in \SS^{n-1}$ and then extend the result of Theorem \ref{Ehrhard_Sym_Ineq_I} to this new definition. We will start by recalling that Vol'pert theorem (Theorem \ref{Vol'pert}) actually holds for every direction (see \cite{Ambrosio_Fusco_Pallara}, Theorem 3.108 and \cite{Fusco_Iso}, Theorem 3.21). That is, if $E$ is a set of locally finite perimeter, the {\bf one-dimensional slice of $E$ through $z$ in direction $u$} defined as
$$E_{z,u}:= \{x=z+tu\in E:\  t\in \R\}$$
is also a set of locally finite perimeter. Moreover,  $\left(\partial^{M} E\right)_{z,u}=\partial^{M}\left(E_{z,u}\right)=\partial^{*}\left(E_{z,u}\right)=\left(\partial^{*} E\right)_{z,u}$ and $\nu_{u}^{E}(x):= \langle \nu^E(x),u\rangle \neq 0$ for every $t$ such that $x=z+tu \in \partial^{*} E$ where
$$x=z+tu\in \langle u\rangle^{\perp}\oplus \langle u\rangle.$$
The {\bf Ehrhard symmetrization $E_{A,u}^{s}$ of $E$ with respect to the $u$-direction and matrix $A$} is defined as
\begin{align}\label{new_def}
E^{s}_{A,u}:= \left\{x=z+tu \in \mathbb{R}^{n}: t>\Phi^{-1}_{A,z,u}\left(\mu_{A}(E_{z,u})\right)\right\},
\end{align}
and the {\bf essential projection of $E$ with respect to the $u$-direction and matrix $A$} is defined as 
$$\pi_{+,A,u}(E):=\left\{z\in \langle u\rangle^{\perp} : \mu_{A}\left(E_{z,u}\right)>0\right\},$$
where
$$\mu_{A}(F):=\int_F e^{-|\sqrt{A}x|^2/2}d\mathcal{H}^{1}(x),\quad \Phi_{A,z,u}(s):=\int_{s}^{\infty} e^{-|\sqrt{A}(z+tu)|^2/2}dt.$$
It is not hard to see that $\mu_A(E_{z,u})=\mu_A((E^s_{A,u})_{z,u})$ and $\pi_{+,A,u}(E)=\pi_{+,A,u}(E^s_{A,u})$.
Notice that the definition (\ref{new_def}) agrees with the definition (\ref{old_def}) in Section \ref{Ehrhard_symmetrization_notation_I}, i.e., if $u=-e_n$, we have
$$ \pi_{+,A,-e_n}(E)=\pi_{+}(E),\quad E_{A,-e_n}^s=E^s.$$
Moreover, Theorem \ref{Ehrhard_Sym_Ineq_I} says that
\begin{align*}
P_{\gamma_A}\left(E^{s}_{A,-e_n} ; B \times \mathbb{R}\right) &\leq P_{\gamma_A}(E ; B \times \mathbb{R})\\
&\quad+ \sqrt{2\pi}\|A(-e_n)-\langle A(-e_n),(-e_n)\rangle (-e_n)\|\langle b_{\gamma_A}(E^s_{A,-e_n}\cap (B\times \R))-b_{\gamma_A}(E\cap (B\times \R)),-e_n\rangle.
\end{align*}
Our next goal is to extend this result to the Ehrhard symmetrization $E_{A,u}^{s}$. Before doing that, we need a lemma that helps us handle the rotation of the Ehrhard symmetrization $E_{A,u}^{s}$. The proof of it can be easily deduced by the change of variables and set theory and hence we omit the verification.

\begin{lemma}
Let $O$ be an orthogonal matrix such that $u=O(-e_n)$. Then
\begin{align}\label{rotation_of_sets}
(O^{-1}E)^s_{O^{\mathsf{T}}AO,-e_n}=O^{-1}E^s_{A,u}.
\end{align}
\end{lemma}

\subsection{Proof of Theorem \ref{Ehrhard_Sym_Ineq_II}} We can always find an orthogonal matrix $O$ such that $u=O(-e_n)$. By equation (\ref{rotation_of_sets}), we have
\begin{align*}
 (O^{-1}E)^s_{O^{\mathsf{T}}AO,-e_n}=O^{-1}E^s_{A,u}.
\end{align*}
Now we claim that  
\begin{align}\label{Eharhard_symmetrization_for_general_sets}
E^s_{A,u}=O[(O^{-1}E)^s_{O^{\mathsf{T}}AO,-e_n}]\mbox{\quad is a set of locally finite perimeter in $\R^n.$}
\end{align}
By Proposition \ref{locally finite anisotropic Gaussian perimeter}, $E$ is a set of locally finite perimeter. Applying Proposition \ref{Gaussian_perimeter} (3), 
$$P_{\gamma_{O^{\mathsf{T}}AO}}(O^{-1}E)=P_{\gamma_A}(E)<\infty,$$
i.e., $O^{-1}E$ is a set of finite $O^{\mathsf{T}}AO$-anisotropic Gaussian perimeter. Then Theorem \ref{approximation_theorem_for_BV_maps} tells us that $(O^{-1}E)^s_{O^{\mathsf{T}}AO,-e_n}$ is also a set of locally finite perimeter. By \cite{maggi}, Exercise 15.10, $E^s_{A,u}=O[(O^{-1}E)^s_{O^{\mathsf{T}}AO,-e_n}]$ is a set of locally finite perimeter. \\

\noindent Next we prove the second part of the theorem. By equation (\ref{rotation_of_sets}), Proposition \ref{Gaussian_perimeter} (3), and $O^{-1}B\subset \R^{n-1}$, we have
\begin{align*}
&P_{\gamma_A}\left(E^s_{A,u} ; B \oplus \langle u\rangle\right)=P_{\gamma_A}\left(E^s_{A,u} ; O(O^{-1}B) \oplus O\langle -e_n\rangle\right)=P_{\gamma_A}\left(E^s_{A,u}  ; O(O^{-1}B \times \R)\right)\\
&=P_{\gamma_{O^{\mathsf{T}}AO}}\left(O^{-1}E^{s}_{A,u} ; O^{-1}B \times \mathbb{R}\right)=P_{\gamma_{O^{\mathsf{T}}AO}}\left((O^{-1}E)^{s}_{O^{\mathsf{T}}AO,-e_n} ; O^{-1}B \times \mathbb{R}\right).
\end{align*}
Since $O^{-1}E$ is a set of finite $O^{\mathsf{T}}AO$-anisotropic Gaussian perimeter, we can apply Theorem \ref{Ehrhard_Sym_Ineq_I} with $E$ as $O^{-1}E$, $B$ as $O^{-1}B$, and $A$ as $O^{\mathsf{T}}AO$. Hence,
\begin{align*}
&P_{\gamma_A}\left(E^s_{A,u} ; B \oplus \langle u\rangle\right)= P_{\gamma_{O^{\mathsf{T}}AO}}\left((O^{-1}E)^{s}_{O^{\mathsf{T}}AO,-e_n} ; O^{-1}B \times \mathbb{R}\right) \\
&\leq P_{\gamma_{O^{\mathsf{T}}AO}}(O^{-1}E ; O^{-1}B \times \mathbb{R})+ \sqrt{2\pi}\|O^{\mathsf{T}}AOe_n-\langle O^{\mathsf{T}}AOe_n,e_n\rangle e_n\|\left\langle b_{\gamma_{O^{\mathsf{T}}AO}}((O^{-1}E)^s_{O^{\mathsf{T}}AO,-e_n}\cap (O^{-1}B\times \R))\right.\\
&\quad-\left.b_{\gamma_{O^{\mathsf{T}}AO}}(O^{-1}E\cap (O^{-1}B\times \R)),-e_n\right\rangle\\
&=P_{\gamma_A}\left(E ; B \oplus \langle u\rangle\right)+ \sqrt{2\pi}\|Au-\langle Au,u\rangle u\|\langle b_{\gamma_{A}}(E^s_{A,u}\cap (B \oplus \langle u\rangle ))-b_{\gamma_{A}}(E\cap (B \oplus \langle u\rangle)),u\rangle
\end{align*}
where we have used $(O^{-1}E)^s_{O^{\mathsf{T}}AO,-e_n}=O^{-1}E^s_{A,u}$ and $b_{\gamma_{O^{\mathsf{T}}AO}}(E)=O^{-1}b_{\gamma_A}(OE)$. 
\qed \\

Our next result tells us that the anisotropic Gaussian perimeter decreases when the Ehrhard symmetrization is done along an eigenvector direction of $A$.

\begin{corollary}\label{Eigen_Ehrhard_Sym}
Let $n \geq 2$ and let $E$ be a set of finite $A$-anisotropic Gaussian perimeter in $\mathbb{R}^{n}$. Assume that
$$u\in V_{\lambda}(A)\cap \SS^{n-1}$$
where $V_{\lambda}(A)$ is the eigenspace of $A$ associated with eigenvalue $\lambda$. Then, for every Borel set $B \subseteq \langle u \rangle^{\perp}$, we have
\begin{align}\label{decreasing_ineq}
P_{\gamma_A}\left(E_{A,u}^{s} ; B \oplus \langle u\rangle\right) \leq P_{\gamma_A}\left(E ; B \oplus \langle u\rangle\right),
\end{align}
and in particular $P_{\gamma_A}(E_{A,u}^{s}) \leq P_{\gamma_A}(E)$. Moreover, if $P_{\gamma_A}(E)=P_{\gamma_A}(E_{A,u}^{s})$, then 
$$\mbox{for $\mathcal{H}^{n-1}$-a.e. $z \in \langle u \rangle^{\perp}$, the slice $E_{z,u}$ is $\mathcal{H}^{1}$-equivalent to either $\emptyset$ or $\langle u\rangle$ or a half-line.}$$
\end{corollary}
\begin{proof}
Since for $u\in V_\lambda(A)\cap \SS^{n-1}$,
$$Au=\lambda u\implies \langle Au,u\rangle=\lambda\implies Au-\langle Au,u\rangle u=0.$$
Then Theorem \ref{Ehrhard_Sym_Ineq_II} shows that $P_{\gamma_A}\left(E_{A,u}^{s} ; B \oplus \langle u\rangle\right) \leq P_{\gamma_A}\left(E ; B \oplus \langle u\rangle\right)$.
For the second part, we may assume that $u=-e_n$ since we can always rotate the coordinate system. Thus, by assumption,
$$P_{\gamma_A}(E)=P_{\gamma_A}\left(E^{s}\right),\qquad u=-e_n\in V_\lambda(A)\cap \SS^{n-1}.$$
By (\ref{decreasing_ineq}), we have
$$P_{\gamma_A}\left(E^{s};B\times \R\right)\leq P_{\gamma_A}(E; B\times\R)$$
for every Borel set $B\subset \R^{n-1}$. We claim that  
$$P_{\gamma_A}(E ; B \times \mathbb{R})=P_{\gamma_A}\left(E^{s} ; B \times \mathbb{R}\right)\mbox{\quad for every Borel set $B \subseteq \mathbb{R}^{n-1}$.}$$
Suppose not, $P_{\gamma_A}(E ; B \times \mathbb{R})>P_{\gamma_A}\left(E^{s} ; B \times \mathbb{R}\right)$ for some Borel set $B$. Then
$$P_{\gamma_A}(E)=P_{\gamma_A}(E ; B \times \mathbb{R})+P_{\gamma_A}(E ; B^c \times \mathbb{R})>P_{\gamma_A}(E^s ; B \times \mathbb{R})+P_{\gamma_A}(E^s ; B^c \times \mathbb{R})=P_{\gamma_A}(E^s),$$
which contradicts our assumption. Now we plug in the Borel set $B_{E} \cap B_{E^{s}}$ from Vol'pert Theorem (Theorem \ref{Vol'pert}), i.e., we have 
\begin{align}\label{Eigen_Ehrhard_Sym_eq1}
P_{\gamma_A}\left(E ;\left(B_{E} \cap B_{E^{s}}\right) \times \mathbb{R}\right)=P_{\gamma_A}\left(E^{s} ;\left(B_{E} \cap B_{E^{s}}\right) \times \mathbb{R}\right).
\end{align}
Notice that $Ae_n=\lambda e_n$ and let $\lambda=d^2$ with $d>0$, $x=(z,y)$, and $z=(z_1,\ldots, z_{n-1})$. Then
\begin{align}\label{Eigen_Ehrhard_Sym_eq2}
|\sqrt{A}x|^2=\left\langle Ax,x\right\rangle&=\left\langle \sum_{i=1}^{n-1}z_iAe_i+y\lambda e_n,\sum_{j=1}^{n-1}z_je_j+ye_n\right\rangle\notag\\
&=\left\langle \sum_{i=1}^{n-1}z_iAe_i,\sum_{j=1}^{n-1}z_je_j\right\rangle+\left\langle \sum_{i=1}^{n-1}z_iAe_i,y e_n\right\rangle+\left\langle y\lambda e_n,\sum_{j=1}^{n-1}z_je_j\right\rangle+\left\langle y\lambda e_n,ye_n\right\rangle\notag\\
&=\left\langle A\sum_{i=1}^{n-1}z_ie_i,\sum_{i=1}^{n-1}z_ie_i\right\rangle+\left\langle \sum_{i=1}^{n-1}z_ie_i,yA e_n\right\rangle+0+\lambda y^2\notag\\
&=\left|\sqrt{A}\sum_{i=1}^{n-1}z_ie_i\right|^2+\left\langle \sum_{i=1}^{n-1}z_ie_i,y\lambda e_n\right\rangle+\lambda y^2=\left|\sum_{i=1}^{n-1}z_i\sqrt{A}e_i\right|^2+d^2y^2
\end{align}
since $e_n$ is an eigenvector of $A$ and $A$ is symmetric. Therefore,
\begin{align*}
v_E(z)&=\int_{E_z}e^{-|\sqrt{A}x|^2/2}dy=e^{-\left|\sum_{i=1}^{n-1}z_i\sqrt{A}e_i\right|^2/2}\int_{E_z}e^{-d^2y^2/2}\ dy,\\
v_{E^s}(z)&=\int_{E^s_z}e^{-|\sqrt{A}x|^2/2}dy=e^{-\left|\sum_{i=1}^{n-1}z_i\sqrt{A}e_i\right|^2/2}\int_{E^s_z}e^{-d^2y^2/2}\ dy,
\end{align*}
and hence
$$\int_{E_z}e^{-d^2y^2/2}\ dy=\int_{E^s_z}e^{-d^2y^2/2}\ dy$$
since $v_E(z)=v_{E^s}(z)$. Moreover,
\begin{align}\label{Eigen_Ehrhard_Sym_eq3}
\int_{E_z}\nabla^{\prime}\left(e^{-|\sqrt{A}x|^2/2}\right)dy&=\int_{E_z}\nabla^{\prime}\left(e^{-\left|\sum_{i=1}^{n-1}z_i\sqrt{A}e_i\right|^2/2}\right)e^{-d^2y^2/2}dy\notag\\
&=\nabla^{\prime}\left(e^{-\left|\sum_{i=1}^{n-1}z_i\sqrt{A}e_i\right|^2/2}\right)\int_{E_z}e^{-d^2y^2/2}dy\notag\\
&=\nabla^{\prime}\left(e^{-\left|\sum_{i=1}^{n-1}z_i\sqrt{A}e_i\right|^2/2}\right)\int_{E^s_z}e^{-d^2y^2/2}dy=\int_{E^s_z}\nabla^{\prime}\left(e^{-|\sqrt{A}x|^2/2}\right)dy.
\end{align}
By Lemma \ref{graphical_estimate}, Proposition \ref{basic_properties} (1)(3), (\ref{Eigen_Ehrhard_Sym_eq1}), and (\ref{Eigen_Ehrhard_Sym_eq3}),
\begin{align}\label{Eigen_Ehrhard_Sym_eq4}
P_{\gamma_A}(E ; (B_{E} \cap B_{E^{s}}) \times \mathbb{R}) & \geq \frac{\det\sqrt{A}}{(2\pi)^{(n-1)/2}}\int_{B_{E} \cap B_{E^{s}}} \sqrt{p_{E}(z)^{2}+\left|\nabla^{\prime} v_{E}(z)-\int_{E_z}\nabla^{\prime}\left(e^{-|\sqrt{A}x|^2/2}\right)dy\right|^{2}}\ dz\notag\\
& \geq \frac{\det\sqrt{A}}{(2\pi)^{(n-1)/2}}\int_{B_{E} \cap B_{E^{s}}} \sqrt{p_{E^s}(z)^{2}+\left|\nabla^{\prime} v_{E^s}(z)-\int_{E^s_z}\nabla^{\prime}\left(e^{-|\sqrt{A}x|^2/2}\right)dy\right|^{2}}\ dz\\
&=P_{\gamma_A}\left(E^{s} ; (B_{E} \cap B_{E^{s}}) \times \mathbb{R}\right)=P_{\gamma_A}(E ; (B_{E} \cap B_{E^{s}})\times \mathbb{R})\notag.
\end{align}
Therefore, (\ref{Eigen_Ehrhard_Sym_eq4}) implies that 
\begin{align}\label{Eigen_Ehrhard_Sym_eq5}
\mbox{$p_{E}(z)=p_{E^{s}}(z)$ for $\mathcal{H}^{n-1}$-a.e. $z \in B_E\cap B_{E^s}$}.
\end{align}
Moreover, let $z\in B_E\cap B_{E^s}$,
\begin{align}\label{p_E p_{E^s} I}
p_E(z)&=\int_{\rb E_z}e^{-|\sqrt{A}x|^2/2}\ d\mathcal{H}^0(y)=e^{-\left|\sum_{i=1}^{n-1}z_i\sqrt{A}e_i\right|^2/2}\int_{\rb E_z}e^{-d^2y^2/2}\ d\mathcal{H}^0(y),
\end{align}
\begin{align*}
\phi_z(t)=\int_{-\infty}^{t}e^{-|\sqrt{A}x|^2/2}\ dy&=e^{-\left|\sum_{i=1}^{n-1}z_i\sqrt{A}e_i\right|^2/2}\int_{-\infty}^{t}e^{-d^2y^2/2}\ dy=\frac{\sqrt{2\pi}}{d}e^{-\left|\sum_{i=1}^{n-1}z_i\sqrt{A}e_i\right|^2/2}\phi(dt),
\end{align*}
and hence
\begin{align}\label{relation_to_1D_Iso}
\phi_z^{-1}(v_E(z))=\frac{1}{d}\phi^{-1}\left(\frac{d}{\sqrt{2\pi}}\int_{E_z}e^{-d^2y^2/2}dy\right)=\frac{1}{d}\phi^{-1}(\gamma_{d^2}(E_z)),
\end{align}
where $\gamma_{d^2}$ is the $d^2$-anisotropic Gaussian measure, i.e.,
$$\gamma_{d^2}(F)=\frac{d}{\sqrt{2\pi}}\int_{F}e^{-d^2y^2/2}dy.$$
By equation (\ref{relation_to_1D_Iso}),
\begin{align}\label{p_E p_{E^s} II}
p_{E^s}(z)&=e^{-\left|\sum_{i=1}^{n-1}z_i\sqrt{A}e_i\right|^2/2}\int_{\rb E^s_z}e^{-d^2y^2/2}\ d\mathcal{H}^0(y)=e^{-\left|\sum_{i=1}^{n-1}z_i\sqrt{A}e_i\right|^2/2}\ e^{-d^2[\phi_z^{-1}(v_E(z))]^2/2}\notag\\
&=e^{-\left|\sum_{i=1}^{n-1}z_i\sqrt{A}e_i\right|^2/2}\ e^{-[\phi^{-1}(\gamma_{d^2}(E_z))]^2/2}
\end{align}
where $E^s_z=\left(-\infty,\phi_z^{-1}(v_E(z))\right)$.
Therefore, (\ref{Eigen_Ehrhard_Sym_eq5}), (\ref{p_E p_{E^s} I}), and (\ref{p_E p_{E^s} II}) implies that
$$\int_{\rb E_z}e^{-d^2y^2/2}\ d\mathcal{H}^0(y)=e^{-[\phi^{-1}(\gamma_{d^2}(E_z))]^2/2}$$
for $\mathcal{H}^{n-1}$-a.e. $z \in B_E\cap B_{E^s}$. Since $\mathcal{H}^{n-1}(\pi_+(E)\slash B_{E} \cap B_{E^{s}})=0$, we have
$$P_{\gamma_{d^2}}(E_z)=d\int_{\rb E_z}e^{-d^2y^2/2}d\mathcal{H}^0(y)=e^{-[\phi^{-1}(\gamma_{d^2}(E_z))]^2/2}d$$
for $\mathcal{H}^{n-1}$-a.e. $z \in \pi_+(E)$. Thanks to the equality case of the one-dimensional anisotropic Gaussian isoperimetric inequality (see Theorem \ref{AnisotropicGaussainIso}), for $\mathcal{H}^{n-1}$-a.e. $z \in \pi_+(E)$, $E_z$ is either $\mathcal{H}^{1}$-equivalent to $\emptyset$ or $\R$ or a half-line. Notice that for any $z\in \pi_+(E)^c$,
$$v_E(z)=0\implies \mathcal{H}^{1}(E_z)=0\implies \mbox{$E_z$ is $\mathcal{H}^{1}$-equivalent to $\emptyset$.}$$
In other words,
$$\mbox{for $\mathcal{H}^{n-1}$-a.e. $z \in \langle -e_n\rangle^\perp$, the slice $E_{z,-e_n}$ is $\mathcal{H}^{1}$-equivalent to either $\emptyset$ or $\R$ or a half-line.}$$
\end{proof}

From Theorem \ref{Ehrhard_Sym_Ineq_II}, we see that
\begin{align*}
P_{\gamma_A}(E_{A,u}^{s} ) \leq P_{\gamma_A}\left(E \right)+\sqrt{2\pi}\|Au-\langle Au,u\rangle u\|\langle b_{\gamma_A}(E^s_{A,u})-b_{\gamma_A}(E),u\rangle,
\end{align*}
for any set of finite anisotropic Gaussian perimeter in $\R^n$. A natural question here is whether
\begin{align*}
\left|P_{\gamma_A}(E_{A,u}^{s} ) - P_{\gamma_A}\left(E \right)\right|\leq M\|Au-\langle Au,u\rangle u\|\langle b_{\gamma_A}(E^s_{A,u})-b_{\gamma_A}(E),u\rangle
\end{align*}
for some constant $M$. Our final example shows that this is not the case.

\begin{example}
We give an example to show that the following statement is not true: for any $0<\lambda_1<\lambda_2$, there exists $M>0$ such that for any $\lambda(A)\subset [\lambda_1,\lambda_2]$, for any $u\in \SS^{n-1}$, and for any set of finite anisotropic Gaussian perimeter $E$ in $\R^n$,
\begin{align*}
\left|P_{\gamma_A}(E_{A,u}^{s} ) - P_{\gamma_A}\left(E \right)\right|\leq M\|Au-\langle Au,u\rangle u\|\langle b_{\gamma_A}(E^s_{A,u})-b_{\gamma_A}(E),u\rangle,
\end{align*}
where $\lambda(A)$ is the set of all eigenvalues of $A$, i.e., the {\bf spectrum} of $A$.
\end{example}
\begin{proof}
Consider $\lambda_1=\frac12<\frac32=\lambda_2$. Suppose there exists $M>0$ such that for any $\lambda(A)\subset [\lambda_1,\lambda_2]$, for any $u\in \SS^{n-1}$, and for any set of finite anisotropic Gaussian perimeter $E$ in $\R^n$,
\begin{align*}
\left|P_{\gamma_A}(E_{A,u}^{s}) - P_{\gamma_A}\left(E \right)\right|\leq M\|Au-\langle Au,u\rangle u\|\langle b_{\gamma_A}(E^s_{A,u})-b_{\gamma_A}(E),u\rangle.
\end{align*}
Take
$$n=2,\quad A=\begin{pmatrix}
1&0\\
0 &\frac12
\end{pmatrix},\quad u=-e_2,\quad  E=[-1,1]^2.$$
Then clearly, $\lambda(A)\subset [\lambda_1,\lambda_2]$ and $Au-\langle Au,u\rangle u=0$. That is,
$$P_{\gamma_A}\left(E^{s} \right) = P_{\gamma_A}\left(E \right),$$
where $E^s=E^s_{A,-e_n}$. By Corollary \ref{Eigen_Ehrhard_Sym}, for $\mathcal{H}^{1}$-a.e. $z\in \R^{1}$, the slice $E_{z}$ is $\mathcal{H}^{1}$-equivalent to either $\emptyset$ or $\R$ or a half-line. However, for all $z\in [-1,1]$, the slice $E_z$ is an interval $[-1,1]$. This gives us a contradiction.
\end{proof}

\section{Characterization of Ehrhard symmetrizable measures}\label{Characterization of Ehrhard symmetrizable measures}

\subsection{A regularity lemma for Ehrhard symmetrization sets}\mbox{}\vspace{-.22cm}\\

The anisotropic Gaussian perimeter always decreases if $u\in V_{\lambda}(A)\cap \SS^{n-1}$, where $V_\lambda(A)$ is the eigenspace of $A$ associated with eigenvalue $\lambda$ (see Corollary \ref{Eigen_Ehrhard_Sym}). In fact, this is a necessary and sufficient condition for the anisotropic Gaussian perimeter to be decreasing. We say that the measure $\gamma_A$ is {\bf Ehrhard symmetrizable} if
$$P_{\gamma_A}(E_{A,u}^s)\leq P_{\gamma_A}(E)$$
for all $u\in \SS^{n-1}$, and for all measurable set $E\subset \R^n$. We show that $\gamma_A$ is Ehrhard symmetrizable if and only if $A$ is a multiple of the identity matrix (see Theorem \ref{Uniqueness_Ehrhard_Sym}).

\begin{lemma}[A regularity lemma for $E^s$]\label{regularity_lemma}\mbox{}\\
Let $A=(A_{ij})\in M_n(\R)$ be a symmetric positive definite matrix. Suppose $E=\Omega\times (0,\infty)$ with an open set $\Omega\subset \R^{n-1}$ that contains the origin. Then
$$E^s=E^s_{A,-e_n}=\{x=(z,y)\in \R^{n-1}\times \R:z\in \Omega,y<h(z)\},$$
where $h(z):=\phi_z^{-1}(v_E(z))$ is $C^1(\Omega)$ and $\nabla'h$ is locally Lipschitz on $\Omega$. In particular, $h(0)=0$,
$$\nabla'h(0)=-2\left(\int_0^{\infty}ye^{-|\sqrt{A}(0,y)|^2/2}\ dy\right) A'e_{n},$$
where $\nabla'=(\partial_1,\ldots, \partial_{n-1})$ and $A'\in M_{(n-1)\times n}(\R)$ is the first $n-1$ rows of matrix from $A$. Also,
\begin{align*}
\nabla' h(0)=0&\iff A_{1n}=A_{2n}=\ldots =A_{n-1,n}=0\iff e_n\in V_{A_{nn}}(A).
\end{align*}
Moreover,
$$h(z)=\ell(z)+\int_0^1\langle\nabla' h(tz)-\nabla' h(0),z \rangle\ dt\mbox{\quad for all $z\in \Omega$}$$
where
$$\ell(z):=h(0)+\langle \nabla'h(0),z\rangle=-2\left(\int_0^{\infty}ye^{-|\sqrt{A}(0,y)|^2/2}\ dy\right) A'e_{n}\cdot z$$
\end{lemma}
\begin{proof}
Recall that
$$
E^s=\left\{(z, y) \in \mathbb{R}^{n-1}\times \R: z\in \Omega, y<h(z)=\phi^{-1}_{z}\left(v_{E}(z)\right)\right\}
$$
where
$$v_E(z)=\int_{E_z}e^{-|\sqrt{A}x|^2/2}dy,\qquad \phi_z(t)=\int_{-\infty}^{t} e^{-|\sqrt{A}x|^2/2}dy.$$
Since $E=\Omega\times (0,\infty)$, $E_z=(0,\infty)$ for all $z\in \Omega$. By Lemma \ref{computational_lemma} (2)(a)(b),
$$v_E(z)=\int_{E_z}e^{-|\sqrt{A}x|^2/2}dy=\int_{0}^{\infty}e^{-|\sqrt{A}x|^2/2}dy\mbox{\quad is differentiable on $\Omega$}$$
and
$$\nabla'v_E(z)=-\int_{0}^{\infty}e^{-|\sqrt{A}x|^2/2}A'x\ dy.$$
Now we claim that $\nabla'v_E$ is locally Lipschitz on $\Omega$ and hence, by Lemma \ref{regularity_estimates_for_C1}, 
$$h:z\mapsto \phi^{-1}_z(v_E(z))\mbox{ is in $C^1(\Omega)$}$$
and $\nabla'h$ is locally Lipschitz on $\Omega$. Let $K$ be any compact set in $\Omega$ and let $z_1,z_2\in K$. Then
\begin{align*}
|\nabla'v_E(z_1)-\nabla'v_E(z_2)|&=\left|\int_{0}^{\infty}e^{-|\sqrt{A}(z_1,y)|^2/2}A'(z_1,y)dy-\int_{0}^{\infty}e^{-|\sqrt{A}(z_2,y)|^2/2}A'(z_2,y)dy\right|\\
&\leq\left|\int_{0}^{\infty}e^{-|\sqrt{A}(z_1,y)|^2/2}A'(z_1,y)dy-\int_{0}^{\infty}e^{-|\sqrt{A}(z_1,y)|^2/2}A'(z_2,y)dy\right|\\
&\quad+\left|\int_{0}^{\infty}e^{-|\sqrt{A}(z_1,y)|^2/2}A'(z_2,y)dy-\int_{0}^{\infty}e^{-|\sqrt{A}(z_2,y)|^2/2}A'(z_2,y)dy\right|\\
&\leq\left|\int_{0}^{\infty}e^{-|\sqrt{A}(z_1,y)|^2/2}A'(z_1-z_2,0)dy\right|\\
&\quad+\int_{0}^{\infty}\left|e^{-|\sqrt{A}(z_1,y)|^2/2}-e^{-|\sqrt{A}(z_2,y)|^2/2}\right||A'(z_2,y)|dy\\
&\leq \int_{0}^{\infty}e^{-\|(\sqrt{A})^{-1}\|^{-2}|y|^2/2}\sqrt{\lambda_{\max}(A^{\prime\mathsf{T}}A^{\prime})}|z_1-z_2|dy\\
&\quad+ \int_0^{\infty}\lambda_{\max}(A^{\prime\mathsf{T}}A^{\prime})\Big(r(K)+|y|\Big)^2e^{-\|(\sqrt{A})^{-1}\|^{-2}|y|^2/2} |z_1-z_2|\ dy\\
&= C(K,A')|z_1-z_2|,
\end{align*}
 where we have used the estimate (\ref{exp_estimate}), $r(K)=\sup_{\zeta\in K}|\zeta|$, and
$$C(K,A'):=\sqrt{\lambda_{\max}(A^{\prime\mathsf{T}}A^{\prime})}\left(\int_0^{\infty} \left(1+\sqrt{\lambda_{\max}(A^{\prime\mathsf{T}}A^{\prime})}\Big(r(K)+|y|\Big)^2\right)e^{-\|(\sqrt{A})^{-1}\|^{-2}|y|^2/2} \ dy\right).$$
Next, notice that
\begin{align}\label{regularity_lemma_eq1}
\int_0^{\infty}e^{-|\sqrt{A}x|^2/2}dy=v_E(z)=v_{E^s}(z)=\int_{-\infty}^{h(z)}e^{-|\sqrt{A}x|^2/2}dy.
\end{align}
Setting $z=0$, we have
\begin{align*}
\int^0_{-\infty}e^{-|\sqrt{A}(0,y)|^2/2}dy=\int_0^{\infty}e^{-|\sqrt{A}(0,y)|^2/2}dy=\int_{-\infty}^{h(0)}e^{-|\sqrt{A}(0,y)|^2/2}dy.
\end{align*}
Therefore, 
\begin{align}\label{regularity_lemma_eq2}
h(0)=0.
\end{align}
Taking the derivative on both sides with respect to $z$ of equation (\ref{regularity_lemma_eq1}), by Lemma \ref{computational_lemma} (1),
$$-\int_0^{\infty}e^{-|\sqrt{A}x|^2/2}A'x\ dy=(\nabla' h(z))e^{-|\sqrt{A}(z,h(z))|^2/2}-\int_{-\infty}^{h(z)}e^{-|\sqrt{A}x|^2/2}A'x\ dy.$$
Setting $z=0$ again,
\begin{align}\label{regularity_lemma_eq3}
\nabla' h(0)&=-\int_0^{\infty}e^{-|\sqrt{A}(0,y)|^2/2}A'(0,y)\ dy+\int_{-\infty}^{0}e^{-|\sqrt{A}(0,y)|^2/2}A'(0,y)\ dy\notag\\
&=-2\int_0^{\infty}e^{-|\sqrt{A}(0,y)|^2/2}A'(0,y)\ dy=-2\left(\int_0^{\infty}ye^{-|\sqrt{A}(0,y)|^2/2}\ dy\right) A'e_{n}.
\end{align}
Thus,
\begin{align*}
\nabla' h(0)=0&\iff A'e_n=0\iff A_{1n}=A_{2n}=\ldots =A_{n-1,n}=0\iff e_n\in V_{A_{nn}}(A)
\end{align*}
since $A$ is symmetric and 
$$\int_0^{\infty}ye^{-|\sqrt{A}(0,y)|^2/2}\ dy\geq \int_0^{\infty}ye^{-\|\sqrt{A}\|^2|y|^2/2}dy=\frac{1}{\|\sqrt{A}\|^2}>0.$$
Applying \cite{Dmitriy}, Theorem 1.14 with (\ref{regularity_lemma_eq2}) and (\ref{regularity_lemma_eq3}), we have
$$h(z)=\ell(z)+\int_0^1\langle\nabla' h(tz)-\nabla' h(0),z \rangle\ dt$$
where
$$\ell(z):=h(0)+\langle \nabla'h(0),z\rangle=-2\left(\int_0^{\infty}ye^{-|\sqrt{A}(0,y)|^2/2}\ dy\right) A'e_{n}\cdot z.$$
\end{proof}

\subsection{Proof of Theorem \ref{Uniqueness_Ehrhard_Sym}} For the first part, we just need to show that 
\begin{align}\label{Uniqueness_Ehrhard_Sym_Part1}
&P_{\gamma_A}(E_{A,u}^{s}) \leq P_{\gamma_A}\left(E \right) \mbox{ for all finite $A$-anisotropic Gaussian perimeter set $E$ in $\mathbb{R}^{n}$}\notag\\
&\implies u\in V_\lambda(A)\cap \SS^{n-1}\mbox{ for some $\lambda >0$}
\end{align} 
since Corollary \ref{Eigen_Ehrhard_Sym} gives us the converse of the statement.\\

\noindent{\bf Step 1}: Assume that $u=-e_n$ and we have $P_{\gamma_A}\left(E^{s}  \right) \leq P_{\gamma_A}\left(E \right)$ for all finite $A$-anisotropic Gaussian perimeter set $E$ in $\mathbb{R}^{n}$, where $E^s=E^s_{A,-e_n}$.
Our goal is to show that 
$$e_n\in V_\lambda(A)\cap \SS^{n-1}$$
for some $\lambda>0$. Let $K=[-1,1]^{n-1}$ and $\Omega$ be an open convex set that contains $K$. Consider
$$E=\Omega\times (0,\infty),\quad E_{\alpha}=[-\alpha,\alpha]^{n-1}\times (0,\infty)$$
for $\alpha\in (0,1)$. By Lemma \ref{regularity_lemma}, the Ehrhard symmetrization of $E$ has the form
$$E^s=\{x=(z,y)\in \R^{n-1}\times \R:z\in \Omega,y<h(z)\}$$
where $h(z):=\phi_z^{-1}(v_E(z))$ is $C^1(\Omega)$ and $\nabla'h$ is locally Lipschitz on $\Omega$.  Hence $\nabla'h$ is Lipschitz on $K$. Also, the Ehrhard symmetrization of $E_\alpha$ has the form
$$E^s_{\alpha}=\{x=(z,y)\in \R^{n-1}\times \R:z\in [-\alpha,\alpha]^{n-1},y<h(z)\}.$$
We claim that
$$P_{\gamma_A}(E_\alpha)-P_{\gamma_A}(E^s_\alpha)=\frac{\sqrt{\det A}}{(2\pi)^{(n-1)/2}}\left(1-\sqrt{1+[\nabla'h(0)]^2}\right)\alpha+o(\alpha).$$
Let $S_k=C_k\times (0,\infty)$ ($k\geq 1$) be hypersurfaces in $\R^n$, where $\{C_k\}_{k=1}^{2(n-1)}$ are faces of the $(n-1)$ dimension cube $[-\alpha,\alpha]^{n-1}\subset \R^{n-1}$, and  
$$S_0=[-\alpha,\alpha]^{n-1}\times \{0\}.$$
For example, 
\begin{align*}
S_1&=\left([-\alpha,\alpha]^{n-2}\times\{-\alpha\}\right)\times(0,\infty), \\
S_2&=\left([-\alpha,\alpha]^{n-2}\times\{\alpha\}\right)\times(0,\infty), \\
S_3&=\left([-\alpha,\alpha]^{n-3}\times\{-\alpha\}\times[-\alpha,\alpha]\right)\times(0,\infty), \\
S_4&=\left([-\alpha,\alpha]^{n-3}\times\{\alpha\}\times[-\alpha,\alpha]\right)\times(0,\infty),\\ 
S_5&=\left([-\alpha,\alpha]^{n-4}\times\{-\alpha\}\times[-\alpha,\alpha]^2\right)\times(0,\infty),\\
S_6&=\left([-\alpha,\alpha]^{n-4}\times\{\alpha\}\times[-\alpha,\alpha]^2\right)\times(0,\infty),\\
&\hspace{.19cm}\vdots
\end{align*}
and
$$\rb E_\alpha=\bigcup_{k=1}^{2(n-1)}S_k\cup S_0\implies P_{\gamma_A}(E_\alpha)=\sum_{k=1}^{2(n-1)}\mathcal{H}_{\gamma_A}^{n-1}(S_k)+\mathcal{H}^{n-1}_{\gamma_A}(S_0).$$
For $E^s_\alpha$, we also have
\begin{align*}
S^s_0&=\{x=(z,y)\in [-\alpha,\alpha]^{n-1}\times \R:y=h(z)\}, \\
S^s_1&=\left([-\alpha,\alpha]^{n-2}\times\{-\alpha\}\right)\times\left(-\infty,h\left([-\alpha,\alpha]^{n-2}\times\{-\alpha\}\right)\right), \\
S^s_2&=\left([-\alpha,\alpha]^{n-2}\times\{\alpha\}\right)\times \left(-\infty,h\left([-\alpha,\alpha]^{n-2}\times\{\alpha\}\right)\right), \\
S^s_3&=\left([-\alpha,\alpha]^{n-3}\times\{-\alpha\}\times[-\alpha,\alpha]\right)\times \left(-\infty,h\left([-\alpha,\alpha]^{n-3}\times\{-\alpha\}\times[-\alpha,\alpha]\right)\right), \\
S^s_4&=\left([-\alpha,\alpha]^{n-3}\times\{\alpha\}\times[-\alpha,\alpha]\right)\times \left(-\infty,h\left([-\alpha,\alpha]^{n-3}\times\{\alpha\}\times[-\alpha,\alpha]\right)\right),\\ 
S^s_5&=\left([-\alpha,\alpha]^{n-4}\times\{-\alpha\}\times[-\alpha,\alpha]^2\right)\times \left(-\infty,h\left([-\alpha,\alpha]^{n-4}\times\{-\alpha\}\times[-\alpha,\alpha]^2\right)\right),\\
S^s_6&=\left([-\alpha,\alpha]^{n-4}\times\{\alpha\}\times[-\alpha,\alpha]^2\right)\times \left(-\infty,h\left([-\alpha,\alpha]^{n-4}\times\{\alpha\}\times[-\alpha,\alpha]^2\right)\right),\\
&\hspace{.19cm}\vdots
\end{align*}
and
$$\rb E^s_\alpha=\bigcup_{k=1}^{2(n-1)}S^s_k\cup S^s_0\implies P_{\gamma_A}(E^s_\alpha)=\sum_{k=1}^{2(n-1)}\mathcal{H}_{\gamma_A}^{n-1}(S_k^s)+\mathcal{H}^{n-1}_{\gamma_A}(S_0^s).$$
Therefore,
\begin{align}\label{square_sides}
P_{\gamma_A}(E_\alpha)-P_{\gamma_A}(E^s_\alpha)&=\sum_{k=1}^{2(n-1)}\left(\mathcal{H}_{\gamma_A}^{n-1}(S_k)-\mathcal{H}_{\gamma_A}^{n-1}(S_k^s)\right)+\left(\mathcal{H}^{n-1}_{\gamma_A}(S_0)-\mathcal{H}_{\gamma_A}^{n-1}(S^s_0)\right).
\end{align}

\noindent (a) First we claim that
$$\left|\mathcal{H}_{\gamma_A}^{n-1}(S_{k})-\mathcal{H}_{\gamma_A}^{n-1}(S_{k}^s)\right|=o(\alpha^{n-1})\mbox{\quad for all $k\geq 1$}.$$
For $k=1$, we have $S_1=[-\alpha,\alpha]^{n-2}\times\{-\alpha\}\times(0,\infty)$
and
$$S^s_1=\left([-\alpha,\alpha]^{n-2}\times\{-\alpha\}\right)\times\left(-\infty,h\left([-\alpha,\alpha]^{n-2}\times\{-\alpha\}\right)\right).$$
Let $r(u,v)=(u,-\alpha,v)\in [-\alpha,\alpha]^{n-2}\times\{-\alpha\}\times(0,\infty)$.
Then $J(r)=\sqrt{\det (Dr)^{\mathsf{T}}(Dr)}=1$ and
\begin{align*}
\mathcal{H}^{n-1}_{\gamma_A}(S_1)&=\frac{\sqrt{\det A}}{(2\pi)^{(n-1)/2}}\int_{S_1}e^{-|\sqrt{A}x|^2/2}d\mathcal{H}^{n-1}=\frac{\sqrt{\det A}}{(2\pi)^{(n-1)/2}}\int_{[-\alpha,\alpha]^{n-2}}\int_{0}^\infty e^{-|\sqrt{A}( u,-\alpha,v)|^2/2}\ dvdu.
\end{align*}
Similarly, we have
\begin{align*}
\mathcal{H}^{n-1}_{\gamma_A}(S^s_1)&=\frac{\sqrt{\det A}}{(2\pi)^{(n-1)/2}}\int_{S_1^s}e^{-|\sqrt{A}x|^2/2}d\mathcal{H}^{n-1}=\frac{\sqrt{\det A}}{(2\pi)^{(n-1)/2}}\int_{[-\alpha,\alpha]^{n-2}}\int_{-\infty}^{h(u,-\alpha)} e^{-|\sqrt{A}( u,-\alpha,v)|^2/2}\ dvdu.
\end{align*}
Let $z=(u,-\alpha)$. We now estimate the following two quantities:
$$\text{(i)}\int_{0}^\infty e^{-|\sqrt{A}(z,v)|^2/2}dv,\qquad \text{(ii)}\int_{-\infty}^{h(z)} e^{-|\sqrt{A}(z,v)|^2/2}dv.$$
For (i), using the Taylor expansion on the map 
$$z\mapsto \int_{0}^\infty e^{-|\sqrt{A}(z,v)|^2/2}dv,$$ 
we have 
$$\nabla'  \int_{0}^\infty e^{-|\sqrt{A}(z,v)|^2/2}dv= \int_{0}^\infty e^{-|\sqrt{A}(z,v)|^2/2}\left(-A'(z,v)\right)dv$$
and
\begin{align*}
\int_{0}^\infty e^{-|\sqrt{A}(z,v)|^2/2}dv&=\int_{0}^\infty e^{-|\sqrt{A}(0,v)|^2/2}dv+\int_{0}^\infty e^{-|\sqrt{A}(0,v)|^2/2}(-A'(0,v))\ dv\cdot z+o(|z|).
\end{align*}
For (ii), using the Taylor expansion on the map 
$$z\mapsto \int_{-\infty}^{h(z)} e^{-|\sqrt{A}(z,v)|^2/2}dv,$$ 
we have 
$$\nabla'  \int_{-\infty}^{h(z)} e^{-|\sqrt{A}(z,v)|^2/2}dv= \nabla'h(z)e^{-|\sqrt{A}(z,h(z))|^2/2}+\int_{-\infty}^{h(z)} e^{-|\sqrt{A}(z,v)|^2/2}\left(-A'(z,v)\right)dv$$
and
\begin{align*}
 \int_{-\infty}^{h(z)} e^{-|\sqrt{A}(z,v)|^2/2}dv&=\int_{-\infty}^0 e^{-|\sqrt{A}(0,v)|^2/2}dv+\left(\nabla'h(0)+\int_{-\infty}^0 e^{-|\sqrt{A}(0,v)|^2/2}(-A'(0,v))\ dv\right)\cdot z+o(|z|)\\
&=\int_{0}^{\infty} e^{-|\sqrt{A}(0,v)|^2/2}dv+\left(\nabla'h(0)+\int_{0}^{\infty} e^{-|\sqrt{A}(0,v)|^2/2}A'(0,v)\ dv\right)\cdot z+o(|z|).
\end{align*}
since $h(0)=0$. Therefore,
\begin{align}\label{two_Taylor_estimates}
\text{(i)}-\text{(ii)}&=-\nabla'h(0)\cdot z+\left(-2\int_{0}^{\infty} e^{-|\sqrt{A}(0,v)|^2/2}A'(0,v)\ dv\right)\cdot z+o(|z|)\notag\\
&=-\nabla'h(0)\cdot z+\left(-2\int_{0}^{\infty} ve^{-|\sqrt{A}(0,v)|^2/2}\ dv\right)A'e_n\cdot z+o(|z|)\notag\\
&=o(|z|)
\end{align}
where we have used Lemma \ref{regularity_lemma}, i.e.,
$$\nabla' h(0)=-2\left(\int_0^{\infty}ye^{-|\sqrt{A}(0,y)|^2/2}\ dy\right) A'e_{n}.$$
Since $u\in [-\alpha,\alpha]^{n-2}$, $|z|=|(u,-\alpha)|\leq \sqrt{n-1} \alpha$, by equation (\ref{two_Taylor_estimates}),
\begin{align*}
\int_{0}^\infty e^{-|\sqrt{A}( u,-\alpha,v)|^2/2}dv-\int_{-\infty}^{h(u,-\alpha)} e^{-|\sqrt{A}( u,-\alpha,v)|^2/2}dv=o(\alpha)
\end{align*}
and hence
\begin{align*}
\mathcal{H}^{n-1}_{\gamma_A}(S_1)-\mathcal{H}^{n-1}_{\gamma_A}(S^s_1)&=\frac{\sqrt{\det A}}{(2\pi)^{(n-1)/2}}\int_{[-\alpha,\alpha]^{n-2}}\left(\int_{0}^\infty e^{-|\sqrt{A}( u,-\alpha,v)|^2/2}dv-\int_{-\infty}^{h(u,-\alpha)} e^{-|\sqrt{A}( u,-\alpha,v)|^2/2}dv\right)du\\
&=\frac{\sqrt{\det A}}{(2\pi)^{(n-1)/2}}\int_{[-\alpha,\alpha]^{n-2}}o(\alpha)\ du =o(\alpha^{n-1}).
\end{align*}
That is,
$$\left|\mathcal{H}_{\gamma_A}^{n-1}(S_{1})-\mathcal{H}_{\gamma_A}^{n-1}(S_{1}^s)\right|=o(\alpha^{n-1}).$$
Similarly for all $k\geq 1$, we have
$$\left|\mathcal{H}_{\gamma_A}^{n-1}(S_{k})-\mathcal{H}_{\gamma_A}^{n-1}(S_{k}^s)\right|=o(\alpha^{n-1}).$$
\noindent (b) Next we claim that
\begin{align*}
P_{\gamma_A}(E_\alpha)-P_{\gamma_A}(E^s_\alpha)&=\frac{\sqrt{\det A}}{(2\pi)^{(n-1)/2}}\left(\int_{S_0}e^{-|\sqrt{A}x|^2/2}d\mathcal{H}^{n-1}-\int_{S^s_0}e^{-|\sqrt{A}x|^2/2}d\mathcal{H}^{n-1}\right)+o(\alpha^{n-1})\\
&=\frac{\sqrt{\det A}}{(2\pi)^{(n-1)/2}}\left(1-\sqrt{1+[\nabla'h(0)]^2}\right)(2\alpha)^{n-1}+o(\alpha^{n-1}).
\end{align*}
Let $r(z)=(z,h(z))$. Then $J(r)=\sqrt{\det (Dr)^{\mathsf{T}}(Dr)}=\sqrt{1+|\nabla' h(z)|^2}$ and
\begin{align*}
\mathcal{H}^{n-1}_{\gamma_A}(S_0^s)&=\frac{\sqrt{\det A}}{(2\pi)^{(n-1)/2}}\int_{S_0^s}e^{-|\sqrt{A}x|^2/2}d\mathcal{H}^{n-1}=\frac{\sqrt{\det A}}{(2\pi)^{(n-1)/2}}\int_{[-\alpha,\alpha]^{n-1}} e^{-|\sqrt{A}(z,h(z))|^2/2}\sqrt{1+|\nabla' h(z)|^2}\  dz.
\end{align*}
Define $f:\Omega\to \R$ as
$$f(z)=e^{-|\sqrt{A}(z,0)|^2/2}-e^{-|\sqrt{A}(z,h(z))|^2/2}\sqrt{1+[\nabla'h(z)]^2}.$$
Since $h\in C^{1}(\Omega)$, $f$ is continuous on $\Omega$ and hence
$$\lim_{\alpha\to 0^+}\frac{1}{(2\alpha)^{n-1}}\int_{[-\alpha,\alpha]^{n-1}}f(z) dz= f(0).$$
Therefore,
\begin{align}\label{S_0_estimate}
&\int_{S_0}e^{-|\sqrt{A}x|^2/2}d\mathcal{H}^{n-1}-\int_{S^s_0}e^{-|\sqrt{A}x|^2/2}d\mathcal{H}^{n-1}\notag\\
&=\int_{[-\alpha,\alpha]^{n-1}}e^{-|\sqrt{A}(z,0)|^2/2}dz-\int_{[-\alpha,\alpha]^{n-1}}e^{-|\sqrt{A}(z,h(z))|^2/2}\sqrt{1+[\nabla'h(z)]^2}\ dz\notag\\
&=\int_{[-\alpha,\alpha]^{n-1}}\left(e^{-|\sqrt{A}(z,0)|^2/2}-e^{-|\sqrt{A}(z,h(z))|^2/2}\sqrt{1+[\nabla'h(z)]^2}\right)\ dz\notag\\
&=\int_{[-\alpha,\alpha]^{n-1}}f(0)+\left(f(z)-f(0)\right)\ dz\notag\\
&=\int_{[-\alpha,\alpha]^{n-1}}\left(1-\sqrt{1+[\nabla'h(0)]^2}\right) dz+\int_{[-\alpha,\alpha]^{n-1}}\left(f(z)-f(0)\right) dz\notag\\
&=\left(1-\sqrt{1+[\nabla'h(0)]^2}\right)(2\alpha)^{n-1}+o(\alpha^{n-1}).
\end{align}
By equation (\ref{square_sides}) and (a),
\begin{align}\label{all_faces_estimates}
&\left|P_{\gamma_A}(E_\alpha)-P_{\gamma_A}(E^s_\alpha)-\frac{\sqrt{\det A}}{(2\pi)^{(n-1)/2}}\left(\int_{S_0}e^{-|\sqrt{A}x|^2/2}d\mathcal{H}^{n-1}-\int_{S^s_0}e^{-|\sqrt{A}x|^2/2}d\mathcal{H}^{n-1}\right)\right|\notag\\
&\leq \sum_{k=1}^{2(n-1)}\left|\mathcal{H}_{\gamma_A}^{n-1}(S_{k})-\mathcal{H}_{\gamma_A}^{n-1}(S_{k}^s)\right|= o(\alpha^{n-1}).
\end{align}
Combining (\ref{S_0_estimate}) and (\ref{all_faces_estimates}), we have 
\begin{align*}
P_{\gamma_A}(E_\alpha)-P_{\gamma_A}(E^s_\alpha)&=\frac{\sqrt{\det A}}{(2\pi)^{(n-1)/2}}\left(1-\sqrt{1+[\nabla'h(0)]^2}\right)(2\alpha)^{n-1}+o(\alpha^{n-1}).
\end{align*}
(c) We claim that
$$A_{1n}=A_{2n}=\ldots =A_{n-1,n}=0.$$
By our assumption and (b),
\begin{align*}
0\leq P_{\gamma_A}(E_\alpha)-P_{\gamma_A}(E^s_\alpha)&=\frac{\sqrt{\det A}}{(2\pi)^{(n-1)/2}}\left(1-\sqrt{1+[\nabla' h(0)]^2}\right)(2\alpha)^{n-1}+o(\alpha^{n-1}).
\end{align*}
Dividing $(2\alpha)^{n-1}$ on both sides and taking $\alpha\to 0^+$, by Lemma \ref{regularity_lemma},
\begin{align*}
0&\leq \left(1-\sqrt{1+[\nabla' h(0)]^2}\right)\implies \nabla' h(0)=0\implies A_{1n}=A_{2n}=\ldots =A_{n-1,n}=0.
\end{align*}
Hence, $e_n\in V_{\lambda}(A)\cap \SS^{n-1}$ for some $\lambda>0$.\\

\noindent{\bf Step 2}: For general $u\in \SS^{n-1}$, there exists an orthogonal matrix $O$ such that $O(-e_n)=u$. Let $B=O^{\mathsf{T}}AO$. Given any finite $B$-anisotropic Gaussian perimeter set $\widetilde{E}$ and let $E=O\widetilde{E}$. By Proposition \ref{locally finite anisotropic Gaussian perimeter}, $\widetilde{E}$ is a set of locally finite perimeter. Applying Proposition \ref{Gaussian_perimeter} (3) with $E$ as $\widetilde{E},$ $A$ as $B$, and $O$ as $O^{-1}$, 
$$P_{\gamma_A}(E)=P_{\gamma_{OBO^{\mathsf{T}}}}(O\widetilde{E})=P_{\gamma_B}(\widetilde{E})<\infty,$$
i.e., $E$ is a set of finite $A$-anisotropic Gaussian perimeter. Then Theorem \ref{Ehrhard_Sym_Ineq_II} tells us that $E_{A,u}^s$ is also a set of locally finite perimeter. Since $\gamma_A$ is Ehrhard symmetrizable and $E$ is a set of finite $A$-anisotropic Gaussian perimeter, Proposition \ref{Gaussian_perimeter} (3) and equation (\ref{rotation_of_sets}) give us
\begin{align*}
P_{\gamma_B}(\widetilde{E}^s_{B,-e_n})=P_{\gamma_{O^{\mathsf{T}}AO}}\left((O^{-1}E)^s_{O^{\mathsf{T}}AO,-e_n}\right)&=P_{\gamma_{O^{\mathsf{T}}AO}}(O^{-1}E_{A,u}^s)\\
&=P_{\gamma_A}(E_{A,u}^s)\leq P_{\gamma_A}(E)=P_{\gamma_{O^{\mathsf{T}}AO}}(O^{-1}E)=P_{\gamma_B}(\widetilde{E}).
\end{align*}
Applying Step 1 on $\gamma_B$ and $\widetilde{E}$, we conclude that $e_n\in V_{\lambda}(B)\cap \SS^{n-1}$ for some eigenvalue $\lambda$, i.e., $Be_n=\lambda e_n$ and hence
$$Au=AO(-e_n)=OB(-e_n)=-O\lambda e_n=\lambda O(-e_n)=\lambda u.$$
Thus, if $P_{\gamma_A}(E_{A,u}^{s}) \leq P_{\gamma_A}\left(E \right)$ for all finite $A$-anisotropic Gaussian perimeter set $E$ in $\mathbb{R}^{n}$, we have
\begin{equation}\label{eigenvalue_of_A}
u\in V_{\lambda}(A)\mbox{\quad for some $\lambda> 0$.}
\end{equation}
This finishes the first part of the theorem.\\

For the second part, it is enough to prove that
$$\mbox{$\gamma_A$ is Ehrhard symmetrizable}\implies A=aI_n\mbox{ for some constant $a>0$}$$
since we can apply Corollary \ref{Eigen_Ehrhard_Sym} again, and conclude the converse of the statement. Suppose now we have two distinct eigenvalues $\lambda_1,\lambda_2$ of $A$ with eigenvectors $u_1,u_2$ in $\SS^{n-1}$. Notice that $\langle u_1,u_2\rangle =0$ since $A$ is symmetric and $\lambda_1\not=\lambda_2$. Consider
$$u:=\frac{u_1+u_2}{\sqrt{2}}\in \SS^{n-1}.$$
Since $\gamma_A$ is Ehrhard symmetrizable, by (\ref{Uniqueness_Ehrhard_Sym_Part1}), we have $u\in V_\lambda(A)\cap \SS^{n-1}$ for some $\lambda>0$. However,
$$\lambda\left(\frac{u_1+u_2}{\sqrt{2}}\right)=\lambda u=Au=\frac{\lambda_1u_1+\lambda_2u_2}{\sqrt{2}}\implies (\lambda_1-\lambda)u_1+(\lambda_2-\lambda)u_2=0\implies \lambda_1=\lambda=\lambda_2.$$
This contradicts the assumption that $\lambda_1\not=\lambda_2$. Therefore, all the eigenvalues of $A$ are the same and hence $A=aI_n$ for some $a>0$. \qed

\bibliographystyle{amsalpha}
\bibliography{References}

\providecommand{\bysame}{\leavevmode\hbox to3em{\hrulefill}\thinspace}
\providecommand{\MR}{\relax\ifhmode\unskip\space\fi MR }
\providecommand{\MRhref}[2]{%
  \href{http://www.ams.org/mathscinet-getitem?mr=#1}{#2}
}
\providecommand{\href}[2]{#2}
\begin{thebibliography}{CFMP11}

\bibitem[AD00]{Ambrosio-Dancer}
L.~Ambrosio and N.~Dancer, \emph{Calculus of variations and partial differential equations, topics on geometrical evolution problems and degree theory}, Springer-Verlag Berlin Heidelberg, 2000.

\bibitem[ADM90]{Ambrosio_DalMaso}
L.~Ambrosio and G.~Dal~Maso, \emph{A general chain rule for distributional derivatives}, Proceedings of the American Mathematical Society \textbf{108} (1990), no.~3, 691--702.

\bibitem[AFP00]{Ambrosio_Fusco_Pallara}
L.~Ambrosio, N.~Fusco, and D.~Pallara, \emph{Functions of bounded variation and free discontinuity problems}, Oxford Mathematical Monographs, 2000.

\bibitem[Bob96]{Bobkov}
S.~Bobkov, \emph{Extremal properties of half-spaces for log-concave distributions}, The Annals of Probability \textbf{24} (1996), no.~1, 35 -- 48.

\bibitem[Bor75]{Borell}
C.~Borell, \emph{The {Brunn-Minkowski} inequality in {Gauss} space}, Inventiones mathematicae \textbf{30} (1975), no.~1, 207–216.

\bibitem[Bor03]{Borell_EhrhardIneq}
\bysame, \emph{The {Ehrhard} inequality}, Comptes Rendus Mathematique \textbf{337} (2003), no.~10, 663--666.

\bibitem[CCF05]{Chlebik-Cianchi-Fusco}
M.~Chlebík, A.~Cianchi, and N.~Fusco, \emph{The perimeter inequality under {Steiner} symmetrization: Cases of equality}, Annals of Mathematics \textbf{162} (2005), no.~1, 525--555.

\bibitem[CFMP11]{Cianchi-Fusco-Maggi-Pratelli}
A.~Cianchi, N.~Fusco, F.~Maggi, and A.~Pratelli, \emph{On the isoperimetric deficit in the {Gauss} space}, American Journal of Mathematics \textbf{133} (2011), no.~1, 131--186.

\bibitem[CK01]{Carlen-Kerce}
E.~Carlen and C.~Kerce, \emph{On the case of equality in {Bobkov's} inequality and {Gaussian} rearrangement}, Calculus of Variations and Partial Differential Equations \textbf{13} (2001), 1–18.

\bibitem[Dan08]{Daners}
D.~Daners, \emph{Chapter 1 - domain perturbation for linear and semi-linear boundary value problems}, Handbook of Differential Equations, North-Holland, 2008.

\bibitem[Dru20]{Dmitriy}
D.~Drusvyatskiy, \emph{Convex analysis and nonsmooth optimization}, \url{https://sites.math.washington.edu/~ddrusv/crs/Math_516_2020/bookwithindex.pdf}, October 2020.

\bibitem[EG92]{evansgariepy}
L.C. Evans and R.F. Gariepy, \emph{Measure theory and fine properties of functions}, CRC, 1992.

\bibitem[Ehr83]{Ehrhard1983}
A.~Ehrhard, \emph{Symétrisation dans l'espace de gauss.}, Mathematica Scandinavica \textbf{53} (1983), 281--301.

\bibitem[Fed69]{federergmt}
H.~Federer, \emph{Geometric measure theory}, vol. 1996, Springer New York, 1969.

\bibitem[Fus04]{Fusco_Iso}
N.~Fusco, \emph{The classical isoperimetric theorem}, Rendiconto Dell'accademia Delle Scienze Fisiche E Matematiche (2004), 63--107.

\bibitem[HJ12]{Horn-Johnson}
R.~Horn and C.~Johnson, \emph{Matrix analysis 2nd edition}, Cambridge University Press, 2012.

\bibitem[Lat03]{Latala_Ehrhard}
R.~Lata{\l}a, \emph{On some inequalities for {Gaussian} measures}, arXiv Mathematics e-prints (2003), math/0304343.

\bibitem[LB96]{Ledoux-Bakry}
M.~Ledoux and D.~Bakry, \emph{L{\'e}vy-gromov's isometric inequality for an infinite dimensional diffusion generator.}, Inventiones mathematicae \textbf{123} (1996), no.~2, 259--282.

\bibitem[Led96]{LedouxGaussianIso}
M.~Ledoux, \emph{Isoperimetry and {Gaussian} analysis}, Lectures on Probability Theory and Statistics, vol. 1648, Springer, Berlin, Heidelberg, 1996.

\bibitem[Led99]{Ledoux_LogSobolev}
\bysame, \emph{Concentration of measure and logarithmic sobolev inequalities}, S{\'e}minaire de Probabilit{\'e}s XXXIII (Berlin, Heidelberg) (Jacques Az{\'e}ma, Michel {\'E}mery, Michel Ledoux, and Marc Yor, eds.), Springer Berlin Heidelberg, 1999, pp.~120--216.

\bibitem[Liv21]{Livshyts}
G.~Livshyts, \emph{On a conjectural symmetric version of {Ehrhard's} inequality}, arXiv Mathematics e-prints (2021), https://arxiv.org/abs/2103.11433.

\bibitem[Mag12]{maggi}
F.~Maggi, \emph{{Sets of Finite Perimeter and Geometric Variational Problems: An Introduction to Geometric Measure Theory}}, Cambridge University Press, 2012.

\bibitem[Nay17]{Nayar}
P.~Nayar, \emph{Inequalities in convex geometry}, \url{https://www.mimuw.edu.pl/~nayar/iicg_2018/iicg.pdf}, 2017.

\bibitem[ST78]{Sudakov-Tsirelson}
V.~Sudakov and B.~Tsirelson, \emph{Extremal properties of half-spaces for spherically invariant measures}, Journal of Soviet Mathematics \textbf{9} (1978), 9--18.

\bibitem[Vol67]{Vol'pert}
A.I. Vol'pert, \emph{Spaces {BV} and quasilinear equations}, Matematicheskii Sbornik \textbf{73(115)} (1967), no.~2, 255--302.

\end{thebibliography}

\end{document}